%% file: kromatic-note-draft3.tex
\documentclass[10pt]{article}

\usepackage[backend=bibtex]{biblatex}
\input{preamble.tex}

\begin{document}
\title{Kromatic quasisymmetric functions}
\author{
    Eric Marberg\\
    Department of Mathematics \\
    Hong Kong University of Science and Technology \\
    {\tt eric.marberg@gmail.com}
}

\date{}

\maketitle

\begin{abstract}
We provide a construction for the kromatic symmetric function $\bar X_G$ of a graph introduced by Crew, Pechenik, and Spirkl using combinatorial (linearly compact) Hopf algebras. As an application, we show that $\bar X_G$ has a positive expansion into multifundamental quasisymmetric functions. We also study two related quasisymmetric $q$-analogues of $\bar X_G$, which are $K$-theoretic generalizations of the quasisymmetric chromatic function of Shareshian and Wachs. We classify exactly when one of these analogues is symmetric. For the other, we derive a positive expansion into symmetric Grothendieck functions when $G$ 
is the incomparability graph of a natural unit interval order.
\end{abstract}

%\tableofcontents
\section{Introduction}

The purpose of this note is to re-examine the algebraic
origins of the \defn{kromatic symmetric function} of a graph that was
recently introduced by Crew, Pechenik, and Spirkl \cite{CPS}, and to study two
quasisymmetric analogues of this power series. 

Let $\NN = \{0,1,2,\dots\}$,
$\PP = \{1,2,3,\dots\}$,
and $[n] = \{ i \in \PP : i \leq n\}$ for $n \in \NN$, so that $[0]=\varnothing$.
All graphs are undirected by default, and
are assumed to be simple with a finite set of vertices.
We do not distinguish between isomorphic graphs. 

If $G$ is any graph then 
we write $V(G)$ for its set of vertices and $E(G)$ for its set of edges.
A \defn{proper coloring} of $G$ is a map $\kappa :  V(G) \to \PP$ with $\kappa(u) \neq \kappa(v)$ for all   $\{u,v\} \in E(G)$.
For maps $\kappa : V \to \PP$ let $x^\kappa = \prod_{i \in V } x_{\kappa(i)}$ where $x_1,x_2,\dots$ are commuting variables.

\begin{definition}[Stanley \cite{Stanley95}]
\label{stan-def}
The \defn{chromatic symmetric function} of $G$ is 
the power series $ X_G := \sum_\kappa x^\kappa $
where the sum is over all proper colorings $\kappa$ of $G$.
\end{definition}
 
%\begin{example} If $G$ has $n$ vertices and no edges then $X_G= e_{1^n}$.
%\end{example}

\begin{example} If $G=K_n$ is the \defn{complete graph} with $V(G)=[n]$ then $X_G$ is $n!$ times 
 the \defn{elementary symmetric function} $e_n := \sum_{1 \leq i_1 <i_2< \dots <i_n} x_{i_1} x_{i_2}\cdots x_{i_n}$.
%\[X_G = \sum_{\substack{i_1,i_2,\dots,i_n \in \PP\\ |\{i_1,i_2,\dots,i_n\}|=n}} x_{i_1}x_{i_2}\cdots x_{i_n}= n! \cdot e_n = n! \cdot m_{1^n}.\]
%where $e_\lambda$ and $m_\lambda$ are the usual elementary and monomial symmetric functions.
\end{example}

%It is apparent that $X_G \in \Sym $ is a symmetric function. 
A poset is \defn{$(3+1)$-free} if it does not contain a 3-element chain $a<b<c$ whose elements are all 
incomparable to some fourth element $d$.
The \defn{incomparability graph} of a poset is the graph whose vertices are the set of elements in the poset and 
whose edges are the unordered pairs $\{x,y\}$ with $x \not \leq y$ and $y \not \leq x$.
The \defn{Stanley--Stembridge conjecture} \cite{stanley.stembridge}
asserts that if $G$ is the incomparability graph of a $(3+1)$-free poset then $X_G$ has a positive expansion into elementary symmetric functions.
This conjecture has several refinements and generalizations \cite{guay,SW2016},
and has been resolved in a number of interesting special cases \cite{dahl,hamel2019,HaradaPrecup}. 
Hikita \cite{Hikita} has recently announced a general proof of the conjecture. 

%\begin{theorem} If $G$ is the incomparability graph of a $(3+1)$-free poset then $X_G $ is Schur positive.
%\end{theorem}
%
%\begin{theorem} If the $e$-expansion of $X_G$ is $X_G = \sum_{\lambda} c_\lambda e_\lambda$ then the number of acyclic orientations of $G$ with exactly $j$ sources is $\sum_{\ell(\lambda) =j} c_\lambda\in \NN$.
%\end{theorem}
%
%The coefficients $c_\lambda$ can be negative integers. However:

%\begin{conjecture} If $G$ is the incomparability graph of a $(3+1)$-free poset then $X_G$ is $e$-positive.
%\end{conjecture}

Let $G$ be an \defn{ordered graph}, that is, a graph with a total order $<$ on its vertex set $V(G)$.
An \defn{ascent} (respectively, \defn{descent}) of a map $\kappa : V(G) \to \PP$ is an edge $\{u,v\} \in E(G)$ with $u<v$ 
and  $\kappa(u) <\kappa(v)$
(respectively, $\kappa(u)>\kappa(v)$). Let $\asc_{G}(\kappa)$
and
$\des_{G}(\kappa)$ be the number of ascents and descents of $\kappa$.
Shareshian and Wachs \cite{SW2016} introduced the following $q$-analogue of $X_G$:

\begin{definition}[\cite{SW2016}]\label{cqf-def}
The \defn{chromatic quasisymmetric function} of an ordered graph
 $G$ is   $X_G(q) = \sum_\kappa q^{\asc_G(\kappa)} x^\kappa 
 %\in  \NN[q]\llbracket x_1,x_2,\dots\rrbracket 
 $ where the sum is
 over all proper colorings $\kappa$ of $G$.
 \end{definition}

 \begin{example}\label{nq-ex}  One has $X_{K_n}(q)= [n]_q!   e_n$
 for $[i]_q := \tfrac{1-q^i}{1-q}$ and $[n]_q! := \prod_{i=1}^n [i]_q$.
 \end{example}

  Let $\Set(\PP)$ be the set of finite nonempty subsets of positive integers.
For a map
 $\kappa : V \to \Set(\PP)$
define $x^\kappa = \prod_{i \in V}  \prod_{j \in \kappa(i)} x_j$.
A \defn{proper set-valued coloring} is a map $\kappa : V(G) \to \Set(\PP)$
with $\kappa(u) \cap \kappa(v) = \varnothing$ for all $\{u,v\} \in E(G)$.
Using set-valued colorings in Definition~\ref{stan-def} leads to a ``$K$-theoretic'' analogue of $X_G$:

\begin{definition}[Crew, Pechenik, and Spirkl \cite{CPS}]
The \defn{kromatic symmetric function} of a graph
 $G$ is the power series $\bX_G = \sum_\kappa x^\kappa \in \ZZ\llbracket x_1,x_2,\dots\rrbracket $
 where the sum is
 over all proper set-valued colorings of $G$.
\end{definition}

%\begin{example} If $G$ has $n$ vertices and no edges then $\bX_G= \bare_{1^n}$.\end{example}

\begin{example}\label{kro-ex}
 If $G=K_n$ then $ \bX_G= n!\sum_{r= n}^\infty {r \brace n} e_r $
%\[\textstyle \bX_G= n!\sum_{r= n}^\infty {r \brace n} e_r =  n! \sum_{r= n}^\infty {r -1\brace n-1} \bare_r\]
where ${r \brace n}$ is the Stirling number of the second kind.  
\end{example}

\begin{remark}\label{clan-rmk}
Given $\alpha : V \to \PP$, let $ \cCl_\alpha(V)$ be the set of pairs $(v,i)$ with $v \in V$ and $i \in [\alpha(v)]$. 
If $G$ is a graph and  $\alpha : V(G) \to \PP$ is any map, then
the \defn{$\alpha$-clan graph}  $\Cl_\alpha(G)$ has vertex set $\cCl_\alpha(V(G))$
and edges
 $\{ (v,i),(w,j)\}$  whenever $\{v,w\} \in E(G)$ or 
 both $v=w$ and $i\neq j$.
As observed in \cite{CPS}, one has 
$
\bX_G = \sum_{\alpha : V(G) \to \PP} \tfrac{1}{\alpha!} X_{\Cl_\alpha(G)}
$
where $\alpha! := \prod_v \alpha(v)!$.
Many   properties of $X_G$ extend to $\bX_G$
via this identity, but some interesting features of $\bX_G$ cannot be explained by this formula alone.
 \end{remark}
 
Our main results in Section~\ref{main-sect}
 provide a natural construction for $\bX_G$  
 using the theory of \defn{combinatorial Hopf algebras}.
 This approach requires some care, as $\bX_G$ is not a symmetric function of bounded degree. We explain things precisely in terms of \defn{linearly compact Hopf algebras} after reviewing a similar, simpler construction of $X_G$ in Section~\ref{back-sect}, following \cite{ABS}.

As an application of our approach, we show that $\bX_G$ has a positive expansion into \defn{multifundamental quasisymmetric functions}. We also study two related $q$-analogues of $\bX_G$, which give $K$-theoretic generalizations of $X_G(q)$.
We classify exactly when one of these analogues is symmetric. 
For the other, we extend a theorem of Crew, Pechenik, and Spirkl
(also lifting a theorem of Shareshian and Wachs) to derive a positive expansion into symmetric Grothendieck functions 
when $G$ 
is the incomparability graph of a natural unit interval order.
These results are contained in Section~\ref{qana-sect1}.

\subsection*{Acknowledgments}

This work was partially supported by Hong Kong RGC grants 16306120 and 16304122. We thank Darij Grinberg, Joel Lewis, Brendan Pawlowski, Oliver Pechenik, and Travis Scrimshaw for many useful comments and discussions.

\section{Background}\label{back-sect}

Let $\KK$ be an integral domain; in most results, one can assume this is $\ZZ$.

\subsection{Hopf algebras}\label{hopf-sect1}

Write $\otimes = \otimes_\KK$ for the tensor product over $\KK$.
A \defn{$\KK$-algebra} is a $\KK$-module $A$ with $\KK$-linear
product $\nabla : A\otimes A \to A$
and unit $\iota : \KK\to A$ 
maps.
Dually, a \defn{$\KK$-coalgebra} is a $\KK$-module $A$ with $\KK$-linear 
coproduct $\Delta : A \to A\otimes A$
and 
counit $\epsilon : A \to \KK$ maps.
 The (co)product and (co)unit maps must satisfy several associativity axioms;
see \cite[\S1]{GrinbergReiner}.

A $\KK$-module $A$ that is both a $\KK$-algebra and a $\KK$-coalgebra is a \defn{$\KK$-bialgebra}
if the coproduct and counit maps are algebra morphisms. 
A bialgebra with a direct sum decomposition $A = \bigoplus_{n \in \NN}A_n$ is \defn{graded} if its (co)product and (co)unit are graded maps.
A bialgebra $A = \bigoplus_{n \in \NN} A_0$ is \defn{connected} if there is an isomorphism $A_0\cong \KK$ and the unit and counit 
are obtained by composing this isomorphism with the usual inclusion and projection maps
$ A_0 \hookrightarrow A$ and  $A \twoheadrightarrow A_0$.

%Suppose $A$ is a $\KK$-bialgebra. 
Let $\End(A)$ denote the set of $\KK$-linear maps $A \to A$. This set is a $\KK$-algebra for the 
product 
$
f \ast g := \nabla \circ (f\otimes g) \circ \Delta
$. 
The unit of this \defn{convolution algebra}
is the composition $ \iota\circ \epsilon$ of the unit and counit of $A$.
A bialgebra $A$ is a \defn{Hopf algebra} if the identity map $\id : A \to A$ has a
two-sided inverse 
$\antipode : A \to A$ in $\End(A)$. When it exists, we call $\antipode$ the \defn{antipode} of $A$.

We mention a few common examples of Hopf algebras.

\begin{example}
Let $\cG_n$ be the free $\KK$-module spanned by all isomorphism classes of undirected graphs with $n$ vertices, and set $\cG = \bigoplus_{n \in \NN} \cG_n$.
One views $\cG$ as 
a connected, graded Hopf algebra 
 with product 
$\nabla(G \otimes H) = G \sqcup H$ and coproduct  
$ \Delta(G) = \sum_{S\sqcup T = V(G)} G|_S \otimes G|_T$
for   graphs $G$ and $H$,
 where $\sqcup$ means disjoint union and  $G|_S$ is the induced subgraph of $G$ on $S$. 
\end{example}

A \defn{lower set} in a directed acyclic graph $D=(V,E)$ is a set $S\subseteq V$ 
such that if a directed path connects $v \in V$ to $s \in S$ then $v \in S$.
An \defn{upper set} in $D$  is a set $S\subseteq V$ 
such that if a directed path connects $s \in S$ to $v \in V$ then $v \in S$.
Finally, an \defn{antichain} in $D$ is a set of vertices $S$ such that 
no directed path connects any $s \in S$ to any other $t\in S$.

\begin{example}
Let $\cO_n$ for $n \in \NN$ be the free $\KK$-module spanned by all isomorphism classes of directed acyclic graphs with $n$ vertices, and set 
$\cO = \bigoplus_{n \in \NN} \cO_n$. 
One views $\cO$ as a connected, graded Hopf algebra  with product  
$\textstyle \nabla(C \otimes D) = C \sqcup D$ and 
coproduct $ \Delta(D) = \sum D|_S \otimes D|_T$
for directed acyclic graphs  graphs $C$ and $D$,
 where the sum is over all disjoint unions $S\sqcup T = V(D)$ 
 with $S$ a lower set and $T$ an upper set.
\end{example}

A \defn{labeled poset} is a pair $(D,\gamma)$ consisting of a directed acyclic graph $D$ and  an injective map $\gamma : V(D) \to \ZZ$.
We consider $(D,\gamma)=(D',\gamma')$   if 
there is an isomorphism $D \xrightarrow{\sim} D'$, written $v\mapsto v'$, such that
$\gamma(u) - \gamma(v)$ and $\gamma'(u') - \gamma'(v')$ have the same sign for all edges $u\to v\in E(D)$. 

If $(D_1,\gamma_1)$ and $(D_2,\gamma_2)$ are labeled posets then let $\gamma_1\sqcup \gamma_2 : V(D_1\sqcup D_2)\to \ZZ$  be
any injective map such that $(\gamma_1\sqcup \gamma_2)(u) - (\gamma_1\sqcup \gamma_2)(v)$ has the same sign as $\gamma_i(u) - \gamma_i(v)$ for all $i \in \{1,2\}$ and $u,v \in V(D_i)$.

\begin{example}
Let $\cP_n$ be the free $\KK$-module spanned by all labeled poset with $n$ vertices, and set $\cP = \bigoplus_{n \in \NN} \cP_n$.
This is a connected, graded Hopf algebra with product 
$ \nabla((D_1,\gamma_1) \otimes (D_2,\gamma_2)) = (D_1\sqcup D_2, \gamma_1\sqcup \gamma_2)$
and coproduct  $
\Delta((D,\gamma)) = \sum  (D|_S,\gamma|_S) \otimes (D|_T,\gamma|_T)
$
where the sum is over all disjoint decompositions $S\sqcup T = V(D)$ 
 with $S$ a lower set and $T$ an upper set.
 \end{example}

A \defn{(strict) composition} $\alpha = (\alpha_1,\alpha_2,\dots,\alpha_l)$ is a finite sequence of positive integers,
called its \defn{parts}.
We say that $\alpha$ is a composition of $|\alpha| := \sum_i \alpha_i\in\NN$.

\begin{example}
Fix a composition $\alpha$ and let $x_1,x_2,\dots$ be a countable sequence of commuting variables.
The \defn{monomial quasisymmetric function} of $\alpha$  is the power series 
$M_{\alpha} = \sum_{1\leq i_1<i_2<\dots<i_l} x_{i_1}^{\alpha_1} x_{i_2}^{\alpha_2}\cdots  x_{i_l}^{\alpha_l}.$
Let $\QSym =  \KK\spanning\{ M_\alpha : \alpha\text{ any composition}\}$ 
be the ring of quasisymmetric functions of bounded degree.
This ring is a graded connected Hopf algebras 
for the coproduct   
$\Delta(M_\alpha) = \sum_{\alpha=\alpha'\alpha''} M_{\alpha'}\otimes M_{\alpha''}$
where $\alpha'\alpha''$ denotes concatenation of compositions,
and the counit that acts on power series by setting 
$x_1=x_2=\dots=0$.
\end{example}

A \defn{partition} is a composition sorted into decreasing order. 
We sometimes write $\lambda = 1^{m_1}2^{m_2}\cdots$ to denote the partition with exactly $m_i$ parts equal to $i$. We also let $\ell(\lambda) = m_1+m_2+\dots$ denote the number of parts of $\lambda$.

\begin{example}
The \defn{elementary symmetric function} of a partition $\lambda$
is the product $e_\lambda := e_{\lambda_1}e_{\lambda_2}\cdots$ where $e_n := M_{1^n}$.
These power series are a basis for the Hopf subalgebra $\Sym \subset \QSym$ of symmetric functions of bounded degree.
\end{example}

\subsection{Combinatorial Hopf algebras}

Following \cite{ABS}, a \defn{combinatorial Hopf algebra} $(H,\zeta)$ is a graded, connected Hopf algebra $H$ of finite graded dimension with an algebra
homomorphism $\zeta : H \to \KK$.

\begin{example}
The pair $(\QSym,\zetaq)$ is an example of a combinatorial Hopf algebra,
where 
 $\zetaq : \QSym\to \KK$ is the map
$\zetaq(f) = f(1,0,0,\dots)$, which sends $M_{(n)} \mapsto 1$ and  $M_\alpha \mapsto 0$ for all compositions $\alpha$ with at least two parts. 
\end{example}

For a graph $G$ define 
 $\zeta_\cG(G) = 0^{|E(G)|}$
 where $0^0 := 1$.
 For a directed acyclic graph $D$ likewise set 
 $\zeta_\cO(D) = 0^{|E(D)|}$ for each directed acyclic graph $D$.
 These formulas extend to linear maps $\cG \to \KK$ and $\cO \to \KK$.
Finally let  $\zeta_\cP : \cP \to \KK$ be the linear map with $\zeta_\cP((D,\gamma))=1$
if $\gamma(u)<\gamma(v)$ for all edges $u\to v \in E(D) $
and
with $\zeta_\cP((D,\gamma))=0$ otherwise.
 
\begin{example}
The pairs $(\cG,\zeta_{\cG})$,
$(\cO,\zeta_{\cO})$, and
$(\cP,\zeta_{\cP})$ are all combinatorial Hopf algebras.
\end{example}

A morphism of combinatorial Hopf algebras
 $\Psi :(H,\zeta) \to (H',\zeta')$ is a graded Hopf algebra morphism $\Psi : H \to H'$ with $\zeta = \zeta' \circ \Psi$.
Results in \cite{ABS} show that there exists a unique morphism (with an explicit formula) from any combinatorial Hopf algebra to $(\QSym,\zetaq)$.
Moreover, the image of this morphism is contained in the Hopf subalgebra   $\Sym\subset \QSym$
if $H$ is cocommutative.
Of greatest relevance here \cite[Ex.~4.5]{ABS}, the unique morphism 
$(\cG,\zeta_\cG)  \to (\QSym,\zetaq)$
assigns  each graph $G$ to its chromatic symmetric function $X_G$.

%There is an explicit formula for this morphism in \cite{ABS}, which  gives the following for our examples above.

We might as well also describe the unique morphisms from $(\cO,\zeta_{\cO})$ and
$(\cP,\zeta_{\cP})$ to $ (\QSym,\zetaq)$. This requires a little more notation.
 A \defn{$P$-partition} \cite{St1} of a labeled poset $P=(D,\gamma)$
 is a map $\kappa : V(D) \to \PP$ with 
\ben
\item[(a)] $\kappa(u) \leq \kappa(v)$ whenever $u\to v \in E(D)$ and $\gamma(u) < \gamma(v)$,
and  
\item[(b)] $\kappa(u) < \kappa(v)$ whenever $u\to v \in E(D)$ and $\gamma(u) > \gamma(v)$.
\een
Define
$\Gamma(D,\gamma)= \sum_{\kappa} x^\kappa$ 
where the sum is over all $P$-partitions for $P=(D,\gamma)$.

Given directed acyclic graph $D$,  let  $\gamma^\op$ denote an arbitrary \defn{decreasing labeling} of $D$, meaning
a injective map $\gamma^\op : V(D) \to \ZZ$ with
\be\quad\text{$\gamma^\op(x) > \gamma^\op(y)$ whenever $x \to y \in E(D)$.}\ee
Then define $\Gamma(D) :=\Gamma(D,\gamma^\op) \in \NN\llbracket x_1,x_2,\dots\rrbracket $.

%For a directed acyclic graph $D$, let $\Gamma(D) =  \sum_{\kappa} x^\kappa \in \NN\llbracket x_1,x_2,\dots\rrbracket 
%$
%where the sum is over all maps $\kappa : V(D) \to \PP$ with 
%$\kappa(u) < \kappa(v)$ whenever $u\to v \in E(D)$.

 Finally, for any graph $G$, let $\AO(G)$ be its set of \defn{acyclic orientations}, that is, the
 directed acyclic graphs that recover $G$ when we ignore edge orientations.
The following is well-known \cite{ABS} and can be seen as a special case of Theorem~\ref{diagram2-thm}.
 
\begin{proposition}\label{diagram1-prop}
There is a commutative diagram of morphisms
{\[
\begin{tikzcd}
(\cG,\zeta_\cG) \arrow[r,hook] \arrow[dr] & (\cO,\zeta_\cO) \arrow[r,hook] \arrow[d]  & (\cP,\zeta_\cP) \arrow[dl] \\
  & (\QSym,\zetaq)  &  
  \end{tikzcd}
\]}%
in which the horizontal maps send $G\mapsto \sum_{D \in \AO(G)} D$
and $D \mapsto (D,\gamma^\op)$,
and the $\QSym$-valued maps send $G \mapsto X_G$
and $D \mapsto \Gamma(D)$
and $(D,\gamma) \mapsto \Gamma(D,\gamma)$.
\end{proposition}

This proposition recovers the well-known identity $X_G = \sum_{D \in \AO(G)}\Gamma(D)$. 

%In particular, the unique morphism 
%$(\cG,\zeta_\cG)  \to (\QSym,\zetaq)$
%assigns   $G$ to its chromatic symmetric function, which is the sum
%$X_G = \sum_{D \in \AO(G)}\Gamma(D)$. 

\section{$K$-theoretic generalizations}\label{main-sect}

We now explain how the results in the previous  can be extended ``$K$-theoretically'' to construct interesting quasisymmetric functions of unbounded degree, including $\bX_G$.
 This requires a brief discussion of \defn{linearly compact modules}.

\subsection{Linearly compact modules}\label{completions-sect}

Let $A$ and $B$ be $\KK$-modules
with a $\KK$-bilinear form $\langle\cdot,\cdot\rangle : A \times B \to \KK$.
Assume  $A$ is free and $\langle\cdot,\cdot\rangle$
is \defn{nondegenerate} in the sense that
$b\mapsto \langle \cdot,b\rangle$ is a bijection $B\to \Hom_\KK(A,\KK)$.

Fix a basis $\{a_i\}_{i \in I}$ for $A$.
For each $i \in I$, there exists a unique $b_i \in B$ with
$\langle a_i, b_j \rangle = \delta_{ij}$ for all $i,j \in I$,
and we identify $b \in B$
with the formal linear combination $\sum_{i \in I} \langle a_i ,b\rangle b_i$.
We call $\{b_i\}_{i \in I}$ a \defn{pseudobasis} for $B$.

We give $\KK$ the discrete topology.
Then the \defn{linearly compact topology} \cite[\S I.2]{Dieudonne} on $B$ is the coarsest topology 
in which the maps $\langle a_i, \cdot \rangle : B \to \KK$ are all continuous.
This topology depends on $\langle\cdot,\cdot\rangle$ but not on the choice of basis for $A$.
For a basis of open sets in the linearly compact topology, see \cite[Eq.\ (3.1)]{Marberg2018}.

\begin{definition}\label{lc-def}
A \defn{linearly compact} (or \defn{LC} for short) $\KK$-module is a $\KK$-module  $B$ with a nondegenerate
bilinear form $A \times B \to \KK$ for some free $\KK$-module $A$, given the linearly compact topology;
in this case we say that $B$ is the \defn{dual} of $A$.
Morphisms between such modules are continuous $\KK$-linear maps.
\end{definition}

Let $B$ and $B'$ be linearly compact $\KK$-modules dual to the free $\KK$-modules $A$ and $A'$.
Let $\langle\cdot,\cdot\rangle$ denote both of the 
associated forms. 
Every linear map $\phi : A' \to A$ has a unique adjoint $\psi : B\to B'$ such that 
$\langle \phi(a), b\rangle = \langle a,\psi(b)\rangle$.  
A linear map $B \to B'$ is continuous when it is the adjoint 
of a linear map $A' \to A$.

\begin{definition}
Define
$B \htimes B' := \Hom_\KK(A\otimes A',\KK)$ and give this the LC-topology from the pairing 
$(A\otimes A') \times \Hom_\KK(A\otimes A',\KK) \to \KK$.
\end{definition}

If $\{b_i\}_{i \in I}$ and $\{b'_j \}_{j \in J}$ are pseudobases for $B$ and $B'$,
then we can realize the \defn{completed tensor product} $B\htimes B'$ concretely as the linearly compact $\KK$-module
with the set of tensors $\{ b_i \otimes b_j'\}_{(i,j) \in I \times J}$ as a pseudobasis.

 \begin{example}\label{lc-ex0}
Let $A = \KK[x]$ and $B=\KK\llbracket x \rrbracket$.
Define $\langle\cdot,\cdot\rangle : A \times B \to \KK$ to be the nondegenerate $\KK$-bilinear form 
$\left\langle \sum_{n\geq 0} a_n x^n, \sum_{n \geq 0} b_n x^n\right\rangle := \sum_{n\geq 0} a_n b_n.$
Then the set $\{x^n\}_{n\geq 0}$ is a basis for $A$ and a pseudobasis for $B$,
and we have \[\KK\llbracket x \rrbracket \otimes \KK\llbracket y \rrbracket \neq \KK\llbracket x \rrbracket\htimes \KK\llbracket y \rrbracket \cong \KK\llbracket x,y \rrbracket.\]
\end{example}

Suppose $\nabla : B \htimes B \to B$ and $\iota : B \to \KK$ are continuous linear maps
which are the adjoints of linear maps $\epsilon : \KK \to A$ and $\Delta : A \to A \otimes A$.
We say that $(B,\nabla,\iota)$ is an \defn{LC-algebra}
if $(A,\Delta,\epsilon)$ is a $\KK$-coalgebra.
Similarly, we say that 
$\Delta : B \to B\htimes B$ and $\epsilon : B \to \KK$ 
make $B$ into an \defn{LC-coalgebra}
if $\Delta$ and $\epsilon$ are the adjoints of the product and unit 
maps of a $\KK$-algebra on $A$.
We define
\defn{LC-bialgebras} and \defn{LC-Hopf algebras} analogously; see \cite{Marberg2018}. 
If $B$ is an LC-Hopf algebra then its \defn{antipode} is the adjoint of the antipode of 
  $A$.

\subsection{Combinatorial LC-Hopf algebras}

Following \cite{Marberg2018}, define a \defn{combinatorial LC-Hopf algebra} to be a pair $(H,\zeta)$ where
\bei
\item[(a)] $H$ is an LC-Hopf algebra;
\item[(b)] $\zeta : H \to \KK\llbracket t\rrbracket$ is a homomorphism of linearly compact algebras;
\item[(c)] the counit of $H$ coincides with $\zeta(\cdot)|_{t\mapsto 0} $.
\eei
A morphism of combinatorial LC-Hopf algebras $\Psi :(H,\zeta) \to (H',\zeta')$ is a LC-Hopf algebra morphism $\Psi : H \to H'$ with $\zeta = \zeta' \circ \Psi$.

\begin{example}
Let $\mQSym$ be the set of all quasisymmetric power series in $\KK\llbracket x_1,x_2,\dots,\rrbracket$
of unbounded degree. 
 The (co)product, (co)unit, and antipode  $\QSym$ all extend to continuous $\KK$-linear maps 
that make $\mQSym$
into an LC-Hopf algebra, with $\{M_\alpha\}$ as a pseudobasis.
Then $(\mQSym,\wzetaq)$ is a combinatorial LC-Hopf algebra when $\wzetaq$ is the map $\wzetaq: f \mapsto f(t,0,0,\dots)$,
which sends
\[ \wzetaq : M_\alpha \mapsto \begin{cases} t^{|\alpha|} &\text{if $\alpha \in \{\emptyset, (1), (2), (3),\dots\}$} \\  0&\text{otherwise}.\end{cases} \]
\end{example}
%\[ \wzetaq(M_\alpha) = \begin{cases} t^{|\alpha|} &\text{if $\alpha \in \{\emptyset, (1), (2), (3),\dots\}$} \\  0&\text{otherwise}.\end{cases} \]

%\begin{example}
%Let $\Set(\PP)$ be the set of finite, nonempty subsets of $\PP$. 
%If $S,T \in \Set(\PP)$ then write $S \prec T$ if $\max(S) < \min(T)$ and $S\preceq T$ if $\max(S) \leq \min(T)$.
%Then for $n\in \PP$ and $S \subset[n-1]$, the power series  
%\[ \bL_{n,S} := \sum_{\substack{
%A_1,A_2,\dots,A_n \in \Set(\PP) \\ 
%A_1 \preceq A_2 \preceq \dots \preceq A_n \\ 
%A_i \prec A_{i+1}\text{ if }i \in S}} %\beta^{|A_1| + |A_2| + \dots +|A_n|-n} 
%x^{A_1}x^{A_2}\cdots x^{A_n}\]
%is in $\mQSym$ but not $\QSym$.
%For a nonempty composition $\alpha$ let $\bL_\alpha := \bL_{n,S}$ where $n=|\alpha|$ and $S = I(\alpha)$.
%Also set $\bL_\emptyset = 1$. The \defn{multifundamental quasisymmetric functions} $\bL_\alpha$ form another pseudobasis for $\mQSym$.
%% as $\alpha$ ranges over all compositions 
%\cite{LamPylyavskyy}.
%\end{example}

The preceding example is an instance of a general construction.
 If $A$ is a free $\KK$-module with basis $S$, then its \defn{completion}
$\overline A$ is the 
set of arbitrary $\KK$-linear combinations of $S$.
We view $\overline A$ as a  linearly compact $\KK$-module with $S$ as a pseudobasis,
relative 
to the nondegenerate bilinear form $A \times \overline A \to \KK$
 making $S$ orthonormal.
 
Suppose $H=\bigoplus_{n \in \NN}H_n$ is a connected Hopf algebra that is graded as an algebra with finite graded dimension,
and filtered as a coalgebra in the sense that if $h \in H_n$ then $\Delta(h) \in \bigoplus_{i+j \geq n} H_i \otimes H_j$.
If $\zeta : H \to \KK$ is any linear map,
then we write
 $\overline \zeta : \overline H \to \KK\llbracket t\rrbracket$ 
for the unique continuous linear map satisfying
\be\label{ozeta-eq} \overline \zeta (h) := \zeta(h) t^n\quad\text{for $n \in \NN$ and $h \in H_n$.}\ee
%Assume that $H_0 = \KK$ and that the unit and counit are the natural inclusion $H_0 \hookrightarrow H$  and projection $H \to H_0$
 There is a unique way of extending  the (co)unit and (co)product of $H$
to continuous linear maps on its completion
 $\overline H$, for which
the following holds.

\begin{lemma}\label{lc-ext-lem}
The   structures just given make $\overline H$ into an LC-Hopf algebra.
If $(H,\zeta)$ is a combinatorial Hopf algebra (so the coproduct is graded rather than just filtered),
then  $(\overline H,\overline\zeta)$ is a combinatorial LC-Hopf algebra,
and the unique morphism  $(H,\zeta)\to (\QSym,\zetaq)$ extends to a morphism 
$(\overline H,\overline\zeta)\to (\mQSym,\wzetaq)$.
\end{lemma}

\begin{proof}
In this setup, the graded dual $H^\ast = \bigoplus_{n \in \NN} \Hom_\KK(H_n,\KK)$
is a well-defined Hopf algebra, which is graded as a coalgebra with finite graded dimension,
and filtered as an algebra in the sense that the product sends 
\[\textstyle\Hom_\KK(H_i,\KK)\otimes \Hom_\KK(H_j,\KK)\to \bigoplus_{n\leq i+j} \Hom_\KK(H_n,\KK).\]
Our definition of $\overline H$  is an LC-Hopf algebra because it is the (ungraded) dual of $H^\ast$ in the sense of 
Definition~\ref{lc-def}, via the natural bilinear form $H^\ast \times \overline H \to \KK$.

Now suppose $(H,\zeta)$ is a combinatorial Hopf algebra.
Since $\zeta$ is an algebra homomorphism, the map $\overline \zeta$ is also an algebra homomorphism.
As $H$ is connected, we may assume that $H_0 = \KK$ and that the counit map $H \to H_0=\KK$ is the natural projection.
Then the algebra homomorphism $\zeta : H \to \KK$  must restrict to the identity map on $H_0=\KK$,
so the counit $\varepsilon$ of $\overline H$ evaluated at a generic element
$h = \sum_{n \in \NN} h_n \in \overline H$ with $h_n \in H_n$
has the formula 
$\varepsilon(h) = h_0 = \zeta(h_0) = \overline\zeta(h)|_{t=0}.$
We conclude that $(\overline H,\overline \zeta)$ is a combinatorial LC-Hopf algebra.

The morphism $(H,\zeta) \to(\QSym,\zetaq)$ is a graded linear map, so
it extends to a morphism of LC-Hopf algebras $\Psi:\overline H \to\mQSym$. 
This map satisfies 
\[\textstyle\wzetaq ( \Psi(h)) = \sum_{n \in \NN} \wzetaq ( \Psi(h_n))
 = \sum_{n \in \NN} \zetaq ( \Psi(h_n) )t^n
  = \sum_{n \in \NN} \zeta(h_n) t^n =   \overline\zeta(h) \]
  for any $h = \sum_{n \in \NN} h_n \in \overline H$ with $h_n \in H_n$,
  so we have $\wzetaq\circ \Psi = \overline\zeta$.
\end{proof}

The pair $(\mQSym,\wzetaq)$ is a final object in the category of combinatorial LC-Hopf algebras,
meaning there is a unique morphism $(H,\zeta) \to(\mQSym,\wzetaq)$ for any combinatorial LC-Hopf algebra.
Specifically,
if $H$ has
 coproduct $\Delta$,
then define $\Delta^{(0)} = \id_H$ and  
\[
\Delta^{(k)} 
= (\Delta^{(k-1)}\htimes \id) \circ \Delta 
= (\id \htimes \Delta^{(k-1)}) \circ \Delta
 : H \to H^{\htimes(k+1)} 
 \quad\text{for }k\in \PP.
 \]
 For each nonempty composition $\alpha = (\alpha_1,\alpha_2,\dots,\alpha_k)$,
let $\zeta_\alpha : H \to \KK$ be the map sending $h \in H$ to the coefficient of $t^{\alpha_1}\otimes t^{\alpha_2}\otimes \cdots \otimes t^{\alpha_k}$
in $\zeta^{\otimes k} \circ \Delta^{(k-1)}(h) \in \KK\llbracket t \rrbracket$.
When $\alpha = \emptyset$ is empty let $\zeta_\emptyset := \zeta(\cdot)|_{t\mapsto 0}$ be the counit of $H$.

\begin{theorem}[{\cite[Thm.~2.8]
{LM2019}}]
\label{lc-thm}
If $(H,\zeta)$ is a combinatorial LC-Hopf algebra then  
$\Psi_{H,\zeta} : h \mapsto \sum_\alpha \zeta_\alpha(h) M_\alpha$
is the unique morphism  $(H,\zeta)\to (\mQSym,\wzetaq)$.
\end{theorem}

\begin{remark}
Let $\mSym$ be the LC-Hopf subalgebra of symmetric functions in $\mQSym$.
When $H$ cocommutative, it is apparent from the formula given above that the morphism $\Psi_{H,\zeta}$ has its image in $\mSym$.
\end{remark}

 \subsection{Graphs and weighted graphs}
 
For each graph $G$ define 
\be\label{bDel-eq}\textstyle \bDelta(G) := \sum_{S\cup T = V(G)}   G|_S \otimes G|_T.\ee
This differs
from our other coproduct in allowing vertex decompositions that are not disjoint.
Use the continuous linear extension of \eqref{bDel-eq}
to replace the natural coproduct in $\wG$
and denote the resulting modified structure as $\mG$.
Also let $\wzeta_\cG : \mG \to \KK\llbracket t\rrbracket$
be the continuous linear map 
with 
\[\wzeta_\cG(G) := 0^{|E(G)|} t^{|G|}\]
so that $\wzeta_\cG$ extends
 $\zeta_\cG$
via the formula \eqref{ozeta-eq}.

The kromatic symmetric function $\bX_G$ emerges from these constructions
via the theory in the previous section, in the following sense:

\begin{proposition}\label{cG-lc-prop}
The pair $(\mG,\wzeta_\cG)$ is a combinatorial LC-Hopf algebra,
and the unique morphism 
$(\mG,\wzeta_\cG)  \to (\mQSym,\wzetaq)$
assigns  each graph $G$ to its kromatic symmetric function $\bX_G$.
\end{proposition}

We skip the proof of this result and instead derive a slight generalization.
A \defn{weighted graph} is a pair $(G,\omega)$ consisting of a (finite, undirected) graph $G$ 
and a map $\omega : V(G) \to \PP$ assigned positive integer weights to all vertices.
We do not distinguish between two weighted graphs $(G_1,\omega_1)$ and $(G_2,\omega_2)$
if there is a graph isomorphism $\phi : G_1 \to G_2$ with $\omega_1 = \omega_2 \circ \phi$.
%Given such a pair, define $|\omega| = \sum_{v \in V(G)} \omega(v) \geq |G|$.

\begin{definition}[\cite{CPS}]
The \defn{kromatic symmetric function} of 
 a weighted graph $(G,\omega)$ is
 $\bX_{(G,\omega)} = \sum_\kappa \prod_{v \in V(G)}  \Bigl(\prod_{i \in \kappa(v)}  x_i\Bigr)^{\omega(v)} \in\mSym$
 where the sum is 
 over all proper set-valued colorings $\kappa : V(G)\to \Set(\PP)$.
\end{definition}

Recall from Example~\ref{kro-ex} that ${r \brace n}$ is the Stirling number of the second kind.
For a partition $\nu$, write $m_\nu = \sum_{\sort(\alpha)=\nu} x^\alpha  \in \Sym$ for the usual monomial symmetric function.

\begin{example}\label{kam-ex}
Crew, Pechenik, and Spirkl \cite[\S3.2]{CPS} define the
\defn{$K$-theoretic augmented monomial symmetric function} of a partition $\lambda$
to be $\kam_\lambda := \bX_{(G,\omega)}$  for the complete graph $G = K_{\ell(\lambda)}$  with weight map $\omega: i\mapsto \lambda_i$.
Let $r_i(\lambda) $ be the number of parts of $\lambda$ equal to $i$.
One has $\kam_{\emptyset} =1$, and if $\lambda$ is nonempty then 
\[
\textstyle \kam_\lambda = \sum_\nu \Bigl(  \prod_{i \in \PP} { r_i(\nu) \brace r_i(\lambda) } \cdot r_i(\lambda)!\Bigr) m_\nu
\]
where the sum is over all integer partitions $\nu$ that have the same set of parts $\{\nu_1,\nu_2,\dots\} = \{\lambda_1,\lambda_2,\dots\}$ as $\lambda$ and that satisfy $r_i(\nu) \geq r_i(\lambda)$ for all $i \in \PP$.
\end{example}

It is natural to consider this generalization of the kromatic symmetric function since one can expand $\bX_G$ 
into a sum of $\kam_\lambda$'s
via a deletion-contraction relation \cite[\S3.3]{CPS}. This is perhaps the most efficient method available to 
actually compute $\bX_G$. 

We can again interpret $\bX_{(G,\omega)}$ via combinatorial LC-Hopf algebras.
Let $\wGwt$ be the linearly compact $\KK$-module with a pseudobasis provided by
all weighted graphs. 
The LC-Hopf algebra structure on $\mG$ extends to $\wGwt$ in the following way. 
The product 
is again given by disjoint union, the unit element is the unique weighted graph with no vertices,
and the counit sends $(G,\omega)\mapsto 0^{|G|}$.
Extend the coproduct $\bDelta $ of $\mG$ by setting
\be\textstyle\bDelta((G,\omega)) := \sum_{S\cup T = V(G)}   (G|_S,\omega|_S) \otimes (G|_T, \omega|_T)\ee
for each weighted graph $(G,\omega)$, and let
\be\textstyle \wzetaw((G,\omega)) := 0^{|E(G)|} \prod_{v \in V(G)} t^{\omega(v)}.\ee
This extends to a  continuous algebra homomorphism
$\wGwt \to \KK\llbracket t \rrbracket$.
It is important to require our vertex weightings $\omega : V(G) \to \PP$
to take positive values since if we allowed $\omega : V(G) \to\NN$ then $\wzetaw$ would not be well-defined
when applied to infinite linear combinations of weighted graphs.

\begin{proposition} \label{wzetaw-prop}
The pair $(\wGwt,\wzetaw)$ is a combinatorial LC-Hopf algebra,
and the unique morphism 
$(\wGwt,\wzetaw)  \to (\mQSym,\wzetaq)$
assigns  each weighted graph $(G,\omega)$ to its kromatic symmetric function $\bX_{(G,\omega)}$.
\end{proposition}

We recover Proposition~\ref{cG-lc-prop} by identifying
$\mG$ with the LC-Hopf subalgebra with a pseudobasis given by the weighted graphs $(G,\omega)$
with $\omega : G \to \{1\}$.

\begin{proof}
Restricting the (co)unit and (co)product of $\wGwt$ to the $\KK$-span of all weighted graphs gives a Hopf algebra satisfying the conditions in Lemma~\ref{lc-ext-lem}:
the relevant axioms are straightforward to check.

For example, the product is clearly associative and gives a graded algebra structure of finite graded dimension
when we set the degree of $(G,\omega)$ to be $\sum_{v \in V(G)} \omega(v)$.
For this grading, the coproduct similarly defines a filtered coalgebra.

The compatibility axiom for the product and coproduct holds since if $\tau : M\otimes N \to N \otimes M$ is the usual twist morphism for $\KK$-modules,
then both $\bDelta \circ \nabla $ and 
$(\nabla\otimes \nabla) \circ (1 \otimes \tau \otimes 1)\circ (\bDelta \otimes \bDelta)  $
applied to the tensor $(G_1,\omega_1)\otimes (G_2,\omega_2)$ of two weighted graphs 
 expand as
\[
 \sum_{
\substack{
S_1\cup T_1 = V(G_1)
\\
S_2\cup T_2 = V(G_2)
}}
\( (G_1|_{S_1},\omega_1|_{S_1})\sqcup  (G_2|_{S_2},\omega_2|_{S_2})\)
\otimes
\( (G_1|_{T_1},\omega_1|_{T_1})\sqcup  (G_2|_{T_2},\omega_2|_{T_2})\).
\]

We conclude from Lemma~\ref{lc-ext-lem} that $\mG$ is an LC-Hopf algebra.
The map $\wzetaw$ is a continuous algebra homomorphism
with $\wzetaw(\cdot)|_{t\mapsto 0}$ equal to the counit of $\wGwt$,
so
  $(\wGwt,\wzetaw)$ is a combinatorial LC-Hopf algebra.

 Let $\zeta = \wzetaw$ and $\Psi:=\Psi_{\wGwt,\wzetaw}$ as in Theorem~\ref{lc-thm}. 
 For a weighted graph $(G,\omega)$ and a subset $S\subseteq V(G)$, 
 let $|S|_\omega := \sum_{i \in S} \omega(i)$. Then  
\[
  \zeta^{\otimes k}\circ \bDelta^{(k-1)}((G,\omega))  = \sum
  t^{|S_1|_\omega} \otimes t^{|S_2|_\omega}\otimes \cdots \otimes t^{|S_k|_\omega} \]
  for each $k \in \KK$,
  where the sum is over all (not necessarily disjoint) decompositions
$  S_1 \cup S_2 \cup \dots \cup S_k = V(G)$ such that $G|_{S_i}$ is a discrete graph for all $i \in [k]$, so that $E(G|_{S_i})=\varnothing$.
Therefore
\[
 \Psi((G,\omega))  = \sum_{k \in \NN} \sum_{
  \substack{S_1 \cup S_2 \cup \dots \cup S_k = V(G) 
  \\
  E(G|_{S_i})  = \varnothing\ \forall i\in[k]}
  } \sum_{1\leq i_1<i_2<\dots <i_k} x_{i_1}^{|S_1|_\omega} x_{i_2}^{|S_2|_\omega} \cdots x_{i_k}^{|S_k|_\omega}.
\]
The monomial in the innermost sum is exactly 
$\textstyle\prod_{v \in V(G)}  \Bigl(\prod_{i \in \kappa(v)}  x_i\Bigr)^{\omega(v)}$ for the proper set-valued coloring of $G$
with $\kappa(v) = \{ i_j : j \in [k]\text{ with } v \in S_j\}$ for $v \in V(G)$,
and as the parameters indexing the sums vary, the corresponding $\kappa$ ranges over all proper set-valued colorings of $G$.
Thus $ \Psi((G,\omega)) = \bX_{(G,\omega)}$.
\end{proof}

Recall that $r_i(\lambda) $ is the number of parts equal to $i$ in a partition $\lambda$.
If $\lambda$ and $\nu$ are partitions with $r_i(\lambda) \leq r_i(\nu)$ for all $i$ then let $\nu\ominus \lambda$ be the partition with 
$r_i(\nu\ominus \lambda) = r_i(\nu) - r_i(\lambda)$ for all $i$. 
Also let $\nu\oplus \lambda$ be the partition with $r_i(\nu\oplus\lambda) = r_i(\nu) + r_i(\lambda)$ for all $i$.

 \begin{corollary}
 If $\nu$ is a partition then 
 \[\textstyle \Delta(\kam_\nu) = \sum_{\lambda,\mu} \Big( \prod_{i\in\PP} \tbinom{r_i(\nu)}{r_i(\nu\ominus \lambda),\ r_i(\nu\ominus\mu),\ r_i((\lambda\oplus \mu)\ominus\nu)} \Big)\kam_\lambda \otimes \kam_\mu
 %\frac{r_i(\nu)!}{(r_i(\nu)- r_i(\lambda))!   (r_i(\nu)-r_i(\mu))!   (r_i(\lambda) + r_i(\mu) - r_i(\nu))!} \kam_\lambda \otimes \kam_\mu
 \]
 where the sum is over all partitions $\lambda$ and $\mu$ with 
 \[0\leq r_i(\lambda) \leq r_i(\nu)\quand 0\leq r_i(\mu)\leq r_i(\nu)\leq r_i(\lambda)+r_i(\mu)\quad\text{for all }i\in \PP.\]
 \end{corollary}
 
 \begin{proof}
Proposition~\ref{wzetaw-prop} implies that 
 \be
 \textstyle\Delta( \bX_{(G,\omega)}) = \sum_{S\cup T = V(G)}  \bX_{(G|_S,\omega|_S)} \otimes \bX_{(G|_S,\omega|_S)}
 \ee
 for all weighted graphs $(G,\omega)$. Applying this to Example~\ref{kam-ex} gives
 \be
 \textstyle\Delta(\kam_\nu) = \sum_{S\cup T= [\ell(\nu)]} \kam_{\nu_S} \otimes \kam_{\nu_T}
 \ee
 where if $S= \{i_1<i_2<\dots\}$ then $\nu_S$ means the partition $(\nu_{i_1} \geq \nu_{i_2} \geq \dots)$.
It is a straightforward exercise to check that the number of not necessarily disjoint unions $S\cup T  = [\ell(\nu)]$
that give $\nu_S=\lambda$ and $\nu_T=\mu$ is the product of multinomial coefficients $ \prod_{i\in\PP} \tbinom{r_i(\nu)}{r_i(\nu\ominus \lambda),\ r_i(\nu\ominus\mu),\ r_i((\lambda\oplus \mu)\ominus\nu)} $.
 \end{proof}
\subsection{Set-valued $P$-partitions}

If $S,T \in \Set(\PP)$ then write $S \prec T$ if $\max(S) < \min(T)$ and $S\preceq T$ if $\max(S) \leq \min(T)$.
For a labeled poset $(D,\gamma)$ define
\[
\bGamma(D,\gamma)= \sum_{\kappa} x^\kappa
 \in \KK\llbracket x_1,x_2,\dots\rrbracket \]
where the sum is over all maps $\kappa : V(D) \to \Set(\PP)$ 
that are \defn{set-valued $P$-partitions} for $P=(D,\gamma)$ in the sense of \cite{LamPylyavskyy,LM2019}, meaning   
\ben
\item[(a)] $\kappa(u) \preceq \kappa(v)$ whenever $u\to v \in E(D)$ and $\gamma(u) < \gamma(v)$,
and  
\item[(b)] $\kappa(u) \prec \kappa(v)$ whenever $u\to v \in E(D)$ and $\gamma(u) > \gamma(v)$.
\een
%Such maps $\kappa$ are called \defn{set-valued $P$-partitions} for $P=(D,\gamma)$ in \cite{LamPylyavskyy,LM2019}.

\begin{example}
If $D = (1 \to 2 \to 3 \to \dots \to n)$ is an $n$-element chain and $S$ is the set of $i \in [n-1] $ with $\gamma(i) > \gamma(i+1)$ then 
we define
$ \bL_{n,S}:=\bGamma(D,\gamma) $.
Following \cite{LamPylyavskyy}, we construct the \defn{multifundamental quasisymmetric function} of a composition $\alpha$  as the power series
$\bL_{\alpha} :=  \bL_{n,S}$ where $n=|\alpha|$ and $S =I(\alpha)$ for 
\[  I(\alpha):= \{\alpha_1,\alpha_1+\alpha_2,\alpha_1+\alpha_2+\alpha_3,\dots\}\setminus\{n\}.\]
These elements form another pseudobasis for $\mQSym$ \cite{LamPylyavskyy}.
We say that a power series in $\mQSym$ is \defn{multifundamental positive} if
it can be expressed as a possibly infinite sum of (not necessarily distinct) $\bL_{\alpha} $'s.
\end{example}

 For a directed acyclic graph $D$, let
 $\bGamma(D) = \bGamma(D,\gamma^\op) =\sum_{\kappa} x^\kappa$
where the sum is over all maps $\kappa : V(D) \to \Set(\PP)$ with 
$\kappa(u) \prec \kappa(v)$ if $u\to v \in E(D)$.

\begin{example}\label{bare-ex}
If $D = (1 \to 2 \to 3 \to \dots \to n)$ is an $n$-element path
then define 
$ \bare_n := \bGamma(D)
% = \sum_{\substack{S = S_1\sqcup S_2 \sqcup \dots \sqcup S_n \\ 
%S_1,S_2,\dots,S_n \in \Set(\PP) \\ S_1\prec S_2\prec \dots \prec S_n}} 
%  x^{S}
  = \sum_{k=n}^\infty \tbinom{k-1}{n-1}  e_{k}.$
For each partition $\lambda$
let $\bare_\lambda := \bare_{\lambda_1}\bare_{\lambda_2}\cdots.$
These functions are a pseudobasis for $\mSym$.
\end{example}

Recall from Example~\ref{kro-ex} that $ \bX_{K_n}= n!\sum_{r= n}^\infty {r \brace n} e_r $

\begin{proposition}\label{bXk-prop}
If $n \in \PP$ then $ \bX_{K_n}=  n! \sum_{r= n}^\infty {r -1\brace n-1} \bare_r$.
\end{proposition}

\begin{proof}
Substituting the $e$-expansion of $\bare_n$ from Example~\ref{bare-ex} gives  \[ 
 \sum_{s= n}^\infty {\textstyle{s -1\brace n-1}} \bare_s =
  \sum_{s= n}^\infty\sum_{r=s}^\infty {\textstyle{s -1\brace n-1}}\tbinom{r-1}{s-1} e_{r} = 
   \sum_{r=n}^\infty \(\sum_{s= n-1}^{r-1}  {\textstyle{s\brace n-1}}\tbinom{r-1}{s}\) e_r .
\]
One has $\sum_{s= n-1}^{r-1}  {\textstyle{s \brace n-1}}\tbinom{r-1}{s} =\sum_{S\subseteq [r-1]} {|S| \brace n-1} = {r \brace n}$
since 
to form an $n$-block set partition of $[r]$, one can first choose any subset of $[r-1]$ to form the complement of the block containing $r$, and then divide this set into $n-1$ blocks.
Thus, we have $ n! \sum_{s= n}^\infty {s -1\brace n-1} \bare_s = n!\sum_{r= n}^\infty {r \brace n} e_r = \bX_{K_n}$.
\end{proof}

A \defn{multilinear extension} of a directed acyclic graph $D$  with $n$ vertices 
is a sequence $w=(w_1,w_2,\dots,w_N)$ with $V(D) = \{w_1,w_2,\dots,w_N\}$
such that $i<j$ whenever $w_i \to w_j \in E(D)$, and $w_i \neq w_{i+1}$ for all $i \in [N-1]$.
If $\mL(D)$ is the set of all multilinear extensions of $D$
and $\gamma : V(D) \to \ZZ$ is injective, then
\be\label{ML-eq}\bGamma(D,\gamma) = \sum_{w \in \mL(D)}  \bL_{\ell(w), \Des(w,\gamma)}\ee
where $\Des(w,\gamma) := \{ i\in [\ell(w)-1] : \gamma(w_i)>\gamma(w_{i+1})\}$ for $w \in \mL(D)$ \cite{LamPylyavskyy}.

 \subsection{Acyclic multi-orientations}
 
 Let $G$ be a graph. 
 An \defn{acyclic multi-orientation} of $G$ 
is an acyclic orientation of the $\alpha$-clan graph $\Cl_\alpha(G)$ from Remark~\ref{clan-rmk}
for some $\alpha : V(G) \to \PP$,
such that for each  $v \in V(G)$ and each $i \in [\alpha(v)-1]$ there are at least two directed paths 
from the vertex $(v,i+1) \in V(\Cl_\alpha(G))$ to the vertex $(v,i) \in V(\Cl_\alpha(G))$.

Departing from our usual convention, we consider two acyclic multi-orientations to be the same if and only if they have the same vertices and the same directed edges, rather than if they are merely isomorphic.
Let $\MAO(G)$ be the set of all acyclic multi-orientations of $G$.
We say that $D \in \MAO(G)$ has type $\alpha$ if $D$ is an acyclic orientation of $\Cl_\alpha(G)$.
 
\begin{proposition}\label{MAO-prop}
Let $D$ be an acyclic orientation of $\Cl_\alpha(G)$.
Then $D \in \MAO(G)$ if and only if for each $v \in V(G)$ both of the following properties hold:
\ben
\item[(a)] If $i,j \in [\alpha(v)]$ have $i>j$  then $(v,i)\to (v,j)$ is a directed edge in $D$.
\item[(b)] If $i \in [\alpha(v)-1]$ then there exists a directed path in $D$ 
 involving no edges of the form $(v,j) \to (v,k)$
from $(v,i+1)$ to $(v,i)$.
\een
 \end{proposition}

 \begin{proof}
 If $D$ has property (a) then $(v,i+1)\to (v,i)$ is a directed path in $D$, and if $D$ also has property (b)
 then there is a second directed path from $(v,i+1)$ to $ (v,i)$, so $D \in \MAO(G)$.
 
Conversely, assume $D \in \MAO(G)$.
 Suppose $i,j \in [\alpha(v)]$ and $i>j$. Then $D$ contains a directed path from $(v,i)$ to $(v,j)$,
so  $(v,j)\to (v,i)$ cannot be an edge in $D$ as this would create a cycle.  As $\{(v,i),(v,j)\} $ is an edge
in $\Cl_\alpha(G)$, the edges of $D$ must contain $(v,i)\to (v,j)$ instead.

Since $D$ is acyclic, no directed path in $D$ passing through the vertex $(v,j)$ can ever reach $(v,k)$
for $k \geq j$. Thus, a directed path in $D$ from $(v,i+1)$ to $(v,i)$ can only involve an edge of the form
$(v,j) \to (v,k)$ if the path has length one and $j=i+1$ and $k=i$.
As $D$ contains a second directed path from $(v,i+1)$ to $(v,i)$ in addition to this one, property (b) must hold. 
\end{proof}

\begin{example}\label{MAO-ex}
If $|E(G)|=0$ then  $\MAO(G) = \{G\}$ and otherwise
 $\MAO(G)$ is infinite.
If $G=K_2$ % is the complete graph on two vertices
  then $\MAO(G)$ 
has two (isomorphic but considered to be distinct) elements with $k$ vertices for each $k\geq 2$; see Figure~\ref{MAO-fig}.
\end{example}

\begin{figure}[h]
\[\centerline{
    \begin{tikzpicture}[xscale=1.2, yscale=1.8,>=latex,inner sep=0.5mm,baseline=(z.base)]
    \node at (0,1.25) (z) {};
      \node at (-.25,2) (T1) {$v_1$};
      \node at (0,1) (T2) {$v_2$};
      \node at (-.25,0) (T3) {$v_3$};
      \node at (1.25,2.5) (U1) {$w_1$};
      \node at (1,1.5) (U2) {$w_2$};
      \node at (1.25,0.5) (U3) {$w_3$};
      \draw[->]  (T3) -- (T2);
      \draw[->]  (T2) -- (T1);
      \draw[->]  (U3) -- (U2);
      \draw[->]  (U2) -- (U1);
      \draw[->]  (T3) -- (U3);
      \draw[->]  (U3) -- (T2);
      \draw[->]  (T2) -- (U2);
      \draw[->]  (U2) -- (T1);
      \draw[->]  (T1) -- (U1);
      \draw[->]  (T3) -- (T1);
      \draw[->]  (U3) -- (U1);
     \end{tikzpicture}
     \qquad
         \begin{tikzpicture}[xscale=1.2, yscale=1.8,>=latex,inner sep=0.5mm,baseline=(z.base)]
    \node at (0,1.25) (z) {};
      \node at (-0.25,2.5) (T1) {$v_1$};
      \node at (0,1.5) (T2) {$v_2$};
      \node at (-0.25,0.5) (T3) {$v_3$};
      \node at (1.25,2) (U1) {$w_1$};
      \node at (1,1) (U2) {$w_2$};
      \node at (1.25,0) (U3) {$w_3$};
      \draw[->]  (T3) -- (T2);
      \draw[->]  (T3) -- (T1);
      \draw[->]  (T2) -- (T1);
      \draw[->]  (U3) -- (U2);
      \draw[->]  (U3) -- (U1);
      \draw[->]  (U2) -- (U1);
      \draw[->]  (U3) -- (T3);
      \draw[->]  (T3) -- (U2);
      \draw[->]  (U2) -- (T2);
      \draw[->]  (T2) -- (U1);
      \draw[->]  (U1) -- (T1);
     \end{tikzpicture}
\qquad
    \begin{tikzpicture}[xscale=1.2, yscale=1.8,>=latex,inner sep=0.5mm,baseline=(z.base)]
    \node at (0,1.5) (z) {};
      \node at (-0.5,3) (T0) {$v_1$};
      \node at (0,2) (T1) {$v_2$};
      \node at (0,1) (T2) {$v_3$};
      \node at (-0.5,0) (T3) {$v_4$};
      \node at (1.25,2.5) (U1) {$w_1$};
      \node at (1,1.5) (U2) {$w_2$};
      \node at (1.25,0.5) (U3) {$w_3$};
      \draw[->]  (T3) -- (T2);
      \draw[->]  (T3) -- (T1);
      \draw[->]  (T3) -- (T0);
      \draw[->]  (T2) -- (T1);
      \draw[->]  (T2) -- (T0);
      \draw[->]  (T1) -- (T0);
      \draw[->]  (U3) -- (U2);
      \draw[->]  (U2) -- (U1);
      \draw[->]  (U3) -- (U1);
      \draw[->]  (T3) -- (U3);
      \draw[->]  (U3) -- (T2);
      \draw[->]  (T2) -- (U2);
      \draw[->]  (U2) -- (T1);
      \draw[->]  (T1) -- (U1);
      \draw[->]  (U1) -- (T0);
     \end{tikzpicture}
\qquad
    \begin{tikzpicture}[xscale=1.2, yscale=1.8,>=latex,inner sep=0.5mm,baseline=(z.base)]
    \node at (0,1.5) (z) {};
      \node at (-0.25,2.5) (T1) {$v_1$};
      \node at (0,1.5) (T2) {$v_2$};
      \node at (-0.25,0.5) (T3) {$v_3$};
      \node at (1.5,3) (U0) {$w_1$};
      \node at (1,2) (U1) {$w_2$};
      \node at (1,1) (U2) {$w_3$};
      \node at (1.5,0) (U3) {$w_4$};
      \draw[->]  (T3) -- (T2);
      \draw[->]  (T3) -- (T1);
      \draw[->]  (T2) -- (T1);
      \draw[->]  (U3) -- (U2);
      \draw[->]  (U3) -- (U1);
      \draw[->]  (U3) -- (U0);
      \draw[->]  (U2) -- (U1);
      \draw[->]  (U2) -- (U0);
      \draw[->]  (U1) -- (U0);
      \draw[->]  (U3) -- (T3);
      \draw[->]  (T3) -- (U2);
      \draw[->]  (U2) -- (T2);
      \draw[->]  (T2) -- (U1);
      \draw[->]  (U1) -- (T1);
      \draw[->]  (T1) -- (U0);
     \end{tikzpicture}
}\]
\caption{
The distinct acyclic multi-orientations with 6 and 7 vertices for the complete graph with vertex set $\{v,w\}$. Here, we have written the vertices of these directed graphs  
as $v_i$ and $w_i$ instead of $(v,i)$ and $(w,i)$ to save space.
}\label{MAO-fig}
\end{figure}

A \defn{source} in a directed graph is a vertex with no incoming edges.
One can relate the coefficients  in the $\bare$-expansion of $\bX_G$
to the source counts of its acyclic multi-orientations, generalizing a result of Stanley \cite[Thm.~3.3]{Stanley95}.

\begin{remark}
For this type of statement to make sense,
the relevant coefficients must be uniquely defined integers,
so we should assume that the integral domain $\KK$ contains $\ZZ$.
Since we need to divide by $2$ in the next lemma, for the rest of this subsection we 
make the slightly strongly assumption that $\KK\supseteq \QQ$.
\end{remark}

%We first  prove a quick lemma.
Define $\varphi : \mQSym \to \KK\llbracket t\rrbracket$ to be the continuous linear map with 
\[ \varphi(\bL_{n,S}) = \begin{cases} t(\tfrac{t-1}{2})^i&\text{if }S = [n-1] \setminus [i]\text{ for some }0\leq i <n
\\
0&\text{otherwise}
\end{cases}
\]
for $n \in \PP$ and $S\subseteq [n-1]$. We also set $\varphi(\bL_{0,\varnothing}) = 1$.
The next result generalizes the key \emph{Claim} in the proof of \cite[Thm.~3.3]{Stanley95}, and follows by a similar argument.
\begin{lemma}\label{m-lem} If $D$ is a directed acyclic graph then $\varphi(\bGamma(D)) = t^{m}$ where $m$ is the number of source vertices in $D$.
\end{lemma}

\begin{proof}
Choose a decreasing labeling $\gamma^\op : V(D) \to \ZZ$ for $D$. % : V(D) \to \ZZ$ with $\gamma^\op(u) > \gamma^\op(v)$ whenever $u \to v \in E(D)$.
Then by \eqref{ML-eq} we have $\bGamma(D)  =  \sum_{w \in \mL(D)} \bL_{\ell(w), \Des(w,\gamma^{\op})}$.

The only way to obtain a multilinear extension $w=(w_1,w_2,\dots,w_n)$ of $D$ with $\Des(w,\gamma^\op) = [n-1] \setminus [i]$ is as follows. Let $w_{i+1} \in V(D)$ be the vertex maximizing $\gamma^\op(w_{i+1})$; this element must be  a source vertex. Then define $w_1,w_2,\dots,w_i$ to be any $i$ distinct sources vertices other than $w_{i+1}$ listed in increasing order of their labels under $\gamma^{\op}$. 
Finally, choose a subset of $\{w_1,w_2,\dots,w_i\}$  
and
define $w_{i+2},w_{i+3},\dots,w_{n}$ to consist of these vertices 
plus all elements of $V(D) \setminus \{w_1,w_2,\dots,w_{i+1}\}$
arranged
in decreasing order of their $\gamma^{\op}$ labels.

If $D$ has $m$ source vertices then there are $\binom{m-1}{i}$ choices for $(w_1,w_2,\dots,w_i)$
and $2^i$ choices for the subset of these vertices repeated in   $(w_{i+2},w_{i+3},\dots,w_{n})$, so using the binomial formula we get
$\textstyle
\varphi(\bGamma(D)) 
= t  \sum_{i=0}^{m-1}\tbinom{m-1}{i} 2^i(\tfrac{t-1}{2})^i  
 =
t^m
$.
\end{proof}

\begin{proposition} Let $G$ be a graph and suppose 
$\bX_G = \sum_{\lambda} c_\lambda \bare_\lambda$ for some coefficients $c_\lambda \in \ZZ$.
Then the number of acyclic multi-orientations of $G$ with exactly $j$ sources and $k$ vertices is $\sum_{{\ell(\lambda)=j, |\lambda| =k}} c_\lambda\in \NN$.
\end{proposition}

\begin{proof}
Let $\source(G,j,k)$ be the number of acyclic multi-orientations of $G$ with exactly $j$ sources and $k$ vertices.
Applying $\varphi$ to $\bX_G = \sum_{D \in \MAO(G)}   \bGamma(D)$ using  Lemma~\ref{m-lem}
gives $\varphi(\bX_G) = \sum_{D \in \MAO(G)} \source(G,j,k) t^j$.
On the other hand, $\bare_\lambda = \bGamma(D)$ where $D$ is the disjoint union of directed chains of 
sizes $\lambda_1,\lambda_2,\dots$, so $\varphi(\bare_\lambda) = t^{\ell(\lambda)}$.
Therefore we also have $\varphi(\bX_G) = \sum_{\lambda} c_\lambda  t^{\ell(\lambda)}$.
The result follows by taking the coefficients of $t^j$ in both formulas.
\end{proof}

As noted in \cite{CPS},  the coefficients $c_\lambda$ appearing in $\bX_G = \sum_{\lambda} c_\lambda \bare_\lambda$
can be negative, even  when $G$ in the incomparability graph of a $(3+1)$-free poset.

\begin{example}
Fix $n>0$. Then
every $D \in \MAO(K_n)$ is an acyclic orientation of a nonempty complete graph, so has a single source vertex. Consulting Proposition~\ref{bXk-prop},
we see that the number of $k$-vertex acyclic multi-orientations of the complete graph $K_n$ is zero if $k<n$ and $n! {k -1\brace n-1}$ if $k\geq n$.
For $n=2$ this number is always $2$, which matches the description of $\MAO(K_2)$ in Example~\ref{MAO-ex}.
\end{example}

\subsection{Diagrams and isomorphisms}

For each directed acyclic graph $D$ and labeled poset $P = (D,\gamma)$,
 define 
\be\label{blck-eq}\textstyle
\bDelta(D) = \sum  D|_S \otimes D|_T
\quand
\bDelta(P) = \sum  (D|_S,\gamma|_S) \otimes (D|_T,\gamma|_T),
\ee
where both sums are over all (not necessarily disjoint) vertex decompositions $S\cup T =  V(D)$ in which $S$ is a lower set, $T$ is an upper set, and $S\cap T$ is an antichain.

Use the continuous linear extensions of the operations \eqref{blck-eq}
to replace the coproducts in the completions of  $\cO$ and $\cP$,
and denote the resulting structures as $ \mO$ and $ \mP$.
%to distinguish them from the completions 
%$ \wO$ and $ \wP$
%specified in Lemma~\ref{lc-ext-lem}.
Also let $\wzeta_\cO$ and $\wzeta_\cP$ 
be the continuous linear maps to $\KK\llbracket t\rrbracket$
extending 
$\zeta_\cO$ and $\zeta_\cP$
as in \eqref{ozeta-eq}.

\begin{theorem}\label{diagram2-thm}
Each of the pairs 
\[(\mG,\wzeta_\cG), \quad (\mO,\wzeta_\cO), \quand (\mP,\wzeta_\cP) \] is a combinatorial LC-Hopf algebra,
and there is a commutative diagram 
of morphisms of combinatorial LC-Hopf algebras
{\[
\begin{tikzcd}
(\mG,\wzeta_\cG) \arrow[r,hook] \arrow[dr] & (\mO,\wzeta_\cO) \arrow[r,hook] \arrow[d]  & (\mP,\wzeta_\cP) \arrow[dl] \\
  & (\mQSym,\wzetaq)  &  
  \end{tikzcd}
\]}%
in which the horizontal maps send graphs and DAGs to
\[\textstyle G\mapsto \sum_{D \in \MAO(G)} D\quand D \mapsto (D,\gamma^\op),\]
and the $\mQSym$-valued maps send $G \mapsto \bX_G$,
$D \mapsto \bGamma(D)$,
and $(D,\gamma) \mapsto \bGamma(D,\gamma)$.
\end{theorem}

\begin{proof}
We observed in 
 Proposition~\ref{cG-lc-prop} that $(\mG,\wzeta_\cG)$ is a combinatorial LC-Hopf algebra
 and that the bottom left map sends $G \mapsto \bX_G$.
 
 For a proof that $ (\mP,\wzeta_\cP) $ is a combinatorial LC-Hopf algebra, see \cite[\S3.1]{LM2019}.
In \cite{LM2019}, the scalar ring for $\mP$ is $\ZZ[\beta]$ rather than a generic integral domain $\KK$,
and the definition of a labeled poset $(P,\gamma)$ assumes that $P$ is (the Hasse diagram of) a finite poset rather than any finite DAG,
but all arguments carry over unchanged to our slightly more general setting.
With the same caveats, the claim that the given map $ \mP \to \mQSym$ is the unique morphism $(\mP,\wzeta_\cP) \to (\mQSym,\wzetaq) $
is \cite[Thm.~3.8]{LM2019}.

The pair $(\mO,\wzeta_\cO)$ is a combinatorial LC-Hopf algebra as its structure maps are the unique ones compatible with the inclusion
$\mO \hookrightarrow \mP$ sending $D \mapsto (D,\gamma^\op)$.
The two right-most arrows  in the diagram are morphisms of combinatorial LC-Hopf algebras, and their composition gives the vertical arrow by definition.
%$(\mO,\wzeta_\cO)\to(\mQSym,\wzetaq)$ by definition.

It remains to show that linear map acting on graphs as $G\mapsto \sum_{D \in \MAO(G)} D$
is a morphism  $(\mG,\wzeta_\cG) \to (\mO,\wzeta_\cO) $. Denote this map by $\psi$. 
If a graph $G$ has at least one edge then so does every $D \in \MAO(G)$ and 
%\[\wzeta_\cO \circ \psi(G)  = 0 = \wzeta_\cG(G).\]
if $G$ has no edges then $\MAO(G) = \{G\}$.
%so \[\wzeta_\cO \circ \psi(G)  = t^{|G|} = \wzeta_\cG(G).\] 
In either case we have $\wzeta_\cO \circ \psi = \wzeta_\cG$.
Likewise, 
 $\psi$ commutes with the (co)unit maps of $\mG$ and $\mO$
since $G=\emptyset$ is empty if and only if $\MAO(G) = \{G\} = \{\emptyset\}$.
 
  Since $\MAO(G \sqcup H) = \{ D_1 \sqcup D_2 : (D_1, D_2) \in \MAO(G) \times \MAO(H)\}$ for two graphs $G$ and $H$, we also have $\psi \circ \nabla_{\mG} = \nabla_{\mO}\circ (\psi\htimes \psi)$. To check that $(\psi\htimes \psi) \circ \Delta_{\mG} = \Delta_{\mO}\circ \psi$, we must show for any graph $G$ that
  \be\label{psi-psi-eq}
  (\psi\htimes \psi) \circ \Delta_{\mG}(G) = \sum_{S\cup T = V(G)} \sum_{D_S \in \MAO(G|_S)} \sum_{D_T \in \MAO(G|_T)} D_S \otimes D_T
  \ee
  is equal to 
 \be\label{psi-psi-eq2}
 \Delta_{\mO}\circ \psi(G)  = \sum_{\tilde D \in \MAO(G)} \sum_{
 \tilde S \cup \tilde T = V(\tilde D)} \tilde D|_{\tilde S} \otimes \tilde D|_{\tilde T}
   \ee
    where in the last sum $\tilde S$ must be a lower set of $\tilde D$,
$   \tilde T$ must be an upper set, and $\tilde S \cap \tilde T$ must be an antichain.
For this, it suffices to construct a bijection from the set of tuples $(S,T,D_S,D_T)$
indexing the summands in the first expression
to the set of triples $(\tilde D, \tilde S ,\tilde T)$ indexing the summands in the second expression such that $D_S \cong \tilde D|_{\tilde S}$ and $D_T \cong \tilde D|_{\tilde T}$.
Such a bijection is constructed as follows.

Fix a graph $G$. Suppose $S$ and $T$ are (not necessarily disjoint)
 sets such that $S\cup T = V(G)$. Choose $D_S \in \MAO(G|_S)$ and $D_T \in \MAO(G|_T)$.
 
 Suppose $D_T$ has type $\alpha : T \to \PP$. Define $\alpha(v)=0$ for $v \in S\setminus T$.
 For each $v \in S$, let $d_v\in \NN$ be either
 $\alpha(v) -1$ if $v\in T$ and $(v,1)$ is a sink vertex in $D_S$ and $(v,\alpha(v))$ is a source vertex in $D_T$, or else $\alpha(v)$.
 Then 
set 
\[ \tilde S = \{ (v,i + d_v) : (v,i) \in V(D_S)\} \quand \tilde T = V(D_T).\]
Notice that $\tilde S \cap \tilde T$ is the set of pairs $(v,\alpha(v))$ for $v \in S\cap T$ with $d_v = \alpha(v)-1$,
so in $\tilde S \cup \tilde T$ the
 sink vertices in $D_S$  of the form $(v,1)$ are merged with the source vertices in $ D_T$ of the form $(v,\alpha(v))$.
 
Finally define $\tilde D$ to be the directed graph with vertex set $\tilde S \cup \tilde T$
 and edges 
 \be\label{edge-eqs}\ba 
  (u,i + d_u) &\to (v,j+d_v)&&\quad\text{if }(u,i) \to (v,j) \in E(D_S), \\
  (u,i) &\to (v,j)&&\quad\text{if }(u,i) \to (v,j) \in E(D_T), \\
   (u,i) &\to (v,j) &&\quad\text{if $(u,i) \in \tilde S\setminus \tilde T$, $(v,j) \in \tilde T$,}
   \\ &&&\quad\text{ and $u=v$ or $\{u,v\} \in E(G)$.}
   \ea
   \ee
   
   \begin{example}\label{STD-ex}
Let $G$ be the complete graph on two elements $v$ and $w$. 
Suppose $S=T=V(G)=\{v,w\}$.
Let $D_S \in \MAO(G|_S)$ be the leftmost directed graph in Figure~\ref{MAO-fig} and let $D_T\in \MAO(G|_T) $ be the rightmost directed graph in the same figure.
Writing $v_i$ and $w_i$ instead of $(v,i)$ and $(w,i)$ as in Figure~\ref{MAO-fig}, we have 
\[
\tilde S = \{ v_5,v_4,v_3, w_6,w_5,w_4\} 
\quand
\tilde T = \{v_3,v_2,v_1,w_4,w_3,w_3,w_1\}
  \]
  so $\tilde S \cap \tilde T = \{v_3, w_4\}$.
The directed graph $\tilde D$ on $\tilde S \cup \tilde T$ is shown in Figure~\ref{MAO-fig2}.
\end{example}

\begin{figure}[h]
\[\centerline{
    \begin{tikzpicture}[xscale=3, yscale=1,>=latex,inner sep=0.5mm]
      \node at (0,-1) (T4) {$\boxed{v_4}$};
      \node at (-.25,-2) (T5) {$\boxed{v_5}$};
      \node at (1,-0.5) (U5) {$\boxed{w_5}$};
      \node at (1.25,-1.5) (U6) {$\boxed{w_6}$};
      \node at (-0.25,2.5) (T1) {$v_1$};
      \node at (0,1.5) (T2) {$v_2$};
      \node at (-0.25,0.5) (T3) {$v_3$};
      \node at (1.5,3) (U1) {$w_1$};
      \node at (1,2) (U2) {$w_2$};
      \node at (1,1) (U3) {$w_3$};
      \node at (1.5,0) (U4) {$w_4$};
      \draw[->]  (T5) -- (T4);
      \draw[->]  (T4) -- (T3);
      \draw[->]  (U6) -- (U5);
      \draw[->]  (U5) -- (U4);
      \draw[->]  (T5) -- (U6);
      \draw[->]  (U6) -- (T4);
      \draw[->]  (T4) -- (U5);
      \draw[->]  (U5) -- (T3);
      \draw[->]  (T3) -- (U4);
      \draw[->]  (T5) -- (T3);
      \draw[->]  (U6) -- (U4);
      \draw[->]  (T3) -- (T2);
      \draw[->]  (T3) -- (T1);
      \draw[->]  (T2) -- (T1);
      \draw[->]  (U4) -- (U3);
      \draw[->]  (U4) -- (U2);
      \draw[->]  (U4) -- (U1);
      \draw[->]  (U3) -- (U2);
      \draw[->]  (U3) -- (U1);
      \draw[->]  (U2) -- (U1);
      \draw[->]  (U4) -- (T3);
      \draw[->]  (T3) -- (U3);
      \draw[->]  (U3) -- (T2);
      \draw[->]  (T2) -- (U2);
      \draw[->]  (U2) -- (T1);
      \draw[->]  (T1) -- (U1);
%
%\draw[->]  (T5) -- (T3);
% \draw[->]  (T5) -- (T2);
%  \draw[->]  (T5) -- (T1);
%   \draw[->]  (T5) -- (U4);
%    \draw[->]  (T5) -- (U3);
%     \draw[->]  (T5) -- (U2);
%      \draw[->]  (T5) -- (U1);
%\draw[->]  (T4) -- (T3);
% \draw[->]  (T4) -- (T2);
%  \draw[->]  (T4) -- (T1);
%   \draw[->]  (T4) -- (U4);
%    \draw[->]  (T4) -- (U3);
%     \draw[->]  (T4) -- (U2);
%      \draw[->]  (T4) -- (U1);
%\draw[->]  (U6) -- (T3);
% \draw[->]  (U6) -- (T2);
%  \draw[->]  (U6) -- (T1);
%   \draw[->]  (U6) -- (U4);
%    \draw[->]  (U6) -- (U3);
%     \draw[->]  (U6) -- (U2);
%      \draw[->]  (U6) -- (U1);
%\draw[->]  (U5) -- (T3);
% \draw[->]  (U5) -- (T2);
%  \draw[->]  (U5) -- (T1);
%   \draw[->]  (U5) -- (U4);
%    \draw[->]  (U5) -- (U3);
%     \draw[->]  (U5) -- (U2);
%      \draw[->]  (U5) -- (U1);
     \end{tikzpicture}
}\]
\caption{
The directed graph $\tilde D$ corresponding to Example~\ref{STD-ex}, with some edges omitted
and with vertices $(v,i)$ and $(w,j)$ written as $v_i$ and $w_j$.
The boxed vertices are the elements of  $\tilde S\setminus \tilde T$ while the unboxed vertices are the elements of $\tilde T$,
and $\tilde S \cap \tilde T = \{v_3,w_4\}$.
This picture does not show all of $\tilde D$, which has additional directed edges from each boxed vertex to each unboxed vertex.
}\label{MAO-fig2}
\end{figure}
   
 One can check as follows that  $\tilde S$ is a lower set in $\tilde D$ while $\tilde T$ is an upper set.
  First, notice
 that an edge of the first type in \eqref{edge-eqs} can only go from $\tilde T$ to $\tilde S$
 in the event that $(u,i)\to (v,j)$ is an edge in $D_S$ and $(u,j+d_u) \in \tilde S \cap \tilde T$. 
 However, this never occurs since
 we can only have $(u,j+d_u) \in \tilde S \cap \tilde T$
  when $j = 1$
 and $(u,1)$ is a sink vertex in $D_S$.

Similarly,
an edge of the second type  in \eqref{edge-eqs} can only go from $\tilde T$ to $\tilde S$
 in the event that $(u,i)\to (v,j)$ is an edge in $D_T$ and $(v,j) \in \tilde S \cap \tilde T$. 
 However, this never occurs since
 we can only have $(v,j) \in \tilde S \cap \tilde T$
  when $j = \alpha(v)$
 and $(v,\alpha(v))$ is a source vertex in $D_T$.
 
 As edges of the third type  in \eqref{edge-eqs} all start from $\tilde S\setminus \tilde T$,
  we conclude that 
 $\tilde S$ is a lower set in $\tilde D$ and $\tilde T$ is an upper set.
 
To finish our proof of Theorem~\ref{diagram2-thm}, it suffices to derive the following lemma.
 
\begin{lemma}\label{merge-lem}
The map 
\[(S,T,D_S,D_T) \mapsto (\tilde D, \tilde S, \tilde T) \]
is a bijection from the set of tuples $(S,T,D_S,D_T)$, in which 
$S$ and $T$ are sets with $S\cup T = V(G)$ and
 $D_S \in \MAO(G|_S)$ and $D_T \in \MAO(G|_T)$,
to the set of triples $(\tilde D, \tilde S, \tilde T)$, where $\tilde D \in \MAO(G)$
and 
where
$\tilde S$ and $\tilde T$ are respectively a lower set and an upper set in $\tilde D$
such that $\tilde S \cup \tilde T = V(\tilde D)$ and $\tilde S \cap \tilde T$
is an antichain.
This bijection satisfies $D_S \cong \tilde D|_{\tilde S}$ and $D_T \cong \tilde D|_{\tilde T}$.
\end{lemma}

\begin{proof}
We first check that $\tilde D \in \MAO(G)$. The directed graph $\tilde D$ is acyclic since $D_S$ and $D_T$ are acyclic.
   Fix $v \in V(G)$ and $i \in \PP$ such that $(v,i+1)$ and $(v,i)$ are both vertices in $\tilde D$.
   If these vertices are both in $\tilde S$ (respectively, $\tilde T$),
   then any two directed paths in $D_S$ (respectively, $D_T$)  from $(v,i+1)$ and $(v,i)$
   lift to paths in $\tilde D$.

Suppose  $(v,i+1) \in \tilde S$ and $(v,i) \in \tilde T$.
Then $i = \alpha(v)$ and either $(v,i+1)$ is not a sink in $\tilde D|_{\tilde S} \cong D_S$ or $(v,i)$ is not a source
in $\tilde D|_{\tilde T} = D_T$.
In the first case $\tilde D$ must have an edge of the form $(v,i+1) \to (w,j) \in \tilde T$ where $v\neq w$,
and in the second $\tilde D$ must have an edge of the form $(v,i) \leftarrow (w,j) \in \tilde S$ where $v\neq w$. Either way, we  
have $\{v,w\} \in E(G)$, so $(v,i+1)\to(w,j)\to(v,i)$ is a second directed path in $\tilde D$
besides $(v,i+1)\to (v,i)$. We conclude that $\tilde D \in \MAO(G)$.

We have already seen that $\tilde S$ is a lower set in $\tilde D$ and $\tilde T$ is an upper set
with $\tilde S \cup \tilde T = V(\tilde D)$.
No single edge of $\tilde D$ may connect two vertices in $\tilde S \cap \tilde T$,
since then the vertices could not both be sources in $D_T$. This implies that no directed path in $\tilde D$ may connect two vertices in $\tilde S \cap \tilde T$, since 
 every vertex in this path would have to be in $\tilde S\cap \tilde T$ since 
 $\tilde S$ is a lower set in $\tilde D$ and $\tilde T$ is an upper set.
 Thus $\tilde S \cap \tilde T$ is an antichain in $\tilde D$.
 
 We conclude that every image of the given operation is in the described codomain.
 This operation is a bijection as it is straightforward to construct its inverse.
 Namely, given $ (\tilde D, \tilde S, \tilde T) $
 we first recover $D_T$ and   $\alpha$    as 
 \[ D_T = \tilde D|_{\tilde T}\quand \alpha: v\mapsto \max\{i : (v,i) \in \tilde T\}.\]
 Then we
 recover  $D_S$ from the (isomorphic) induced subgraph $ \tilde D|_{\tilde S}$ by replacing each vertex $(v,i)$ by $(v,i+1-\alpha(v))$,
 and finally we obtain $S$ and $T$  as 
 \[ S = \{ v: (v,i) \in V(D_S)\} \quand T = \{ v: (v,i) \in V(D_T)\}.\]
 We have $S\cup T= \{ v : (v,i) \in V(\tilde D)\} = V(G)$.
 Since 
 $D_S$ and $D_T$ are formed by restricting $\tilde D \in \MAO(G)$ to a lower set or an upper set,
it holds that $D_S \in \MAO(G|_S)$ and $D_T \in \MAO(G|_T)$ as needed.
\end{proof}

The proof of Lemma~\ref{merge-lem} completes the proof of Theorem~\ref{diagram2-thm}.
\end{proof}

Outside the special case when $\omega : V(G) \to \{1\}$,
the kromatic symmetric function $\bX_{(G,\omega)}$ of a weighted graph is not always multifundamental positive,
nor is its expansion into $\bL_\alpha$'s finite.
However, 
since the diagram in Theorem~\ref{diagram2-thm} commutes, \eqref{ML-eq} implies that:

\begin{corollary}  \label{mf-p-cor}
%The unique morphism 
%$(\mG,\wzeta_\cG)  \to (\mQSym,\wzetaq)$
%assigns a graph $G$ to its kromatic symmetric function,  
%which can be expressed as 
If $G$ is a graph then $\bX_G= \sum_{D \in \MAO(G)}   \bGamma(D)$ and consequently this 
symmetric power series is multifundamental positive.
\end{corollary}

Theorem~\ref{diagram2-thm} also leads to a coproduct formula for $\bare_n$ from Example~\ref{bare-ex}.
 
 \begin{corollary}
 If $n$ is a positive integer then 
 \[\Delta(\bare_n) = \sum_{\substack{ i,j \in \NN \\ i+j = n}} \bare_i \otimes \bare_j  + \sum_{\substack{ i,j \in \PP \\ i+j = n+1}} \bare_i \otimes \bare_j.\]
 \end{corollary}
 
 \begin{proof}
 Let $D = (1 \to 2 \to 3 \to \dots \to n)$ be an $n$-element path.
 By Theorem~\ref{diagram2-thm} 
 we have $\Delta(\bare_n) = \Delta(\bGamma(D)) = \sum_{S\cup T= V(D)} \bGamma(D|_S) \otimes \bGamma(D|_T)$
 where the sum is over all lower sets $S$ and upper sets $T$ with $S\cap T$ an antichain in $D$.
 To be an antichain, the intersection $S\cap T$ must either be empty or the singleton set $\{i\}$ for some $i \in [n]$.

Either way, the choices for $S$ and $T$ are limited.
If $S\cap T$ is empty then we must have
 $S = [i]$ and $T =[n]\setminus [i]$ for some $0 \leq i \leq n$, and then $\bGamma(D|_S)\otimes \bGamma(D|_T)  = \bare_i\otimes \bare_{n-i}$.
If $S\cap T = \{i\}$ then we must have $S = [i]$ and $T = [n]\setminus[i-1]$,
 and then $\bGamma(D|_S) \otimes \bGamma(D|_T)  = \bare_i \otimes \bare_{n+1-i}$.
 Adding up these terms gives the result.
 \end{proof}

Fix a directed acyclic graph $D$. When $\alpha : V(D) \to \PP$ is any map,
define $\oCl_\alpha(D)$ to be the
directed acyclic graph with vertices $\cCl_\alpha(V(D))$ and  edges 
\[(v,i)\to (w,j)\quad\text{whenever}\quad\text{$v\to w \in E(D)$ or 
 both $v=w$ and $i> j$.}\]
When $\gamma : V(D) \to \ZZ$ is injective, let $\tilde \gamma$ be any labeling of $\oCl_\alpha(D)$ with
 \[
 \tilde \gamma(v, i)  < \tilde \gamma(w, j)\quad\text{if and only if}\quad\text{$\gamma(v) < \gamma(w)$ or 
both $v=w$ and $i<j$},
\]
and then
define $\oCl_\alpha(D,\gamma) = (\oCl_\alpha(D), \tilde \gamma)$.

Recall that for each   $\cH \in \{ \cG,\cO,\cP\}$
we use the symbols $\fkm \cH$ and $\overline{\cH}$ to denote the same sets with different LC-Hopf algebra structures.
Let 
\[
\ba 
\Psi_{\cG} &: \mG \to \wG,
\\
\Psi_{\cO} &: \mO \to \wO,\text{ and }
\\
\Psi_{\cP} &: \mP \to \wP
\ea
\]
be the continuous linear maps sending 
\be
G\mapsto \sum_{\alpha} \tfrac{1}{\alpha!} \Cl_\alpha(G),
\quad
D\mapsto \sum_{\alpha} \oCl_\alpha(D),
\quad
(D,\gamma) \mapsto \sum_{\alpha }   \oCl_\alpha(D,\gamma)
\ee
for each graph $G$, directed acyclic graph $D$, and labeled poset $(D,\gamma)$,
where the sums are over all maps $\alpha : V \to \PP$ for $V=V(G)$ or $V(D)$ as appropriate.
The map $\Psi_{\cG}$ is only defined if $\QQ \subseteq \KK$, which we assume in the following theorem.

\begin{theorem} 
When defined, the maps $\Psi_{\cG}$, $\Psi_{\cO}$, and $\Psi_{\cP}$ are isomorphisms 
of LC-Hopf algebras, and 
there is a commutative diagram %
{\[
\begin{tikzcd}
(\mG,\wzeta_\cG) \arrow[r,hook] \arrow[d,"\Psi_\cG"] & (\mO,\wzeta_\cO) \arrow[r,hook] \arrow[d,"\Psi_\cO"]  & (\mP,\wzeta_\cP) \arrow[d,"\Psi_\cP"] \\
 (\wG,\wzeta_\cG) \arrow[r,hook] & (\wO,\wzeta_\cO)  \arrow[r,hook]&  (\wP,\wzeta_\cP)
  \end{tikzcd}
\]}%
of combinatorial LC-Hopf algebra morphisms, in which the 
 horizontal arrows extend the maps in the diagrams in Proposition~\ref{diagram1-prop} and Theorem~\ref{diagram2-thm}.
 \end{theorem}

\begin{proof}
The vertical maps are invertible since each has a unitriangular matrix in the standard pseudobases of the source and target
ordered by number of vertices. It is straightforward to see that 
$
\wzeta_\cG = \Psi_\cG\circ \wzeta_\cG$
and
$
\wzeta_\cO = \Psi_\cO\circ \wzeta_\cO
$
and
$
\wzeta_\cP = \Psi_\cP\circ \wzeta_\cP
$, and that each of the vertical maps is compatible with the relevant
product, unit, and counit. It is also easy to see that the right square in the given diagram commutes. 

The nontrivial things to check are that $\Psi_\cG$ and $\Psi_\cP$ 
are coalgebra morphisms, and that the left square in the given diagram commutes.
This amounts to verifying three identities:
\ben
\item[(a)] If $G$ is a graph then the sum
\[ 
\sum_{S\cup T = V(G)} \sum_{\substack{\alpha' : S \to \PP \\ \alpha'' : T \to \PP}} \tfrac{1}{\alpha'!\alpha''!} 
\Cl_{\alpha'}(G|_S)\otimes \Cl_{\alpha'}(G|_T)
\] is equal to 
\[
\sum_{\alpha: V(G) \to \PP} \sum_{S\sqcup T = V(\Cl_\alpha(G))} \tfrac{1}{\alpha!} \Cl_\alpha(G)|_S \otimes \Cl_\alpha(G)|_T.\]

\item[(b)] If $(D,\gamma)$ is a labeled poset then the sum
\[
\sum_{
\substack{S\cup T = V(D)\\
S\text{ lower set} \\
T\text{ upper set} \\
S\cap T\text{ antichain}}} \sum_{\substack{\alpha' : S \to \PP \\ \alpha'':T \to \PP}}  \oCl_{\alpha'}(D|_S,\gamma|_S) \otimes  \oCl_{\alpha''}(D|_T,\gamma|_T)
\]
is equal to \[
\sum_{\alpha: V(D) \to \PP} \sum_{\substack{S\sqcup T = V(\oCl_\alpha(D)) \\ S\text{ lower set} \\ T \text{ upper set}}}  (\oCl_\alpha(D)|_S, \tilde \gamma|_S) \otimes (\oCl_\alpha(D)|_T, \tilde\gamma|_T).\]

\item[(c)] If $G$ is a graph then 
\[ \sum_{D \in \MAO(G)} \sum_{\alpha : V(D) \to \PP} \oCl_\alpha(D)
= \sum_{\alpha : V(G) \to \PP} \sum_{D \in \MAO(\Cl_\alpha(G))} \tfrac{1}{\alpha!} \Cl_\alpha(D).\] 
\een
Each of these identities follows as a straightforward, if slightly cumbersome, exercises in algebra. We omit the details.
\end{proof}

The preceding result reduces the natural question
of finding antipode formulas for $\mG$, $\mO$, and $\mP$ (see, e.g.,
\cite[Problem 3.5]{LM2019})
to the more familiar graded Hopf algebras $\cG$, $\cO$, and $\cP$,
where some formulas are already known  \cite{HumMar}.

\section{Kromatic quasisymmetric functions}
\label{qana-sect1}

In this section we study two quasisymmetric refinements of $\bX_G$.
Our goal in these constructions to find interesting common generalizations
of \cite{CPS} and \cite{SW2016}.

\subsection{Hopf algebras for chromatic quasisymmetric functions}\label{qana-subsect1}

Let $G$ be an \defn{ordered graph}, that is, a graph with a total order $<$ on its vertex set $V(G)$.
One can think of the ordering on $V(G)$ as defining an acyclic orientation on the edges of $G$,
and 
we do not distinguish between $G$ and another ordered graph $H$ if the corresponding directed acyclic graphs are isomorphic.
%An \defn{ascent} (respectively, \defn{descent}) of a map $\kappa : V(G) \to \PP$ is an edge $\{u,v\} \in E(G)$ with $u<v$ 
%and  $\kappa(u) <\kappa(v)$
%(respectively, $\kappa(u)>\kappa(v)$). Let $\asc_{G}(\kappa)$
%and
%$\des_{G}(\kappa)$ denote the number of ascents and descents.

Fix an invertible element $q \in \KK$.
If $\oG_n$ is the free $\KK$-module spanned by all (isomorphism classes of)
ordered graphs with $n$ vertices, then $\oG := \bigoplus_{n \in \NN} \oG_n$
has a graded connected Hopf algebra structure in
which the product is disjoint union and the coproduct is the $\KK$-linear map $\Delta_q$ satisfying
\be\label{bDq-eq}
\textstyle
\Delta_q(G)= \sum_{S\sqcup T = V(G)}  q^{\asc_{G}(S,T)}G|_S \otimes G|_T \in \oG\htimes \oG
\ee
for each ordered graph $G$, where
\[
\asc_{G}(S,T):=|\{(s,t) \in S\times T  : \{s,t\} \in E(G)\text{ and }s<t\}|.
\]
 This structure becomes a combinatorial Hopf algebra
when paired with the algebra morphism  $\ozeta : \oG \to \KK$ sending $G \mapsto 0^{|E(G)|}$. 
It follows from \cite{ABS}
that the unique 
morphism $(\oG,\ozeta) \to (\QSym,\zetaq)$ sends an ordered graph $G$ to the 
chromatic quasisymmetric function $X_G(q)$ of Shareshian and Wachs from Definition~\ref{cqf-def}.

%Although the pair $(\oG,\ozeta)$ is a natural $q$-analogue of $(\cG,\zeta_\cG)$,
We have not been able to find a natural $K$-theoretic  
 generalization of the coproduct $\Delta_q$
that leads to a $q$-analogue of the LC-Hopf algebra $\mG$.
The simplest idea would be to replace the disjoint union
$\sqcup$ in \eqref{bDq-eq} by arbitrary union $\cup$, but this 
does not lead to a co-associative map on the completion of $\oG$. This problem remains if we change the $q$-power exponent $\asc_G(S,T)$
to other forms that are equivalent in the homogeneous case, like $\asc_G(S- T, T)$, $\asc_G(S, T- S)$, or $\asc_G(S- T, T- S)$.

\subsection{A quasisymmetric kromatic function}

Nevertheless, there do appear to be interesting  ways
of defining a simultaneous $K$-theoretic generalization of $X_G(q)$ and $q$-analogue of $\bX_G$. We examine two such constructions, the first of which is given as follows.

\begin{remark}
For the rest of this section we assume $\KK \supseteq \ZZ$ and let $q$ be a formal parameter.
We will consider the polynomial and power series rings \[\Sym[q] \subset \mQSym[q]\subset \mQSym\llbracket q\rrbracket.\]
\end{remark}

Recall for $\kappa : V(G) \to \PP$ that $\asc_G(\kappa)$
and $\des_G(\kappa)$ respectively count the edges $\{u,v\} \in E(G)$ with $u<v$ 
and $\kappa(u) <\kappa(v)$ or $\kappa(u)>\kappa(v)$.

\begin{definition}\label{or-def}
For an ordered graph  $G$ define
\[
  \qX_G(q) = \sum_{\kappa}  q^{\asc_{G}(\max\circ \kappa)} x^\kappa  \in \mQSym[q]
 \]
where the sum is over all proper set-valued colorings  $\kappa : V(G) \to \Set(\PP)$.
\end{definition}

One recovers $X_G(q)$
 from $\qX_G(q)$ as its homogeneous component of lowest degree in the $x_i$ variables, setting $\deg x_i = 1$ and $\deg q = 0$.
 Despite the discussion in Section~\ref{qana-subsect1}, it is clear that if $G=G_1\sqcup G_2$ is the disjoint union of two ordered graphs then 
$
 \qX_G(q) = \qX_{G_1}(q) \qX_{G_2}(q).
$
In addition to the following example, see
Table~\ref{qX-tab} for some computations of $\qX_G(q)$.

\begin{example}\label{or-k-ex} Suppose $G=K_n$ is the complete graph on the vertex set $[n]$.
There are $ {r \brace n} $ ways of partitioning $r$ colors into $n$ nonempty subsets. For each way of assigning these subsets 
to the vertices of $G$ to form a proper set-valued coloring, there is a unique permutation $\sigma \in S_n$ whose letters appear in the same relative order as $\max(\kappa(n))$, $\max(\kappa(n-1))$, \dots, $\max(\kappa(1))$.
The inversion set of $\sigma$ coincides with $\asc( \max\circ \kappa)$, and it is well-known (see \cite[(A.2)]{SW2016}) that 
$\sum_{\sigma \in S_n} q^{\inv(\sigma)} = [n]_q!$ as defined Example~\ref{nq-ex}.
We conclude  that
\[\textstyle
 \qX_{K_n}(q) = [n]_q! \sum_{r= n}^\infty {r \brace n}  e_r =  \frac{[n]_q!}{n!} \bX_{K_n}=  [n]_q! \sum_{r= n}^\infty {r -1\brace n-1}  \bare_r\]
%where ${r \brace n}$ again denotes the Stirling numbers of the second kind. 
using Proposition~\ref{bXk-prop} for the last equality.
% since $\qX_{K_n}(q) = \frac{[n]_q!}{n!} \bX_{K_n}$.
\end{example}

%d=5; n=6; vv=[1,2,3]; ee=[(1,2),(2,3)]; print('\n\n\n'); e=Quasisymmetric.monomial_expansion(max_oriented_kromatic_polynomial(n, vv, ee)); v=Vector({a: b for (a,b) in e.items() if sum(a) < d}).set_variable(0,1); v.sorter = lambda a: (sum(a),sorted(a),a); v.printer = lambda a: 'M_{%s}' % ''.join(map(str,a)); print(str(v).replace('y', 'q').replace('*', ' '))

%d=5; n=6; vv=[1,2,3]; ee=[(1,2),(2,3)]; print('\n\n\n'); e=Quasisymmetric.multifundamental_expansion(max_oriented_kromatic_polynomial(n, vv, ee)).set_beta(1); v=Vector({a: b for (a,b) in e.items() if sum(a) < d}).set_variable(0,1); v.sorter = lambda a: (sum(a),sorted(a),a); v.printer = lambda a: '\\bL_{%s}' % ''.join(map(str,a)); print(str(v).replace('y', 'q').replace('*', ' '))

 \begin{table}[htb]
\centerline{
\scalebox{1}{
\begin{tabular}{|l|l|l|}
\hline & & 
\\[-8pt]
Graph $G$   & $\qX_G(q)$ in the monomial basis $\{M_\alpha\}$ & $\qX_G(q)$ in the multifundamental basis $\{\bL_\alpha\}$
\\[-8pt] & & 
\\ \hline & & 
\\[-8pt]  
  \begin{tikzpicture}[scale=1,baseline=(void.base)]%3-path
  \node at (0, 0)(1){2};
    \node at (1, 0)(2){3};
    \node at (0.5, 0.866)(3){1};
    \node at (0.5,0.433)(void) {};
    \draw[black, thick] (3) -- (1);
    \draw[black, thick] (2) -- (1);
  \end{tikzpicture} 
  &  \pbox[c]{7cm}{\scriptsize$ (1 + 4q + q^2) M_{111} + q M_{12} + q M_{21} + (6 + 24q + 6q^2) M_{1111} + 5q M_{112} + 5q M_{121} + (1 + 3q + q^2) M_{211}+ \dots$} 
  & \pbox[c]{7cm}{\scriptsize $ (1 + 2q + q^2) \bL_{111} + q \bL_{12} + q \bL_{21} + (2 + 5q + 2q^2) \bL_{1111} + 3q \bL_{112} + 3q \bL_{121} + (1 + q + q^2) \bL_{211} + \dots $
}
\\ [-4pt] && \\
  \begin{tikzpicture}[scale=1,baseline=(void.base)]%3-path
  \node at (0, 0)(1){3};
    \node at (1, 0)(2){2};
    \node at (0.5, 0.866)(3){1};
    \node at (0.5,0.433)(void) {};
    \draw[black, thick] (3) -- (1);
    \draw[black, thick] (2) -- (1);
  \end{tikzpicture} 
  &  \pbox[c]{7cm}{\scriptsize$ (2 + 2 q + 2 q^2) M_{111} + M_{12} +  q^2 M_{21} + (12 + 12 q + 12 q^2) M_{1111} + 5 M_{112} + (2 + 3 q^2) M_{121} + (2 q + 3 q^2) M_{211}+ \dots$} 
  &  \pbox[c]{7cm}{\scriptsize$ (1 + 2q + q^2) \bL_{111} + \bL_{12} + q^2 \bL_{21} + (2 + 4q + 3q^2) \bL_{1111} + 3 \bL_{112} + (1 + 2q^2) \bL_{121} + (2q + q^2) \bL_{211}+\dots$}
\\ [-4pt] && \\
  \begin{tikzpicture}[scale=1,baseline=(void.base)]%3-path
  \node at (0, 0)(1){1};
    \node at (1, 0)(2){3};
    \node at (0.5, 0.866)(3){2};
    \node at (0.5,0.433)(void) {};
    \draw[black, thick] (3) -- (1);
    \draw[black, thick] (2) -- (1);
  \end{tikzpicture} 
  &  \pbox[c]{7cm}{\scriptsize$(2 + 2 q + 2 q^2) M_{111} +  q^2 M_{12} + M_{21} + (12 + 12 q + 12 q^2) M_{1111} + 5 q^2 M_{112} + (3 + 2 q^2) M_{121} + (3 + 2 q) M_{211}+ \dots$} 
  & \pbox[c]{7cm}{\scriptsize $ (1 + 2q + q^2) \bL_{111} + q^2 \bL_{12} + \bL_{21} + (3 + 4q + 2q^2) \bL_{1111} + 3q^2 \bL_{112} + (2 + q^2) \bL_{121} + (1 + 2q) \bL_{211} + \dots $}
 \\[-8pt] & & 
 \\ \hline & & 
\\[-8pt]  
    \begin{tikzpicture}[scale=1,baseline=(void.base)]%3-cycle
  \node at (0, 0)(1){2};
    \node at (1, 0)(2){3};
    \node at (0.5, 0.866)(3){1};
    \node at (0.5,.433)(void) {};
    \draw[black, thick] (3) -- (1);
    \draw[black, thick] (2) -- (1);
    \draw[black, thick] (2) -- (3);
  \end{tikzpicture} & 
    \pbox[c]{7cm}{\scriptsize $ (1 + 2 q + 2 q^2 +  q^3) M_{111} + (6 + 12 q + 12 q^2 + 6 q^3) M_{1111}+ \dots $ }  &
  \pbox[c]{7cm}{\scriptsize$ (1 + 2q + 2q^2 + q^3) \bL_{111} + (3 + 6q + 6q^2 + 3q^3) \bL_{1111} + \dots $}
   \\[-8pt] & & 
    \\ \hline & & 
\\[-8pt]  
   \begin{tikzpicture}[scale=1,baseline=(void.base)]%4-cycle
  \node at (0, 0)(1){1};
    \node at (1, 0)(2){2};
    \node at (1, 1)(3){3};
     \node at (0, 1)(4){4};
         \node at (0.5,.5)(void) {};
    \draw[black, thick] (3) -- (2);
    \draw[black, thick] (2) -- (1);
    \draw[black, thick] (4) -- (3);
    \draw[black, thick] (1) -- (4); 
  \end{tikzpicture} & 
    \pbox[c]{7cm}{\scriptsize $(1 + 8q + 6q^2 + 8q^3 + q^4) M_{1111} + (2q + 2q^3) M_{112} + (2q + 2q^3) M_{121} + (2q + 2q^3) M_{211} + (q + q^3) M_{22} + (10 + 80q + 60q^2 + 80q^3 + 10q^4) M_{11111} + (18q + 18q^3) M_{1112} + (18q + 18q^3) M_{1121} + (1 + 14q + 6q^2 + 14q^3 + q^4) M_{1211} + (1 + 14q + 6q^2 + 14q^3 + q^4) M_{2111} + (4q + 4q^3) M_{122} + (4q + 4q^3) M_{212} + (4q + 4q^3) M_{221} + \dots $ }  & 
   \pbox[c]{7cm}{\scriptsize $(1 + 3q + 6q^2 + 3q^3 + q^4) \bL_{1111} + (q + q^3) \bL_{112} + (2q + 2q^3) \bL_{121} + (q + q^3) \bL_{211} + (q + q^3) \bL_{22} + (4 + 16q + 24q^2 + 16q^3 + 4q^4) \bL_{11111} + (7q + 7q^3) \bL_{1112} + (9q + 9q^3) \bL_{1121} + (1 + 5q + 6q^2 + 5q^3 + q^4) \bL_{1211} + (1 + 3q + 6q^2 + 3q^3 + q^4) \bL_{2111} + (3q + 3q^3) \bL_{122} + (2q + 2q^3) \bL_{212} + (3q + 3q^3) \bL_{221} + \dots$}
\\ [-4pt] && \\
      \begin{tikzpicture}[scale=1,baseline=(void.base)]%4-cycle
  \node at (0, 0)(1){1};
    \node at (1, 0)(2){2};
    \node at (1, 1)(3){4};
     \node at (0, 1)(4){3};
         \node at (0.5,.5)(void) {};
    \draw[black, thick] (3) -- (2);
    \draw[black, thick] (2) -- (1);
    \draw[black, thick] (4) -- (3);
    \draw[black, thick] (1) -- (4); 
  \end{tikzpicture} & 
    \pbox[c]{7cm}{\scriptsize $(2 + 4q + 12q^2 + 4q^3 + 2q^4) M_{1111} + 4q^2 M_{112} + (1 + 2q^2 + q^4) M_{121} + 4q^2 M_{211} + 2q^2 M_{22} + (20 + 40q + 120q^2 + 40q^3 + 20q^4) M_{11111} + 36q^2 M_{1112} + (7 + 22q^2 + 7q^4) M_{1121} + (3 + 4q + 22q^2 + 4q^3 + 3q^4) M_{1211} + (2 + 4q + 24q^2 + 4q^3 + 2q^4) M_{2111} + 8q^2 M_{122} + 8q^2 M_{212} + (1 + 6q^2 + q^4) M_{221} + \dots $ }  & 
   \pbox[c]{7cm}{\scriptsize $(1 + 4q + 4q^2 + 4q^3 + q^4) \bL_{1111} + 2q^2 \bL_{112} + (1 + 2q^2 + q^4) \bL_{121} + 2q^2 \bL_{211} + 2q^2 \bL_{22} + (5 + 16q + 22q^2 + 16q^3 + 5q^4) \bL_{11111} + 14q^2 \bL_{1112} + (4 + 10q^2 + 4q^4) \bL_{1121} + (1 + 4q + 8q^2 + 4q^3 + q^4) \bL_{1211} + (1 + 4q + 4q^2 + 4q^3 + q^4) \bL_{2111} + 6q^2 \bL_{122} + 4q^2 \bL_{212} + (1 + 4q^2 + q^4) \bL_{221} + \dots$}
\\ [-4pt] && \\
      \begin{tikzpicture}[scale=1,baseline=(void.base)]%4-cycle
  \node at (0, 0)(1){1};
    \node at (1, 0)(2){3};
    \node at (1, 1)(3){2};
     \node at (0, 1)(4){4};
         \node at (0.5,.5)(void) {};
    \draw[black, thick] (3) -- (2);
    \draw[black, thick] (2) -- (1);
    \draw[black, thick] (4) -- (3);
    \draw[black, thick] (1) -- (4); 
  \end{tikzpicture} & 
    \pbox[c]{7cm}{\scriptsize $(4 + 4q + 8q^2 + 4q^3 + 4q^4) M_{1111} + (2 + 2q^4) M_{112} + 4q^2 M_{121} + (2 + 2q^4) M_{211} + (1 + q^4) M_{22} + (40 + 40q + 80q^2 + 40q^3 + 40q^4) M_{11111} + (18 + 18q^4) M_{1112} + (4 + 28q^2 + 4q^4) M_{1121} + (8 + 4q + 12q^2 + 4q^3 + 8q^4) M_{1211} + (10 + 4q + 8q^2 + 4q^3 + 10q^4) M_{2111} + (4 + 4q^4) M_{122} + (4 + 4q^4) M_{212} + (2 + 4q^2 + 2q^4) M_{221} + \dots $ }  & 
   \pbox[c]{7cm}{\scriptsize $(1 + 4q + 4q^2 + 4q^3 + q^4) \bL_{1111} + (1 + q^4) \bL_{112} + 4q^2 \bL_{121} + (1 + q^4) \bL_{211} + (1 + q^4) \bL_{22} + (6 + 16q + 20q^2 + 16q^3 + 6q^4) \bL_{11111} + (7 + 7q^4) \bL_{1112} + (1 + 16q^2 + q^4) \bL_{1121} + (3 + 4q + 4q^2 + 4q^3 + 3q^4) \bL_{1211} + (1 + 4q + 4q^2 + 4q^3 + q^4) \bL_{2111} + (3 + 3q^4) \bL_{122} + (2 + 2q^4) \bL_{212} + (1 + 4q^2 + q^4) \bL_{221} + \dots$}
   \\[-8pt] & & 
   \\ \hline & & 
\\[-8pt]  
\begin{tikzpicture}[scale=1,baseline=(void.base)]%K_4 minus an edge
  \node at (0, 0)(1){1};
    \node at (1, 0)(2){2};
    \node at (1, 1)(3){3};
     \node at (0, 1)(4){4};
    \node at (0.5,0.5)(void){};
    \draw[black, thick] (3) -- (1);
    \draw[black, thick] (2) -- (1);
    \draw[black, thick] (4) -- (1);
    \draw[black, thick] (4) -- (3);
    \draw[black, thick] (2) -- (4); 
  \end{tikzpicture} & 
  \pbox[c]{7cm}{\scriptsize $(2 + 4q + 6q^2 + 6q^3 + 4q^4 + 2q^5) M_{1111} + (q^2 + q^3) M_{112} + (1 + q^5) M_{121} + (q^2 + q^3) M_{211} + (20 + 40q + 60q^2 + 60q^3 + 40q^4 + 20q^5) M_{11111} + (9q^2 + 9q^3) M_{1112} + (7 + 2q^2 + 2q^3 + 7q^5) M_{1121} + (3 + 2q + 4q^2 + 4q^3 + 2q^4 + 3q^5) M_{1211} + (2q + 7q^2 + 7q^3 + 2q^4) M_{2111}+\dots$  } & 
  \pbox[c]{7cm}{\scriptsize $(1 + 4q + 4q^2 + 4q^3 + 4q^4 + q^5) \bL_{1111} + (q^2 + q^3) \bL_{112} + (1 + q^5) \bL_{121} + (q^2 + q^3) \bL_{211} + (6 + 20q + 22q^2 + 22q^3 + 20q^4 + 6q^5) \bL_{11111} + (6q^2 + 6q^3) \bL_{1112} + (5 + q^2 + q^3 + 5q^5) \bL_{1121} + (1 + 2q + 3q^2 + 3q^3 + 2q^4 + q^5) \bL_{1211} + (2q + 4q^2 + 4q^3 + 2q^4) \bL_{2111}+\dots$} 
   \\[-8pt] & & 
   \\ \hline & & 
\\[-8pt]  
  \begin{tikzpicture}[scale=1,baseline=(void.base)]%claw
  \node at (0, 0)(1){1};
    \node at (1, 0)(2){2};
    \node at (1, 1)(3){3};
     \node at (0, 1)(4){4};
    \node at (0.5,0.5)(void){};
    \draw[black, thick] (3) -- (1);
    \draw[black, thick] (2) -- (1);
    \draw[black, thick] (1) -- (4);
  \end{tikzpicture} 
 & \pbox[c]{7cm}{\scriptsize $(6 + 6q + 6q^2 + 6q^3) M_{1111} + (3q^2 + 3q^3) M_{112} + (3 + 3q^3) M_{121} + (3 + 3q) M_{211} + q^3 M_{13} + M_{31} + (60 + 60q + 60q^2 + 60q^3) M_{11111} + (27q^2 + 27q^3) M_{1112} + (21 + 6q^2 + 27q^3) M_{1121} + (21 + 18q + 15q^3) M_{1211} + (21 + 21q + 12q^2) M_{2111} + 7q^3 M_{113} + (4 + 3q^3) M_{131} + (4 + 3q) M_{311} + 6q^3 M_{122} + 6q^2 M_{212} + 6 M_{221}+\dots$ }
 &  \pbox[c]{7cm}{\scriptsize $(1 + 3q + 3q^2 + q^3) \bL_{1111} + (3q^2 + 2q^3) \bL_{112} + (2 + 2q^3) \bL_{121} + (2 + 3q) \bL_{211} + q^3 \bL_{13} + \bL_{31} + (7 + 12q + 9q^2 + 3q^3) \bL_{11111} + (12q^2 + 8q^3) \bL_{1112} + (7 + 3q^2 + 11q^3) \bL_{1121} + (7 + 12q + 2q^3) \bL_{1211} + (5 + 9q + 6q^2) \bL_{2111} + 5q^3 \bL_{113} + (3 + 2q^3) \bL_{131} + (2 + 3q) \bL_{311} + 5q^3 \bL_{122} + 6q^2 \bL_{212} + 5 \bL_{221}+\dots$} 
  \\[-8pt] & & 
  \\ \hline
\end{tabular}}
}
\caption{Partial computations of $\qX_G(q)$ for some small graphs $G$. The ellipses ``$\dots$'' mask higher order terms  indexed by compositions $\alpha$ with  $|\alpha|>|V(G)|+1$.}\label{qX-tab}
\end{table}
 
Let us clarify the seeming asymmetry in Definition~\ref{or-def}.
 Form $\qX^{\des,\min}_G(q)$  by replacing ``$\asc$'' by ``$\des$''
 and ``$\max$'' by ``$\min$'' in Definition~\ref{or-def}.
 Construct 
 \[\qX^{\asc,\min}_G(q), \quad\qX^{\asc,\max}_G(q),\quand \qX^{\des,\max}_G(q)\] likewise,
 so that we have $\qX_G(q)=\qX^{\asc,\max}_G(q)$.
 Now let $\rho$ be the continuous linear map $\mQSym \to \mQSym$ sending 
 \be \rho : M_{(\alpha_1,\alpha_2,\dots,\alpha_k)} \mapsto M_{(\alpha_k,\dots,\alpha_2,\alpha_1)}.\ee
 This map extends to a ring involution of $\mQSym[q]$ fixing $q$.
 Finally let $\tau $ be the $\mQSym$-linear map  $\mQSym[q] \to \mQSym[q]$ acting on nonzero elements as  
 \be \tau : f(q) \mapsto q^{\deg_q (f)}  f(q^{-1}).\ee
 In other words, $\tau$ reverses the order of the coefficients of $f \neq 0$, viewed as a polynomial in $q$ 
 with coefficients in $\mQSym$.
 
\begin{proposition}
If $G$ is an ordered graph then
%One has
 \[ \qX_G(q) 
= \rho\(\qX^{\des,\min}_G(q)\)
= \tau\(\qX^{\des,\max}_G(q)\)
= \rho\circ \tau\(\qX^{\asc,\min}_G(q)\).
\]
\end{proposition}

\begin{proof}
Fix an ordered graph $G$.
Given any map $\kappa : V(G) \to \Set(\PP)$,
let \[\textstyle S := \bigcup_{v \in V(G)} \kappa(v) = \{i_0 <i_1 < \dots < i_m\}\]
and let $\sigma : S \to S$ be the permutation sending $i_j \mapsto i_{m-j}$ and  
then define $\kappa^\ast : V(G) \to \Set(\PP)$ by $\kappa^\ast(v) =  \{\sigma(i) : i \in \kappa(v)\}$. 
The operation $\kappa\mapsto \kappa^\ast$ is an involution of the set of proper set-valued colorings of $G$
with \[\asc_G(\max \circ \kappa) = \des_G(\min\circ\kappa^\ast),\]
and if $x^\kappa = x_{i_0}^{a_0}x_{i_1}^{a_1}\cdots x_{i_m}^{a_m}$ then $x^{\kappa^\ast} = \rho(x^\kappa) := x_{i_0}^{a_m} x_{i_1} ^{a_{m-1}}\cdots x_{i_m}^{a_0}$.
Thus 
\[ \qX_G(q) =  \sum_{\kappa}  q^{\asc_{G}(\max\circ \kappa)} x^\kappa
=  \sum_{\kappa}  q^{\des_{G}(\min\circ \kappa^\ast)} \rho(x^{\kappa^\ast}) =  \rho\(\qX^{\des,\min}_G(q)\).
\]

Since  $\max(\kappa(u)) \neq \max(\kappa(v))$ if $\kappa$ is proper and $\{u,v\} \in E(G)$,
it is clear that $\qX_G(q) = q^{|E(G)|} \cdot \qX^{\des,\max}_G(q^{-1})$. The latter expression is the same as what we get by applying $\tau$ since $|E(G)|$ is the degree of $\qX_G(q)$ as a polynomial in $q$:
this number is an upper bounded for the degree by definition, and  
if we write $V(G) = \{v_1 < v_2 < \dots <v_n\}$
then $\asc_G(\kappa)=|E(G)|$ for the proper set-valued coloring with $\kappa(v_i) = \{i\}$.

It follows by a similar argument that  $\qX^{\des,\max}_G(q)=\tau\(\qX^{\asc,\min}_G(q)\)$,
which gives the third desired equality.
\end{proof}

Fix $D \in \MAO(G)$.
Each vertex in $D$
has the form $(v,i)$ for some $v \in V(G)$ and $i \in \PP$.
Define 
$ \align(D) :=   |\{  (u,i)\to(v,j) \in E(D) : u<v\text{ and }i=j=1\}|.$

\begin{proposition} \label{MAO-prop2}
If $G$ is an ordered graph  then
\[\textstyle \qX_{G}(q) = \sum_{D \in \MAO(G)} q^{\align(D) } \bGamma(D).\]
This power series is multifundamental positive
 in the sense of being a possibly infinite $\NN[q]$-linear combination of 
 multifundamental quasisymmetric functions.
\end{proposition}

\begin{proof}
Fix $D \in \MAO(G)$ and consider the maps $\kappa : V(D) \to \Set(\PP)$
that have $\kappa(u) \prec\kappa(v)$ whenever $u \to v$ is an edge in $D$,
so that $\bGamma(D) = \sum_\kappa x^\kappa$.
If $v \in V(G)$ and $i>j$ are positive integers with $(v,i),(v,j) \in V(D)$
then $\kappa((v,i)) \prec \kappa((v,j))$ by Proposition~\ref{MAO-prop}.
Define a map $\eta : V(G) \to \Set(\PP)$ by setting $\eta(v) = \bigsqcup_i \kappa((v,i))$
where the union is over all $i \in \PP$ with $(v,i) \in D$.

This is a proper set-valued coloring of $G$
since if $\{u,v\} \in E(G)$
then every vertex of the form $(u,i) \in V(D)$ is connected by an edge to every
vertex of the form $(v,j) \in V(D)$.
Observe that $q^{\asc_G(\max\circ \eta)} = q^{\align(D)}$ and $x^\eta = x^\kappa$.

We claim that every proper set-valued coloring of $G$ arises as $\eta$ for some  $D \in \MAO(G)$ and $\kappa : V(D) \to \Set(\PP)$ as above.
For such a coloring $\eta$,  let \[\eta(v) = \eta(v)_1 \sqcup \eta(v)_2 \sqcup \dots \sqcup \eta(v)_k\]
be the smallest disjoint decomposition into nonempty subsets
such that for each $i \in [k-1]$ an edge $\{v,w\} \in E(G)$ exists
with a color $c \in \eta(w)$  that has  $\eta(v)_i \succ \{c\} \succ \eta(v)_{i+1}$;
then let $\alpha(v) = k$.
Form an acyclic multi-orientation $D \in \MAO(G)$ of type $\alpha$
that has the edge $(v,i) \to (w,j)$ if and only if $v=w$ and $i>j$, or $\{v,w\} \in E(G)$ and $\eta(v)_i \prec \eta(w)_j$.
Then the coloring $\eta$ arises from the map $\kappa : V(D) \to \Set(\PP)$ given by $\kappa((v,i)) = \eta(v)_i$.

As we already know that  $\bX_{G} = \sum_{D \in \MAO(G)} \bGamma(D)$ from 
Corollary~\ref{mf-p-cor},
we conclude that the operator $(D,\kappa)\mapsto \eta$ is a bijection to proper set-valued colorings of $G$.
Thus $\qX_{G}(q) = \sum_\eta q^{\asc_G(\eta)} x^\kappa 
%=  \sum_{D \in \MAO(G)} \sum_\kappa q^{\align(D)} x^\kappa 
=  \sum_{D \in \MAO(G)}q^{\align(D)}\bGamma(D)$.
\end{proof}

We can make this expansion more explicit, generalizing a result in \cite{SW2016}.
Suppose $D$ is a directed acyclic graph. 
%A \defn{decreasing labeling} of $D$
%a injective map $\gamma : V(D) \to \ZZ$ with $\gamma(x) > \gamma(y)$ whenever $x \to y \in E(D)$. 
Write $<_D$ for the transitive closure of the relation on $V(D)$ with $x <_D y$ if  $x\to y \in E(D)$.

\begin{lemma}\label{gamma-lem}
Suppose $G = \inc(P)$ for a finite poset $P$ and $D \in \AO(G)$.
Then there exists a decreasing labeling $\gamma$ of $D$ such that 
$\gamma(x) < \gamma(y)$ for all vertices $x$ and $y$  that are incomparable under $<_D$ but have $x <_P y$.
\end{lemma}

\begin{proof}
If we change the condition ``$x <_P y$'' 
to ``$x >_P y$''  then the lemma becomes the \emph{Claim} established in the proof of \cite[Thm.~3.1]{SW2016}. This alternate version of the lemma is equivalent to the one stated
since we can replace $P$ by its dual $P^\ast$ without changing $G$,
and $x >_{P^\ast} y$ holds if and only if  $x <_{P} y$.

More directly, following the argument in the proof of \cite[Thm.~3.1]{SW2016},
the way to construct $\gamma$ is as follows. All vertices that are maximal under $<_D$ are comparable in $P$, and we let $x_1$ be the $P$-smallest such vertex. Likewise, assuming $x_1,x_2,\dots,x_k$ have been defined, all vertices in $V(D) \setminus\{x_1,x_2,\dots,x_k\}$ that are maximal under $<_D$ are comparable in $P$, and we let $x_{k+1}$ be the $P$-smallest such vertex.
After letting $k$ range from $1$ to $|D|-1$, we set $\gamma(x_i) = i$.

It is clear that $\gamma$ is a decreasing labeling of $D$. Checking that 
$\gamma(x) < \gamma(y)$ when $x$ and $y$    are incomparable under $<_D$ with $x <_P y$ requires justification. For this, one can adapt the argument in the proof of \cite[Thm.~3.1]{SW2016}. 
\end{proof}

Following \cite{LamPylyavskyy}, a \defn{multipermutation} of $n \in \NN$
is a word $w=w_1w_2\cdots w_m$ with $\{w_1,w_2,\dots,w_m\} = \{1,2,\dots,n\}$
and $w_i \neq w_{i+1}$ for all $i\in[m-1]$. Let $\bS_n$ be the set of all multipermutations of $n$.

For $w=w_1w_2\cdots w_m  \in \bS_n$ 
let $\Inv(w)$ be the set of pairs $(w_i,w_j)$ with  $i<j $ and $w_i > w_j$ and 
$\{w_1,w_2,\dots,w_{i-1}\} \cap \{w_i\} = \{w_1,w_2,\dots,w_{j-1}\} \cap \{w_j\} = \varnothing.$
If $P$ is a poset on $[n]$ and $G = \inc(P)$ is its incomparability graph, then set 
\be
\ba
\inv_G(w) &:= |\{ (a,b) \in \Inv(w) : \{a,b\} \in E(G)\}|,
\\
 S(w,P) &:=   \{ m - i : i \in [m - 1] \text{ and }w_i \not >_P w_{i+1}\}.
 \ea
 \ee
The following generalizes \cite[Thm.~3.1]{SW2016}:

\begin{theorem}
If $G = \inc(P)$ for a poset $P$ on $[n]$ then 
\[\textstyle \qX_{G}(q) = \sum_{w \in \bS_n} q^{\inv_G(w) } \bL_{\ell(w), S(w,P)}.\]
\end{theorem}

\begin{proof}
Let $G = \inc(P)$ for a poset $P$ on $[n]$.
By \eqref{ML-eq} and Proposition~\ref{MAO-prop2},
\be\label{sumsum-eq}
\qX_{G}(q) = \sum_{D \in \MAO(G)} \sum_{v \in \mL(D)} q^{\align(D) } \bL_{\ell(v), \Des(v,\gamma_D)}
\ee
where for each $D \in \MAO(G)$ the map $\gamma_D $ is any
{decreasing labeling} of $D$.

Fix $D \in \MAO(G)$ and $v=(v_1,v_2,\cdots, v_m) \in \cM(D)$.
Each entry of $v$ has the form $v_j = (w_j, i_j)$ for some $w_j \in V(G) = [n]$
and $i_j \in \PP$.
If $a := w_j = w_k$ for some $1\leq j<k\leq n$ then we must have $i_j \geq i_k$,
since if $i_j < i_k$ then a directed path must connect $(a,i_k)$ to $(a,i_j)$ in $D$ which would imply that $k<j$.  
We can never have $w_j = w_{j+1}$ since $(a,i+1) <_D (a,i)$ is never a covering relation when $D$ is an acyclic multi-orientation. Hence, the sequence $w=w_1w_2\cdots w_m$ belongs to $\bS_n$.

Let $w^\r = w_m\cdots w_2w_1 \in \bS_n$  and observe that 
\be\label{al-eq} \align(D) = \inv_G(w^\r).\ee
Note that $D$ is an acyclic orientation of the incomparability graph of the partial order  on $V(D)\subset [n]\times \PP$
that has $(x,i) < (y,j)$ if and only if $x<_P y$. By Lemma~\ref{gamma-lem},
we can therefore choose $\gamma_D$ to be a decreasing 
 labeling of $D$ with
 $\gamma_D(x,i) < \gamma_D(y,j)$ for all $(x,i),(y,i) \in V(D) $ that are incomparable under $<_D$
but have $x<_P y$.
   For this labeling, we have 
\be\label{desv-eq} 
\Des(v,\gamma_D) %&=\{  j\in [m - 1]: \gamma(v_j)> \gamma(v_{j+1})\} 
=  \{  j \in [m - 1]: w_{j} \not <_P w_{j+1}\} = S(w^\r,P) 
\ee
since for each $j \in[m-1]$ one of the following occurs:
\bei
\item[(1)] If $\{w_j,w_{j+1}\} \in E(G)$ then
 $v_j$ and $v_{j+1}$ are comparable in $<_D$ by the definition of an acyclic multi-orientation.
In this case $v_j <_D v_{j+1}$ as we cannot have $v_{j+1} <_D v_{j}$ for $v \in \cM(D)$,
so $\gamma_D(v_j) > \gamma_D(v_{j+1})$ and $j \in \Des(v,\gamma_D)$,
while at the same time $w_j \not <_P w_{j+1}$ holds as $G = \inc(P)$.

\item[(2)] If $\{w_j,w_{j+1}\} \notin E(G)$
then $v_j \to v_{j+1}$ and $v_{j+1}\to v_j$ are not edges in $D$,
so $v_j$ and $v_{j+1}$ are incomparable in $<_D$, as any vertices in a directed path 
from $v_j$ to $v_{j+1}$ in $D$ would have to appear between two consecutive elements $v$, which is impossible. In this case,
since
$w_j$ and $w_{j+1}$
are comparable in $P$, 
we have $j \in \Des(v,\gamma_D)$ if and only if 
$w_j \not<_P w_{j+1}$.
\eei

Now observe that $(D,v) \mapsto w$ is a bijection from the set of pairs $(D,v)$ with $D \in \MAO(G)$ and $v \in \cM(D)$ to $\bS_n$. To recover $v=(v_1,v_2,\dots,v_m)$ from $w=w_1w_2\cdots w_m$, 
we set $v_j = (w_j, i_j)$ where either
  $i_j = 1$ if $w_j \notin \{w_{j+1},\dots,w_m\}$,
or if there is a  smallest index $k$ with $j<k\leq m$ such that  $w_j=w_k$, then 
\[
i_j =\begin{cases} i_k&\text{if  
$w_{j+1},w_{j+2},\dots,w_{k-1}$ are all comparable to $w_j=w_k$ in $P$}; \\
 i_k+1 &\text{if
one of  $w_{j+1},w_{j+2},\dots,w_{k-1}$ is incomparable to $w_j=w_k$ in $P$}.
\end{cases}
\]
In turn, we recover $D$ as the directed graph with vertex set $\{v_1,v_2,\dots,v_m\}$
that has a directed edge $v_j \to v_k$ whenever either $w_j = w_k$ and $i_j > i_k$,
or $\{w_j,w_k\} \in E(G)$ and $\{ l \in [m] : v_l = v_j\} \prec \{ l \in [m] : v_l = v_k\}$.

Combining \eqref{al-eq} and \eqref{desv-eq} with this observation turns 
\eqref{sumsum-eq} into
$\qX_{G}(q) = \sum_{w \in \bS_n} q^{\inv_G(w^\r) } \bL_{\ell(w^\r), S(w^\r,P)}$
which is equivalent to the desired formula.
%
%For each $D \in \MAO(G)$, the induced subgraph of $D$ on the set of vertices of the form $(a,1)$ for $a \in V(G)$ defines  an acyclic orientation $O \in \AO(G)$. As explained in the proof of \cite[Thm.~3.1]{SW2016}, one can
%construct a bijection $\gamma_O^\op : [n] \to [n]$ with the following properties:
%\bei
%\item If $x,y \in [n]$ and $x\to y \in E(O)$ then  $\gamma_O^\op(x) > \gamma_O^\op(y)$.
%\item If $x,y \in [n]$ are incomparable in $O$ and $x <_P y$ then $\gamma_O^\op(x) > \gamma_O^\op(y)$.
%\item
%\eei
\end{proof}

Unlike $X_G(q)$, the power series $\qX_G(q)$ is rarely symmetric.
However, we can exactly  characterize the ordered graphs $G$ for which $\qX_G(q) \in \mSym$.

\begin{theorem}\label{cluster-thm}
One has $\qX_G(q) \in \mSym[q]$ if and only if $G$ is a \defn{cluster graph},
that is, a  disjoint union of complete graphs.
%, that is, a disjoint union of complete graphs.
\end{theorem}

The proof of the theorem will use three lemmas. 
For an ordered graph $G$, let $\min G$ be the set of vertices 
$v \in V(G)$ such that if $\{v,w\} \in E(G)$ then $v<w$,
and let $\max G$ be the set of vertices $v \in V(G)$ such that if $\{u,v\} \in E(G)$ then $u<v$.
Both $\min G$ and $\max G$ are independent sets.
Also let $W_G := X_G(0) = \sum_{\kappa} x^\kappa$
where the sum is over proper colorings $\kappa :V(G)\to \PP$
with $\asc_G(\kappa) = \varnothing$.

\begin{lemma}\label{cluster-lem0}
If $W_G \in \Sym[q]$ then $|\min G|=|\max G|$.
\end{lemma}

\begin{proof}
If $f = \sum_\alpha c_\alpha M_\alpha$ is a quasisymmetric function then let 
$[M_\alpha]f := c_\alpha$. Notice that this is also the coefficient of $x^\alpha = x_1^{\alpha_1} x_2^{\alpha_2}\cdots$ in $f$.
%Suppose $G$ is an ordered graph.

 Now let $p = |\min G|$, $m = |\max G|$, and $n = |G|$.
If $p>m$ then \[ [M_{(p,1^{n-p})}] W_G = 0 < [M_{(1^{n-p},p)}] W_G \]
since in this case 
there is no ascent-less proper coloring of $G$ that assigns the color $1$ to $p$ distinct vertices (as to have no ascents every 1-colored vertex must be in $\max G$), but there is an ascent-less proper coloring 
that  assigns the same color
$n-p+1$ to all $p$ minimal vertices while assigning the distinct colors $1,2,\dots,n-p$ to the non-minimal vertices
(if these vertices are totally ordered as $v_1>v_2>\dots>v_{n-p}$ then let vertex $v_i$ have color $i$).

Alternatively, if $p<m$ then \[  [M_{(1^{n-m},m)}] W_G = 0 < [M_{(m,1^{n-m})}] W_G \] 
since in this case 
there is no ascent-less proper coloring of $G$ that assigns the
distinct colors $1,2,\dots,n-m$ some set of vertices while assigning
the same color $n-m+1$ to the $m$ vertices that remain (as to have no ascents every $(n-m+1)$-colored vertex must be in $\min G$), but there is an ascent-less proper coloring 
that  assigns the same color
$1$ to all $m$ maximal vertices while assigning the distinct colors $2,3,\dots,n-m+1$ to the non-maximal vertices
(if these vertices are totally ordered as $v_1>v_2>\dots>v_{n-m}$ then let vertex $v_i$ have color $i+1$).

In either case, we have $[M_\alpha]W_G \neq [M_\beta]W_G$
for two strict compositions $\alpha$ and $\beta$ that sort to the same partition,
so
 $W_G$ is not symmetric.
We conclude that if $W_G \in \Sym$ is symmetric then $p=m$. 
\end{proof}

\begin{lemma}\label{cluster-lem1}
If $\qX_G(q) \in \mSym[q]$ then $W_G \in \Sym$
and every
independent subset of $G$ has size at most $ |\max G|$.
\end{lemma}

\begin{proof}
As $W_G$ is the homogeneous component of degree $|G|$ in $\qX_G(0)$,  
if we have $\qX_G(q) \in \mSym$ then $W_G \in \Sym$.
Let $n:=|G|$. If $S$ is an independent subset of $G$ of size $s := |S| > |\min G|$,
then \[  [M_{(1^n,s)}]\qX_G(0) = 0 < [M_{(s,1^n)}] \qX_G(0) \]
since in this case 
there is no ascent-less proper set-valued coloring of $G$ that assigns its largest color to $s$ distinct vertices (as to have no ascents, each vertex
whose colors include the largest appearing color must be in $\min G$),
but if 
the vertices of $G$ are totally ordered as $v_1>v_2>\dots>v_n$ then the
set-valued coloring $\kappa : G \to \Set(\PP)$
with 
\[\kappa(v_i) = \begin{cases} \{1,i+1\}&\text{if }v_i \in S \\ 
\{i+1\} &\text{if }v_i \notin S\end{cases}
\] is ascent-less and proper, and this contributes at least one monomial term
$x_1^s x_2x_3\cdots x_{n+1}$ to the monomial-positive expansion of $\qX_G(0)$.

Thus if
$G$ has an independent subset with more than $ |\max G|$ elements then
$\qX_G(0)$ is not symmetric.
Therefore if $\qX_G(q) \in \mSym$, then $W_G$ is symmetric so
 $|\min G| = |\max G|$ by Lemma~\ref{cluster-lem0}, and
 $\qX_G(0)$ is symmetric so
every
independent subset of $G$ has size at most $|\min G| = |\max G|$.
\end{proof}

\begin{lemma}\label{cluster-lem2}
Suppose $W_G \in \Sym$
and every independent subset of $G$ has size at most $ |\max G|$.
Then $G$ is a cluster graph.
\end{lemma}

\begin{proof}
Suppose $G$ has an isolated vertex $v$. Let $K$ be the subgraph formed by removing this vertex.
Then 
$W_G = e_1 W_K$ so $W_K$ is symmetric, and if $S$ is an independent subset of $K$
then $|S| \leq |\max K|$ since $S \sqcup \{v\}$ is an independent subset of $G$ and $\max G  = \max K \sqcup \{v\}$. In this case, we may assume by induction that $K$ is a cluster graph, and then $G$ is also a cluster graph.

Assume $G$ has no isolated vertices so that $\min G$ and $\max G$ are disjoint.
Let $m = |\min G|$ and 
define $H$ to be the subgraph $G$ induced on the vertex set $V(G) \setminus \min G$.
If $m=0$ then $G$ is empty so there is nothing to prove. Therefore, assume $m>0$ so that $H$
is a proper subgraph.

If  $\alpha = (\alpha_1,\alpha_2,\dots,\alpha_l)$
is any strict composition of $|H|$ then
 \[[M_\alpha] W_H = [M_{(\alpha_1,\alpha_2,\dots,\alpha_l,m)}] W_G\]
 since in 
 any ascent-less proper coloring of $G$ assigning the color $i$ to $\alpha_i$
 vertices 
 and the color $l+1$ to $m$ vertices, 
 the set of vertices colored by $l+1$ must be exactly $\min G$, 
  so restriction defines a bijection from such colorings to 
 ascent-less proper colorings of $H$ that assign the color $i$ to $\alpha_i$
 vertices.
 Since $W_G \in \Sym$, it follows that $W_H \in \Sym$.
 
Since $G$ has no isolated vertices, one has
 $\max H = \max G$. Therefore $|\min H| = |\max H| = |\min G| = |\max G|=m$. If $S$ is an independent subset of $H$ then $S$ is also an independent subset of $G$,
 so $|S| \leq |\max G| = |\max H|$. 
Thus, we may assume by induction that $H$ is a disjoint union of $m$ complete graphs. 

We now make some observations.
Each element of $\min H$ must be connected by an edge in $G$ to at least one element of $\min G$
since $\min H$ and $\min G$ are disjoint. 
Any two elements of $\min H$ must be connected by edges in $G$ to at least two different elements of $\min G$,
since if $b,c \in \min H$ were connected to the same element $a \in \min G$ and to no other elements of $\min G$,
then $(\min G \setminus\{a\}) \sqcup \{b,c\}$
would be an independent subset of $G$ of the impossible size $ |\min G| + 1=  |\max G|+1$.
Therefore, if  $p =|H| = |G| - m$, and $k_1,k_2,\dots,k_m$ are the sizes of the $m$ connected components of $H$, then 
\[ [M_{(1^p, m)}] W_G = \tbinom{p}{k_1,k_2,\dots,k_m} \geq [M_{(1^{p-1},m,1)}] W_G\]
with equality if and only if  
each element of $\min H$ is connected by an edge in $G$ to exactly one element of $\min G$. 
%%% more? %%%
This last property must hold as $W_G$ is symmetric.

Let $\min G = \{u_1,u_2,\dots,u_m\}$ and $\min H = \{v_1,v_2,\dots,v_m\}$.
By the properties in the paragraph above, we may assume that 
$\{u_i,v_i\} \in E(G)$ for each $i \in [m]$ and that $\{u_i,v_j\} \notin E(G)$ for $i\neq j$.
For each $i \in [m]$, every vertex $w \in V(H)$ with $\{v_i,w\} \in E(G)$
must have $\{u_i,w) \in E(G)$, since otherwise $(\min H\setminus \{v_i\})\sqcup\{u_i,w\}$ would be 
an independent subset of $G$ of the impossible size $m + 1 = |\max G| +1$.

To conclude that $G$ is a cluster graph, suppose $w \in V(H)$  has $\{v_j,w\} \in E(G)$.
We must show that $\{u_i,w\} \notin E(G)$ for all $i\neq j$.
This is trivial if $m=1$, and if $m>1$  
then no such  edge can exist as then the equalities 
\[
[M_{(1^p, m)}] W_G
=
[M_{(1^{a}, m,1^b)}] W_G
=
[M_{(m,1^p)}] W_G
\] cannot hold for all $a,b \in \PP$ with $a+b=p$,
contradicting $W_G \in \Sym$.
%%% more? %%%
 \end{proof}

\begin{proof}[Proof of Theorem~\ref{cluster-thm}]
If $G$ is a disjoint union of $m$ cluster graphs of sizes $n_1,n_2,\dots,n_m$,
then $\qX_G(q) = \prod_{i=1}^m \qX_{K_{n_i}}(q) \in \mSym[q]$ by
Example~\ref{or-k-ex}.
If $\qX_G(q)\in \mSym[q]$, then $G$ is a cluster graph by Lemmas~\ref{cluster-lem1}
and \ref{cluster-lem2}.
\end{proof}

\subsection{Another quasisymmetric analogue} %Chromatic quasisymmetric function}
\label{qana-sect2}

The preceding results indicate that $\qX_G(q)$ is an interesting
quasisymmetric $q$-analogue of $\bX_G$ and $K$-theoretic analogue of $X_G(q)$. 
However, there is another candidate for such a generalization.
Continue to let $G$ be an ordered graph.
Following \cite{Hwang}, an \defn{ascent} (respectively, \defn{descent}) 
of a set-valued map $\kappa : V(G) \to \Set(\PP)$ is 
a tuple $(u,v,i,j)$ with  $\{u,v\} \in E(G)$, $i \in \kappa(u)$, $j\in \kappa(v)$, $u<v$, and $i<j$ (respectively, $i>j$).
Let $\asc_{G}(\kappa)$ denote the number of ascents,
and let  $\des_{G}(\kappa)$ be the number of descents.

\begin{definition}
Given an ordered graph $G$, define
\[ \oX_G(q) = \sum_{\kappa} q^{\asc_{G}(\kappa)} x^\kappa   \in \mQSym\llbracket q \rrbracket
 \]
where the sum is over all proper set-valued colorings $\kappa : V(G) \to \Set(\PP)$.
\end{definition}

As  with $\qX_G(q)$, it is clear that if $G=G_1\sqcup G_2$ is the disjoint union of two ordered graphs, then we have a factorization
$
 \oX_G(q) = \oX_{G_1}(q) \oX_{G_2}(q).
$

The power series $\oX_G(q)$ is closely related to the quasisymmetric functions 
$X_G(\x,q,\mu)$ studied in \cite{Hwang}.
For each map $\mu : V(G) \to \NN$,   Hwang \cite{Hwang} defines
\be X_G(\x,q,\mu) := \sum_\kappa q^{\asc_{G}(\kappa)} x^\kappa\ee where the sum is over all 
proper set-valued colorings $\kappa$ of $G$ with $|\kappa(v)| = \mu(v)$.
Evidently $\oX_G(q) = \sum_{\mu : V(G) \to \PP} X_G(\x,q,\mu)$, and, as noted in \cite[Rem.~2.2]{Hwang},
\be\label{this-eq}
\textstyle
X_G(\x,q,\mu) = \tfrac{1}{[\mu]_q!} X_{\Cl_\mu(G)}(q)
\quad\text{where $[\mu]_q ! := \prod_{v \in V(G)} [\mu(v)]_q!.$}\ee
Here,  we view $\Cl_\mu(G)$ as an ordered graph with $(v,i) < (w,j)$ if either $v<w$ or $v=w$ and $i<j$.

Using \eqref{this-eq}, various positive or alternating expansions of $X_G(q)$
(e.g., into fundamental quasisymmetric functions \cite[Thm.~3.1]{SW2016}, Schur functions \cite[Thm.~6.3]{SW2016},
power sum symmetric functions \cite[Thm.~3.1]{At2015}, or elementary symmetric functions \cite[Conj.~5.1]{SW2016})
can be extended in a straightforward way to $X_G(\x,q,\mu) $ and $\oX_G(q)$.
See Hwang's results
\cite[Thms.~3.3, 4.10, and 4.19]{Hwang}  and his conjecture \cite[Conj.~3.10]{Hwang}.

Some of these statements require $G$ to be isomorphic to the incomparability graph of a \defn{natural unit interval order}, meaning a poset $P$ on a finite subset of $\PP$ such that if  $x<_P z$
then $x<z$ and every $y$ incomparable in $P$ to both $x$ and $z$ has $x<y<z$ \cite[Prop.~4.1]{SW2016}. 

\begin{remark}
A finite poset is isomorphic to a (unique) natural unit interval order if and only if it is both $(3+1)$- and $(2+2)$-free \cite[Prop.~4.2]{SW2016}.
This does not mean that every $(3+1)$- and $(2+2)$-free poset on a finite subset of $\PP$ is a natural unit interval order. For example, the poset $P=\{ 1<2\} \sqcup\{ 3\}$ is $(3+1)$- and $(2+2)$-free but not a natural unit interval order; however, it is isomorphic to the  natural unit interval order $Q = \{1<3\} \sqcup\{2\}$.
\end{remark}

If $G$ is the incomparability graph of a natural unit interval order,
then so are all of its $\alpha$-clans, and the following holds:

\begin{lemma}\label{des-nuo-lem}
If $G$ is the incomparability graph of a natural unit order interval,  then  $\oX_G(q)$ is symmetric and 
$\oX_G(q) = \sum_{\kappa} q^{\des_{G}(\kappa)} x^\kappa $
where the sum is over all proper set-valued colorings $\kappa : V(G) \to \Set(\PP)$.
\end{lemma}

\begin{proof}
In this case $X_G(\x,q,\mu)$ is symmetric by  \cite[Thm.~3.8]{Hwang},
 so the same is true of $\oX_G(q)$, and the alternate formula for $\oX_G(q)$ holds by
 \cite[Prop.~2.1]{Hwang}.
\end{proof}

Some sample computations of $\oX_G(q)$ are shown in Table~\ref{oX-tab}.
One of the few cases where we have explicit formulas is discussed in the following example.

 \begin{table}[htb]
\centerline{
\scalebox{1}{
\begin{tabular}{|l|l|l|}
\hline & & 
\\[-8pt]
Graph $G$   & $\bX_G(q)$ in the monomial basis $\{M_\alpha\}$ & $\bX_G(q)$ in the multifundamental basis $\{\bL_\alpha\}$
\\[-8pt] & & 
\\ \hline & & 
\\[-8pt]  
  \begin{tikzpicture}[scale=1,baseline=(void.base)]%3-path
  \node at (0, 0)(1){2};
    \node at (1, 0)(2){3};
    \node at (0.5, 0.866)(3){1};
    \node at (0.5,0.433)(void) {};
    \draw[black, thick] (3) -- (1);
    \draw[black, thick] (2) -- (1);
  \end{tikzpicture} 
  &  \pbox[c]{7cm}{\scriptsize$(1 + 4q + q^2) M_{111} + q M_{12} + q M_{21} + (3 + 12q + 16q^2 + 4q^3 + q^4) M_{1111} + (2q + 3q^2) M_{112} + (2q + 3q^2) M_{121} + (2q + 3q^2) M_{211}+ \dots$} 
  & \pbox[c]{7cm}{\scriptsize $(1 + 2q + q^2) \bL_{111} + q \bL_{12} + q \bL_{21} + (4q^2 + 4q^3 + q^4) \bL_{1111} + 3q^2 \bL_{112} + 3q^2 \bL_{121} + 3q^2 \bL_{211} + \dots $
}
\\ [-4pt] && \\
  \begin{tikzpicture}[scale=1,baseline=(void.base)]%3-path
  \node at (0, 0)(1){3};
    \node at (1, 0)(2){2};
    \node at (0.5, 0.866)(3){1};
    \node at (0.5,0.433)(void) {};
    \draw[black, thick] (3) -- (1);
    \draw[black, thick] (2) -- (1);
  \end{tikzpicture} 
  &  \pbox[c]{7cm}{\scriptsize$(2 + 2q + 2q^2) M_{111} + M_{12} + q^2 M_{21} + (8 + 8q + 10q^2 + 8q^3 + 2q^4) M_{1111} + (3 + 2q) M_{112} + (2 + q^2 + 2q^3) M_{121} + (2q^2 + 2q^3 + q^4) M_{211}+ \dots$} 
  &  \pbox[c]{7cm}{\scriptsize$(1 + 2q + q^2) \bL_{111} + \bL_{12} + q^2 \bL_{21} + (4q^2 + 4q^3 + q^4) \bL_{1111} + (1 + 2q) \bL_{112} + (1 + 2q^3) \bL_{121} + (2q^3 + q^4) \bL_{211} +\dots$}
\\ [-4pt] && \\
  \begin{tikzpicture}[scale=1,baseline=(void.base)]%3-path
  \node at (0, 0)(1){1};
    \node at (1, 0)(2){3};
    \node at (0.5, 0.866)(3){2};
    \node at (0.5,0.433)(void) {};
    \draw[black, thick] (3) -- (1);
    \draw[black, thick] (2) -- (1);
  \end{tikzpicture} 
  &  \pbox[c]{7cm}{\scriptsize$(2 + 2q + 2q^2) M_{111} + q^2 M_{12} + M_{21} + (8 + 8q + 10q^2 + 8q^3 + 2q^4) M_{1111} + (2q^2 + 2q^3 + q^4) M_{112} + (2 + q^2 + 2q^3) M_{121} + (3 + 2q) M_{211}+ \dots$} 
  & \pbox[c]{7cm}{\scriptsize $(1 + 2q + q^2) \bL_{111} + q^2 \bL_{12} + \bL_{21} + (4q^2 + 4q^3 + q^4) \bL_{1111} + (2q^3 + q^4) \bL_{112} + (1 + 2q^3) \bL_{121} + (1 + 2q) \bL_{211}  + \dots $}
 \\[-8pt] & & 
 \\ \hline & & 
\\[-8pt]  
    \begin{tikzpicture}[scale=1,baseline=(void.base)]%3-cycle
  \node at (0, 0)(1){2};
    \node at (1, 0)(2){3};
    \node at (0.5, 0.866)(3){1};
    \node at (0.5,.433)(void) {};
    \draw[black, thick] (3) -- (1);
    \draw[black, thick] (2) -- (1);
    \draw[black, thick] (2) -- (3);
  \end{tikzpicture} & 
    \pbox[c]{7cm}{\scriptsize $(1 + 2q + 2q^2 + q^3) M_{111} + (3 + 6q + 9q^2 + 9q^3 + 6q^4 + 3q^5) M_{1111}+ \dots $ }  &
  \pbox[c]{7cm}{\scriptsize$(1 + 2q + 2q^2 + q^3) \bL_{111} + (3q^2 + 6q^3 + 6q^4 + 3q^5) \bL_{1111} + \dots $}
   \\[-8pt] & & 
    \\ \hline & & 
\\[-8pt]  
   \begin{tikzpicture}[scale=1,baseline=(void.base)]%4-cycle
  \node at (0, 0)(1){1};
    \node at (1, 0)(2){2};
    \node at (1, 1)(3){3};
     \node at (0, 1)(4){4};
         \node at (0.5,.5)(void) {};
    \draw[black, thick] (3) -- (2);
    \draw[black, thick] (2) -- (1);
    \draw[black, thick] (4) -- (3);
    \draw[black, thick] (1) -- (4); 
  \end{tikzpicture} & 
    \pbox[c]{7cm}{\scriptsize $(1 + 8q + 6q^2 + 8q^3 + q^4) M_{1111} + (2q + 2q^3) M_{112} + (2q + 2q^3) M_{121} + (2q + 2q^3) M_{211} + (q + q^3) M_{22} + (4 + 32q + 56q^2 + 56q^3 + 56q^4 + 32q^5 + 4q^6) M_{11111} + (6q + 9q^2 + 6q^3 + 9q^4 + 6q^5) M_{1112} + (6q + 9q^2 + 6q^3 + 9q^4 + 6q^5) M_{1121} + (6q + 9q^2 + 6q^3 + 9q^4 + 6q^5) M_{1211} + (6q + 9q^2 + 6q^3 + 9q^4 + 6q^5) M_{2111} + (q + 2q^2 + 2q^3 + 2q^4 + q^5) M_{122} + (2q + 2q^2 + 2q^4 + 2q^5) M_{212} + (q + 2q^2 + 2q^3 + 2q^4 + q^5) M_{221}+ \dots $ }  & 
   \pbox[c]{7cm}{\scriptsize $(1 + 3q + 6q^2 + 3q^3 + q^4) \bL_{1111} + (q + q^3) \bL_{112} + (2q + 2q^3) \bL_{121} + (q + q^3) \bL_{211} + (q + q^3) \bL_{22} + (2q^2 + 24q^3 + 22q^4 + 12q^5 + 4q^6) \bL_{11111} + (5q^2 + q^3 + 5q^4 + 3q^5) \bL_{1112} + (7q^2 - q^3 + 7q^4 + 5q^5) \bL_{1121} + (7q^2 - q^3 + 7q^4 + 5q^5) \bL_{1211} + (5q^2 + q^3 + 5q^4 + 3q^5) \bL_{2111} + (2q^2 + q^3 + 2q^4 + q^5) \bL_{122} + (2q^2 - 2q^3 + 2q^4 + 2q^5) \bL_{212} + (2q^2 + q^3 + 2q^4 + q^5) \bL_{221} + \dots$}
\\ [-4pt] && \\
      \begin{tikzpicture}[scale=1,baseline=(void.base)]%4-cycle
  \node at (0, 0)(1){1};
    \node at (1, 0)(2){2};
    \node at (1, 1)(3){4};
     \node at (0, 1)(4){3};
         \node at (0.5,.5)(void) {};
    \draw[black, thick] (3) -- (2);
    \draw[black, thick] (2) -- (1);
    \draw[black, thick] (4) -- (3);
    \draw[black, thick] (1) -- (4); 
  \end{tikzpicture} & 
    \pbox[c]{7cm}{\scriptsize $(2 + 4q + 12q^2 + 4q^3 + 2q^4) M_{1111} + 4q^2 M_{112} + (1 + 2q^2 + q^4) M_{121} + 4q^2 M_{211} + 2q^2 M_{22} + (10 + 20q + 50q^2 + 80q^3 + 50q^4 + 20q^5 + 10q^6) M_{11111} + (9q^2 + 18q^3 + 9q^4) M_{1112} + (3 + 2q + 6q^2 + 14q^3 + 6q^4 + 2q^5 + 3q^6) M_{1121} + (3 + 2q + 6q^2 + 14q^3 + 6q^4 + 2q^5 + 3q^6) M_{1211} + (9q^2 + 18q^3 + 9q^4) M_{2111} + (2q^2 + 4q^3 + 2q^4) M_{122} + (2q^2 + 4q^3 + 2q^4) M_{212} + (2q^2 + 4q^3 + 2q^4) M_{221}+ \dots $ }  & 
   \pbox[c]{7cm}{\scriptsize $(1 + 4q + 4q^2 + 4q^3 + q^4) \bL_{1111} + 2q^2 \bL_{112} + (1 + 2q^2 + q^4) \bL_{121} + 2q^2 \bL_{211} + 2q^2 \bL_{22} + (10q^2 + 12q^3 + 22q^4 + 16q^5 + 4q^6) \bL_{11111} + (-q^2 + 10q^3 + 5q^4) \bL_{1112} + (1 + 2q - 2q^2 + 10q^3 + 2q^4 + 2q^5 + 3q^6) \bL_{1121} + (1 + 2q - 2q^2 + 10q^3 + 2q^4 + 2q^5 + 3q^6) \bL_{1211} + (-q^2 + 10q^3 + 5q^4) \bL_{2111} + (4q^3 + 2q^4) \bL_{122} + (-2q^2 + 4q^3 + 2q^4) \bL_{212} + (4q^3 + 2q^4) \bL_{221} + \dots$}
\\ [-4pt] && \\
      \begin{tikzpicture}[scale=1,baseline=(void.base)]%4-cycle
  \node at (0, 0)(1){1};
    \node at (1, 0)(2){3};
    \node at (1, 1)(3){2};
     \node at (0, 1)(4){4};
         \node at (0.5,.5)(void) {};
    \draw[black, thick] (3) -- (2);
    \draw[black, thick] (2) -- (1);
    \draw[black, thick] (4) -- (3);
    \draw[black, thick] (1) -- (4); 
  \end{tikzpicture} & 
    \pbox[c]{7cm}{\scriptsize $(4 + 4q + 8q^2 + 4q^3 + 4q^4) M_{1111} + (2 + 2q^4) M_{112} + 4q^2 M_{121} + (2 + 2q^4) M_{211} + (1 + q^4) M_{22} + (24 + 24q + 48q^2 + 48q^3 + 48q^4 + 24q^5 + 24q^6) M_{11111} + (10 + 4q + 4q^2 + 4q^4 + 4q^5 + 10q^6) M_{1112} + (4 + 10q^2 + 8q^3 + 10q^4 + 4q^6) M_{1121} + (4 + 10q^2 + 8q^3 + 10q^4 + 4q^6) M_{1211} + (10 + 4q + 4q^2 + 4q^4 + 4q^5 + 10q^6) M_{2111} + (2 + 2q^2 + 2q^4 + 2q^6) M_{122} + (4 + 4q^6) M_{212} + (2 + 2q^2 + 2q^4 + 2q^6) M_{221}+ \dots $ }  & 
   \pbox[c]{7cm}{\scriptsize $(1 + 4q + 4q^2 + 4q^3 + q^4) \bL_{1111} + (1 + q^4) \bL_{112} + 4q^2 \bL_{121} + (1 + q^4) \bL_{211} + (1 + q^4) \bL_{22} + (8q^2 + 16q^3 + 20q^4 + 16q^5 + 4q^6) \bL_{11111} + (1 + 4q + 2q^2 - q^4 + 4q^5 + 4q^6) \bL_{1112} + (1 + 8q^3 + 7q^4 + 2q^6) \bL_{1121} + (1 + 8q^3 + 7q^4 + 2q^6) \bL_{1211} + (1 + 4q + 2q^2 - q^4 + 4q^5 + 4q^6) \bL_{2111} + (1 + 2q^2 + q^4 + 2q^6) \bL_{122} + (2 - 2q^4 + 4q^6) \bL_{212} + (1 + 2q^2 + q^4 + 2q^6) \bL_{221}+ \dots$}
   \\[-8pt] & & 
   \\ \hline & & 
\\[-8pt]  
\begin{tikzpicture}[scale=1,baseline=(void.base)]%K_4 minus an edge
  \node at (0, 0)(1){1};
    \node at (1, 0)(2){2};
    \node at (1, 1)(3){3};
     \node at (0, 1)(4){4};
    \node at (0.5,0.5)(void){};
    \draw[black, thick] (3) -- (1);
    \draw[black, thick] (2) -- (1);
    \draw[black, thick] (4) -- (1);
    \draw[black, thick] (4) -- (3);
    \draw[black, thick] (2) -- (4); 
  \end{tikzpicture} & 
  \pbox[c]{7cm}{\scriptsize $(2 + 4q + 6q^2 + 6q^3 + 4q^4 + 2q^5) M_{1111} + (q^2 + q^3) M_{112} + (1 + q^5) M_{121} + (q^2 + q^3) M_{211} + (10 + 20q + 34q^2 + 44q^3 + 48q^4 + 38q^5 + 28q^6 + 14q^7 + 4q^8) M_{11111} + (3q^2 + 5q^3 + 6q^4 + 3q^5 + q^6) M_{1112} + (3 + 2q + q^2 + 3q^3 + 2q^4 + q^5 + 3q^6 + 2q^7 + q^8) M_{1121} + (3 + 2q + q^2 + 3q^3 + 2q^4 + q^5 + 3q^6 + 2q^7 + q^8) M_{1211} + (3q^2 + 5q^3 + 6q^4 + 3q^5 + q^6) M_{2111}+\dots$  } & 
  \pbox[c]{7cm}{\scriptsize $(1 + 4q + 4q^2 + 4q^3 + 4q^4 + q^5) \bL_{1111} + (q^2 + q^3) \bL_{112} + (1 + q^5) \bL_{121} + (q^2 + q^3) \bL_{211} + (10q^2 + 12q^3 + 16q^4 + 26q^5 + 20q^6 + 10q^7 + 2q^8) \bL_{11111} + (2q^3 + 6q^4 + 3q^5 + q^6) \bL_{1112} + (1 + 2q + 2q^3 + 2q^4 - q^5 + 3q^6 + 2q^7 + q^8) \bL_{1121} + (1 + 2q + 2q^3 + 2q^4 - q^5 + 3q^6 + 2q^7 + q^8) \bL_{1211} + (2q^3 + 6q^4 + 3q^5 + q^6) \bL_{2111}+\dots$} 
   \\[-8pt] & & 
   \\ \hline & & 
\\[-8pt]  
  \begin{tikzpicture}[scale=1,baseline=(void.base)]%claw
  \node at (0, 0)(1){1};
    \node at (1, 0)(2){2};
    \node at (1, 1)(3){3};
     \node at (0, 1)(4){4};
    \node at (0.5,0.5)(void){};
    \draw[black, thick] (3) -- (1);
    \draw[black, thick] (2) -- (1);
    \draw[black, thick] (1) -- (4);
  \end{tikzpicture} 
 & \pbox[c]{7cm}{\scriptsize $(6 + 6q + 6q^2 + 6q^3) M_{1111} + (3q^2 + 3q^3) M_{112} + (3 + 3q^3) M_{121} + (3 + 3q) M_{211} + q^3 M_{13} + M_{31} + (42 + 42q + 48q^2 + 48q^3 + 48q^4 + 6q^5 + 6q^6) M_{11111} + (15q^2 + 15q^3 + 18q^4 + 3q^5 + 3q^6) M_{1112} + (15 + 3q^2 + 18q^3 + 15q^4 + 3q^6) M_{1121} + (18 + 15q + 3q^3 + 18q^4) M_{1211} + (18 + 18q + 18q^2) M_{2111} + (3q^3 + 3q^4 + q^6) M_{113} + (3 + q^3 + 3q^4) M_{131} + (4 + 3q) M_{311} + 6q^4 M_{122} + 6q^2 M_{212} + 6 M_{221}+\dots$ }
 &  \pbox[c]{7cm}{\scriptsize $(1 + 3q + 3q^2 + q^3) \bL_{1111} + (3q^2 + 2q^3) \bL_{112} + (2 + 2q^3) \bL_{121} + (2 + 3q) \bL_{211} + q^3 \bL_{13} + \bL_{31} + (6q^2 + 12q^3 + 9q^4 + 3q^5 + q^6) \bL_{11111} + (6q^3 + 9q^4 + 3q^5 + 2q^6) \bL_{1112} + (2 + 8q^3 + 9q^4 + 2q^6) \bL_{1121} + (5 + 9q - 2q^3 + 9q^4) \bL_{1211} + (2 + 6q + 12q^2) \bL_{2111} + (q^3 + 3q^4 + q^6) \bL_{113} + (2 + 3q^4) \bL_{131} + (2 + 3q) \bL_{311} + (-q^3 + 6q^4) \bL_{122} + 6q^2 \bL_{212} + 5 \bL_{221}+\dots$} 
  \\[-8pt] & & 
  \\ \hline
\end{tabular}}}
\caption{Partial computations of $\oX_G(q)$ for some small graphs $G$. The ellipses ``$\dots$'' mask higher order terms  indexed by compositions $\alpha$ with  $|\alpha|>|V(G)|+1$.}\label{oX-tab}
\end{table}

\begin{example} If $K_n$ is the complete graph on  $[n]$ then 
\[ \oX_{K_n}(q)= \sum_{r=n}^\infty F_r^{(n)}    e_r
\quad\text{for}\quad F_r^{(n)} :=  \sum_{\substack{k_1,k_2,\dots,k_n \in \PP  \\ k_1+k_2+\dots+k_n = r}} \tbinom{r}{k_1, k_2,\dots,k_n}_q\]
where $(q)_n := \prod_{i\in[n]}(1-q^i)$ and $ \tbinom{r}{k_1, k_2,\dots,k_n}_q  = \frac{(q)_r}{(q)_{k_1} (q)_{k_2}\cdots (q)_{k_n}}$. 
\end{example}

When $q$ is a prime power, $F_r^{(n)}$ counts the \defn{strictly increasing flags} of $\FF_q$-subspaces $0=V_0 \subsetneq V_1 \subsetneq \dots \subsetneq V_n = \FF_q^r.$
The article \cite{Vinroot12} contains a recurrence for the
\defn{generalized Galois number} $G_r^{(n)}$,
which counts all increasing flags
 $0=V_0 \subseteq V_1 \subseteq \dots \subseteq V_n = \FF_q^r.$  
An inclusion-exclusion argument shows that
\be\textstyle
G_r^{(n)}  = \sum_{i=0}^n \tbinom{n}{i} F_r^{(i)}
\quand
F_r^{(n)}  = \sum_{i=0}^n (-1)^{n-i}\tbinom{n}{i} G_r^{(i)}
\ee
where  $F_r^{(0)} = G_r^{(0)}=0$ if $r>0$.
Vinroot shows in \cite{Vinroot12} that
\be\textstyle
G_{r+1}^{(n)} = \sum_{i=0}^{n-1} \tbinom{n}{i+1} (-1)^i \tfrac{(q)_r}{(q)_{r-i}} G^{(n)}_{r-i}
\ee
by using a recurrence for the $q$-multinomial coefficients, namely:
\be
\tbinom{r}{k_1, k_2,\dots,k_n}_q
=\sum_{\varnothing\neq J\subseteq [n]} (-1)^{|J|-1} \tfrac{(q)_{r-1}}{(q)_{r-|J|}} \tbinom{r - |J|}{\underline k - \underline e_J}_q
\ee
where $\underline k = (k_1,k_2,\dots,k_n)$
and where $\underline e_J=(e_1,e_2,\dots,e_n)$ is the $n$-tuple with $e_i = |\{i\} \cap J|$. 
 A more complicated recurrence holds for $F_r^{(n)}$.

\begin{proposition} For each $r \in \NN$ and $n \in \PP$ it holds that
% \[\textstyle F_{r+1}^{(n)}
%=\sum_{i=0}^{n-1}    (-1)^{i} \tfrac{(q)_r}{(q)_{r-i}} \Bigr(\sum_{j=n-1-i}^n  \tbinom{n}{j}\tbinom{j}{n-1-i}  F_{r-i}^{(j)}\Bigr).
%\]
  \[\textstyle F_{r+1}^{(n)}
=\sum_{i=0}^{n-1} \sum_{j=n-1-i}^n      \tbinom{n}{j}\tbinom{j}{n-1-i}   (-1)^{i} \tfrac{(q)_r}{(q)_{r-i}} F_{r-i}^{(j)}
\]
under the convention that $F^{(0)}_0 = 1$ and $F^{(n)}_r = 0$ if $r<n$ or $r>0=n$. 
\end{proposition}

\begin{proof}
Starting from the definition, we have
\[
\ba F_{r+1}^{(n)}
= \sum_{{\underline k  \in \PP^n}} \tbinom{r+1}{\underline k }_q
&
=\sum_{{\underline k  \in \PP^n }} \sum_{\varnothing\neq J\subseteq [n]} (-1)^{|J|-1} \tfrac{(q)_r}{(q)_{r-|J|+1}} \tbinom{r+1 - |J|}{\underline k - \underline e_J}_q
\\&
=\sum_{\varnothing\neq J\subseteq [n]}  \sum_{{\underline k  \in \PP^n }}  (-1)^{|J|-1} \tfrac{(q)_r}{(q)_{r-|J|+1}}  \tbinom{r+1 - |J|}{\underline k - \underline e_J}_q
\\&
=\sum_{i=0}^{n-1}\sum_{\substack{ J\subseteq [n] \\ |J|=i+1}}  \sum_{{\underline k  \in \PP^n }}  (-1)^{i} \tfrac{(q)_r}{(q)_{r-i}}  \tbinom{r-i}{\underline k - \underline e_J}_q.
 \ea
 \]
As $\underline k $ varies over all sequences in $\PP^n$ and $J$ varies over all $i+1$ element subsets of $ [n] $,
 the modified sequence $\underline k - \underline e_J$ varies over all sequences in $\PP^n$ as well as some elements of $\NN^n$.
 Let $\reduce(\underline k - \underline e_J)$ be the sequence in $\PP^{m}$ for $m\leq n$ formed by removing all zero entries.
 Each element of $\PP^n$ arises as $\underline k - \underline e_J = \reduce(\underline k - \underline e_J)$ for exactly $\binom{n}{i+1}$ different pairs $(\underline k, J)$.
 More generally, if $j\in[i+1]$ then each of element of $\PP^{n-j}$ arises as $\mathsf{reduce}(\underline k - \underline e_J)$ 
 for exactly $\binom{n}{j}\binom{n-j}{i+1-j}$ pairs $(\underline k, J)$. So, continuing from the displayed equation, we get 
 \[F_{r+1}^{(n)}
=\sum_{i=0}^{n-1} \sum_{j=0}^{i+1} \sum_{\underline k' \in \PP^{n-j}} \tbinom{n}{j}\tbinom{n-j}{i+1-j}   (-1)^{i} \tfrac{(q)_r}{(q)_{r-i}}   \tbinom{r-i}{\underline k'}_q.\]
Now substitute the identities $\binom{n}{j} = \binom{n}{n-j}$ and $F_{r-i}^{(n-j)}= \sum_{\underline k' \in \PP^{n-j}}   \tbinom{r-i}{\underline k'}_q$.
% then becomes $ F_{r+1}^{(n)}=\sum_{i=0}^{n-1}    (-1)^{i} \tfrac{(q)_r}{(q)_{r-i}} \Bigr(\sum_{j=n-1-i}^n  \tbinom{n}{j}\tbinom{j}{n-1-i}  F_{r-i}^{(j)}\Bigr)$.
\end{proof}

The result of applying the coproduct of $\mQSym$ to 
$\oX_G(q) $ cannot be expressed as a finite linear combination of
terms $\oX_{G_1}(q)\otimes \oX_{G_2}(q)$,
and $\oX_G(q)$ does not seem to naturally arise as the image in $\mQSym$ of a combinatorial LC-Hopf algebra.
In general $\oX_G(q)$ is neither $\bare$-positive nor multifundamental-positive (see the examples in Table~\ref{oX-tab}).
However, $\oX_G(q)$ does have a nontrivial positivity property
that is not shared by $X_G(\x,q,\mu)$.

A \defn{set-valued tableau} $T$ of shape $\lambda$ is an assignment of sets $T_{ij} \in \Set(\PP)$
to the cells $(i,j)$ in  
$\D_\lambda = \{ (i,j) \in \PP\times \PP : 1 \leq j \leq \lambda_i\}$
of a partition $\lambda$. We write $(i,j)\in T$ to indicate that $(i,j)$ belongs to the 
shape of $T$.  
A set-valued tableau $T$ is \defn{semistandard} if $T_{ij} \preceq T_{i,j+1}$ 
and $T_{ij} \prec T_{i+1,j}$
for all relevant positions.
Let $x^T: = \prod_{(i,j) \in T} \prod_{k \in T_{ij}} x_k$ and $|T| := \sum_{(i,j) \in T} |T_{ij}|$.

\begin{definition}
The \defn{symmetric Grothendieck function} of a partition $\lambda$ is 
the power series
$ \bars_\lambda := \sum_{ T\in \SetSSYT(\lambda)} (-1)^{|T|-|\lambda|} x^T \in \ZZ\llbracket x_1,x_2,\dots \rrbracket$
where $\SetSSYT(\lambda)$ is the set of all semistandard set-valued tableaux of shape $\lambda$.
\end{definition}

 %The \defn{Young diagram} of a partition $\lambda$ is  $\D_\lambda = \{ (i,j) \in \PP\times \PP : 1 \leq j \leq \lambda_i\}$.
Each $\bars_\lambda $ is in $ \mSym$  and the set of all 
symmetric Grothendieck functions is another pseudobasis for $\mSym$ \cite{Buch2002}. 
When restricted to $n$ variables, these symmetric functions are  
the $K$-theory classes for the Schubert cells in the type $A$ Grassmannian; see the discussion in \cite[\S2.2]{CPS}.
 
%
%
%\begin{corollary}
%Assume $q\neq 0$. If $G$ is an ordered graph with $\qX_G(q) \in \mSym$ then 
%$\qX_G(q)$ is $\bars$-positive and $\bare$-positive.
%\end{corollary}
%
%\begin{proof}
%In this (rare) event, $G$ is a cluster graph by Theorem~\ref{cluster-thm}
%so is a product of $\bars$- and $\bare$-positive expressions . . . .TODO ($\bars$-case)
%\end{proof}

We write $\mu \subseteq \lambda$ for two partitions if $\D_\mu \subseteq \D_\lambda$ and set $\D_{\lambda/\mu} := \D_\lambda \setminus \D_\mu$.
In the definition below, a \defn{semistandard tableau} of shape $\lambda/\mu$ means a filling of $\D_{\lambda/\mu}$
by positive integers such that each row is weakly increasing (as column indices increase)
and each column is strict increasing (as row indices increase).
 
\begin{definition}[{\cite[Def.~3.8]{CPS}}] Suppose $P$ is a finite poset and $\lambda$ is a partition.
A \defn{Grothendieck $P$-tableau} of shape $\lambda$ is a pair   $T=(U,V)$ 
with both of these properties: 
\ben
\item[(a)] $U$ is a filling of $\D_\mu$ by elements of $P$  for some partition $\mu\subseteq\lambda$,
such that each element of $P$ is in at least one cell,
and  for each $(i,j) \in \D_\mu$ 
one has
$U_{ij} <_P U_{i,j+1}\text{ if }(i,j+1) \in \D_\mu $ and 
$U_{ij} \not>_P U_{i+1,j}\text{ if }(i+1,j) \in \D_\mu$;
%Here $U_{ij} \in P$ denotes the entry of $U$ in cell $(i,j)$.
\item[(b)] $V$ is a semistandard tableau of shape $\lambda/\mu$, whose entries in each row $i$ are all less than $i$ 
(so $\D_{\lambda/\mu}$ must have no cells in the first row).
\een
Let $\oT_P$ be the set of Grothendieck $P$-tableaux.
Let $\lambda(T)$ be the shape of $T \in \oT_P$.
\end{definition}

\begin{example}\label{pGroth-ex}
When the elements of $P$ do not include any positive integers,
we can draw a Grothendieck $P$-tableau $T=(U,V)$ of shape $\lambda$ as simply a filling of $\D_\lambda$ in English notation,
as there is no ambiguity about which boxes belong to $U$ and which belong to $V$.
Suppose $P$ is the poset on $\{a,b,c,d,e\}$ with $a<c<d<e$ and $b<c<d<e$, while $a$ and $b$ are incomparable.
Then 
\[
\ytableausetup{boxsize=0.4cm,aligntableaux=center}
\begin{ytableau}
a & c & d & e \\
b & d & e & 1 \\
c  & 1 & 2   \\
3
\end{ytableau}
\quand
 \begin{ytableau}
a & c & d & e \\
b  & c & d & e \\
a  & e & 2   \\
1 
\end{ytableau}
\]
are both Grothendieck $P$-tableaux of shape $\lambda=(4,4,3,1)$.
\end{example}

%A poset $P$  if \defn{$(3+1)$-free} if it does not contain a 3-element chain $a<_Pb<_Pc$ whose elements are all 
%incomparable to some fourth element $d$.

Suppose $P$ is a finite poset on a subset of $\ZZ$,
and let $G = \inc(P)$.
Choose some $T = (U,V) \in \oT_P$ and let $\mu$ be the shape of $U$.
A \defn{$G$-inversion} of $T$
is a pair of cells $(i,j),(k,l) \in \D_\mu$ 
containing incomparable elements of $P$ such that
 $i<k$ and 
$ U_{ij} > U_{kl}$. 
Let $\inv_G(T)$ be the number of all $G$-inversions of $T$.

\begin{example}
Define $P$ as in Example~\ref{pGroth-ex} and set $a=-1$, $b=-2$, $c=-3$, $d=-4$, and $e=-5$.
Then the two Grothendieck $P$-tableaux in Example~\ref{pGroth-ex}
both have only a single $G$-inversion, given by the pair of cells $(1,1)$ and $ (2,1)$.
\end{example}

One of the main results of \cite{CPS} 
establishes that if $G$ is the incomparability graph $\inc(P)$ of a $(3+1)$-free poset $P$ then 
$\bX_G =\sum_{T \in \oT_P} \bars_{\lambda(T)}$.
This is a $K$-theoretic generalization of a theorem of Gasharov \cite{gash}
showing that $X_G$ is Schur positive for the same family of \defn{claw-free incomparability graphs}.

Shareshian and Wachs \cite{SW2016} discovered a refinement of Gasharov's result,
which has the following $K$-theoretic extension.
Both our theorem and its predecessor in  \cite{SW2016}  
only apply to incomparability graphs of natural unit interval orders,
which make up a proper subset of all claw-free incomparability graphs.
This loss of generality is to be expected, since we need $P$ to be a natural unit interval order to deduce
from Lemma~\ref{des-nuo-lem} that $\oX_{\inc(P)}(q)$ is symmetric.

\begin{theorem}  If $G=\inc(P)$ for a natural unit interval order $P$ then 
\[\textstyle\oX_G(q) =\sum_{T \in \oT_P} q^{\inv_G(T)} \bars_{\lambda(T)}.\]
\end{theorem}

Table~\ref{bars-tab} shows some examples of these positive $\bars$-expansions.

\begin{proof}
Our argument is a hybrid of the proofs of \cite[Thm.~3.9]{CPS} and \cite[Thm.~6.3]{SW2016}.
To remove any ambiguity in the following definitions, we assume that the elements of $P$ are all
negative integers,
so that as a set $P$ is disjoint from $\PP=\{1,2,3,\dots\}$.
Fix a partition $\lambda$ and let $n = |\lambda|$ and $N \geq 2n$.

Following \cite{CPS}, define a \defn{Grothendieck $P$-array of type $\lambda$}
to be a pair $(\pi, A)$ where $ \pi \in S_N$ and $A$ is a map
$\PP \times \PP \to P \sqcup \{\emptyset\}\sqcup \PP$, written $(i,j) \mapsto A_{ij}$,
with all of the following properties:
\bei
\item $A_{ij}=\emptyset$ unless $i \leq \ell(\lambda)$ and $j \leq \lambda_{\pi(i)} - \pi(i)+i$;
\item for each $p \in P$ there is some $(i,j) \in \PP \times \PP$ with $A_{ij} = p$;
\item if $A_{ij} \in P$ when $j>1$, then $A_{i(j-1)} \in P$ and $A_{i(j-1)} <_P A_{ij}$;
\item if $A_{ij} \in \PP$, then $A_{ij} \leq \pi(i) - 1$; and 
\item if $A_{ij} \in \PP$ when $j>1$, then
$A_{i(j-1)} \in P$, or $A_{i(j-1)} \in \PP$ and $A_{i(j-1)} \leq A_{ij}$.
\eei
One can think of $A$ as a partial filling of an infinite matrix by elements of the poset $P$
(viewing $A_{ij} = \emptyset$ as an unfilled position),
with left-justified rows
that each consist of an increasing chain in $P$ followed by a weakly increasing
sequence of positive integers.

If $(\pi, A)$ is a Grothendieck $P$-array of type $\lambda$,
then define $\inv_G(A)$ to be the set of pairs of positions
$(i,j),(k,l) \in \PP\times \PP$ with $i<k$,
such that $A_{ij}$ and $A_{kl} $
are incomparable elements of $P$ with
$ A_{ij} > A_{kl}$. 

Write $\oX_G(q) =\sum_\lambda c_\lambda(q) \bars_{\lambda}$
where $c_\lambda(q) \in \ZZ\llbracket q\rrbracket$.
Then, 
repeating the part of the  proof of \cite[Thm.~3.9]{CPS}
through \cite[Eq.\,(3.4)]{CPS},
using the formula $\oX_G(q)= \sum_{\kappa} q^{\des_{G}(\kappa)} x^\kappa$
from Lemma~\ref{des-nuo-lem} in place of $\bX_G$,
 shows that  
\be\label{clam-eq}
c_\lambda(q) = \sum_{(\pi, A)} \sgn(\pi) q^{\inv_G(A)}
\ee
where the sum ranges over all Grothendieck $P$-arrays of type $\lambda$.

As in \cite{CPS}, we refer to a position $(i,j)$ with $i\geq 2$
as a \defn{flaw} of a Grothendieck $P$-array $(\pi,A)$ if one of the following occurs:
\begin{itemize}
\item $A_{ij} \in P$ and either $A_{(i-1)j }\notin P$ or $A_{ij} <_P A_{(i-1)j}$; or 

\item $A_{ij} \in\PP$ and either $A_{(i-1)j}=\emptyset$ or $A_{(i-1)j} \in \PP$ with $A_{ij} \leq A_{(i-1)j}$.
\end{itemize}
One can check that a Grothendieck $P$-array $(\pi,A)$ has no flaws if and only if $\pi=1$ is the identify permutation and $A\in \oT_P$ corresponds to a  Grothendieck $P$-tableau.

The goal now is to produce a sign-reversing, inversion-preserving involution $\Psi$ of the set of all Grothendieck $P$-arrays $(\pi,A)$ that have flaws. The existence of such a map will imply that \eqref{clam-eq} reduces to $c_\lambda(q) = \sum_{T\in \oT_P,\lambda(T)=\lambda}  q^{\inv_G(A)}
$. The involution used in the proof of \cite[Thm.~3.9]{CPS} is sign-reversing but not
inversion-preserving. We can correct this by incorporating ideas in the proof of \cite[Thm.~6.3]{SW2016}.

Given a flawed Grothendieck $P$-array $(\pi,A)$,
let $c$ be the minimal column in which a flaw occurs, and let $r$ be the maximal row such that $(r,c)$ is a flaw. Let $C_{r-1}(A)$ be the set of elements of $P$ that appear in row $r-1$ weakly to the right of column $c$. Let $C_{r}(A)$ be the set of elements of $P$ that appear in row $r$ strictly to the right of column $c$.
Then let $H(A)$ be the subgraph of $G$ induced on $C_{r-1}(A) \cup C_{r}(A)$. 
For $i \in \{r,r+1\}$ let $O_i(A)$ be the set of $x \in C_i(A)$ that belong to a connected component of odd size in $H(A)$. 
All elements of $P$ that belong to both $C_{r-1}(A)$ and $C_r(A)$ 
are comparable in $P$ to all vertices in $H(A)$, so must belong to connected components of size one, and are therefore in both $O_{r-1}(A)$ and $O_r(A)$.

Set $E_i(A) := C_i(A) \setminus O_i(A)$. Then let $I_{r-1}(A)$ be the set of elements of $P$ that appear in row $r-1$ strictly to the left of column $c$, and let $I_r(A)$ be the set of elements of $P$ in row $r$ weakly to the left of column $c$. 
By repeating the argument in the proof of \cite[Thm.~6.3]{SW2016}, 
which uses the hypothesis that $P$ is $(3+1)$-free as a natural unit interval order,
one checks that $I_{r-1}(A) \cup E_{r-1}(A) \cup O_{r}(A)$ is a chain in $P$ in which $I_{r-1}(A)$ is an initial chain, and that 
$I_{r}(A) \cup E_{r}(A) \cup O_{r-1}(A)$ is a chain in $P$ in which $I_{r}(A)$ is an initial chain. Therefore we can define $\Psi(\pi,A)=(\pi',A')$ where $\pi' $ is the product of $\pi$ and the transposition $(r-1\, r)$, where $A'$ is formed from $A$ by changing row $r-1$ and $r$ as follows:
\begin{itemize}
\item replace row $r-1$ by the chain $I_{r-1}(A) \cup E_{r-1}(A) \cup O_r(A)$ followed by any integers strictly to the right of column $c$ in row $r$ of $A$;

\item replace row $r$ by the chain $I_{r}(A) \cup E_{r}(A) \cup O_{r-1}(A)$ followed by any integers weakly to the right of column $c$ in row $r-1$ of $A$.
\end{itemize}
Then one can check that $(\pi',A')$ is a Grothendieck $P$-array also with a flaw at $(r,c)$,
and that $\Psi(\pi',A') = (\pi,A)$. The details are similar to what is presented in the proof of \cite[Thm.~3.9]{CPS}.

Finally, we explain why $\inv_G(A) = \inv_G(A')$. 
This relies on the fact that the connected components in $H(A)$ are paths with numerically increasing vertices \cite[Lem.~4.4]{SW2016}.
In each such component of even size exactly half the elements must appear in row $r-1$ of $A$ and exactly half must appear in row $r$ of $A$. 
In the components of odd size, every other element must appear in row $r-1$ of $A$ with the elements in between appearing in row $r$ of $A$. 
Thus in any given odd sized component with $2k+1>1$ vertices, 
either vertices $1$, $3$, $5$, \dots, $2k-1$ contribute inversions to $\inv_G(A)$
while vertices $2$, $4$, $6$, \dots, $2k$ contribute inversions to $\inv_G(A')$, or vice versa. Either way, we have $\inv_G(A) = \inv_G(A')$.
\end{proof}

 \begin{table}[htb]
\centerline{
\scalebox{1}{
\begin{tabular}{|c|c|l|}
\hline & & 
\\[-8pt]
Order $P$   & $G=\inc(P)$  & $\bX_G(q)$ in the Grothendieck basis $\{\bars_\lambda\}$ 
\\[-8pt] & & 
\\ \hline & & 
\\[-8pt]  
  \begin{tikzpicture}[scale=1,baseline=(void.base)]
  \node at (0, 0)(1){1};
    \node at (0.75, 0.433)(2){2};
    \node at (0, 0.866)(3){3};
    \node at (0.5,0.433)(void) {};
    \draw[black, thick,->] (1) to (3);
  \end{tikzpicture} 
  & 
    \begin{tikzpicture}[scale=1,baseline=(void.base)]%3-path
  \node at (0, 0)(1){1};
    \node at (1, 0)(2){2};
    \node at (0.5, 0.866)(3){3};
    \node at (0.5,0.433)(void) {};
    \draw[black, thick] (1) -- (2);
    \draw[black, thick] (2) -- (3);
  \end{tikzpicture} 
  & \pbox[c]{12cm}{\small $(1 + 2q + q^2) \bars_{111} + q \bars_{21} + (6 + 12q + 10q^2 + 4q^3 + q^4) \bars_{1111} + (4q + 3q^2) \bars_{211} + q \bars_{22} + (12q + 15q^2 + 6q^3 + 2q^4) \bars_{2111} + (4q + 4q^2) \bars_{221} + (12q + 20q^2 + 8q^3 + 3q^4) \bars_{2211} + (4q + 5q^2) \bars_{222} + (12q + 25q^2 + 11q^3 + 4q^4) \bars_{2221} + (12q + 29q^2 + 15q^3 + 5q^4) \bars_{2222}+ \dots $
}  %
 \\[-8pt] & & 
 \\ \hline & & 
\\[-8pt]  
    \begin{tikzpicture}[scale=1,baseline=(void.base)]
  \node at (0, 0.433)(1){1};
    \node at (0.5, 0.433)(2){2};
    \node at (1, 0.433)(3){3};
    \node at (0.5,.433)(void) {};
  \end{tikzpicture} & 
    \begin{tikzpicture}[scale=1,baseline=(void.base)]%3-cycle
  \node at (0, 0)(1){1};
    \node at (1, 0)(2){2};
    \node at (0.5, 0.866)(3){3};
    \node at (0.5,.433)(void) {};
    \draw[black, thick] (3) -- (1);
    \draw[black, thick] (2) -- (1);
    \draw[black, thick] (2) -- (3);
  \end{tikzpicture} & 
  \pbox[c]{12cm}{\small$(1 + 2q + 2q^2 + q^3) \bars_{111} + (6 + 12q + 15q^2 + 12q^3 + 6q^4 + 3q^5) \bars_{1111}
 + \dots $}
   \\[-8pt] & & 
 \\ \hline & & 
\\[-8pt]  
     \begin{tikzpicture}[scale=1,baseline=(void.base)]%
  \node at (0, 0)(1){1};
    \node at (1, 0)(2){2};
    \node at (1, 1)(4){4};
     \node at (0, 1)(3){3};
         \node at (0.5,.5)(void) {};
    \draw[black, thick,->] (1) -- (3);
    \draw[black, thick,->] (1) -- (4);
    \draw[black, thick,->] (2) -- (4); 
  \end{tikzpicture} & 
     \begin{tikzpicture}[scale=1,baseline=(void.base)]%4-path
  \node at (0, 0)(1){1};
    \node at (1, 0)(2){2};
    \node at (1, 1)(3){3};
     \node at (0, 1)(4){4};
         \node at (0.5,.5)(void) {};
    \draw[black, thick] (3) -- (2);
    \draw[black, thick] (2) -- (1);
    \draw[black, thick] (4) -- (3);
  \end{tikzpicture}  & 
  \pbox[c]{12cm}{\small$(1 + 3q + 3q^2 + q^3) \bars_{1111} + (2q + 2q^2) \bars_{211} + (q + q^2) \bars_{22} + (8 + 24q + 30q^2 + 20q^3 + 8q^4 + 2q^5) \bars_{11111} + (12q + 18q^2 + 10q^3 + 2q^4) \bars_{2111} + (6q + 10q^2 + 4q^3) \bars_{221} + (48q + 90q^2 + 82q^3 + 46q^4 + 16q^5 + 6q^6) \bars_{21111} + (24q + 52q^2 + 42q^3 + 17q^4 + 3q^5) \bars_{2211} + (10q + 21q^2 + 12q^3 + 3q^4) \bars_{222} + (80q + 204q^2 + 228q^3 + 166q^4 + 76q^5 + 32q^6 + 8q^7) \bars_{22111} + (40q + 105q^2 + 102q^3 + 59q^4 + 18q^5 + 4q^6) \bars_{2221} + (128q + 385q^2 + 492q^3 + 426q^4 + 236q^5 + 123q^6 + 40q^7 + 12q^8) \bars_{22211} + (56q + 163q^2 + 183q^3 + 131q^4 + 56q^5 + 21q^6 + 3q^7 + q^8) \bars_{2222} + (192q + 635q^2 + 908q^3 + 898q^4 + 580q^5 + 363q^6 + 144q^7 + 64q^8 + 12q^9 + 4q^{10}) \bars_{22221} + (256q + 891q^2 + 1381q^3 + 1514q^4 + 1099q^5 + 797q^6 + 379q^7 + 205q^8 + 65q^9 + 28q^{10} + 4q^{11} + 3q^{12}) \bars_{22222} + \dots $}
    \\[-8pt] & & 
 \\ \hline & & 
\\[-8pt]  
     \begin{tikzpicture}[scale=1,baseline=(void.base)]%4-path
  \node at (0, 0)(1){1};
    \node at (1, 0)(2){2};
    \node at (1, 1)(4){4};
     \node at (0, 1)(3){3};
         \node at (0.5,.5)(void) {};
    \draw[black, thick,->] (1) -- (3);
    \draw[black, thick,->] (1) -- (4);
  \end{tikzpicture} & 
     \begin{tikzpicture}[scale=1,baseline=(void.base)]%4-path with \
  \node at (0, 0)(1){1};
    \node at (1, 0)(2){2};
    \node at (1, 1)(3){3};
     \node at (0, 1)(4){4};
         \node at (0.5,.5)(void) {};
    \draw[black, thick] (3) -- (2);
    \draw[black, thick] (2) -- (1);
    \draw[black, thick] (4) -- (3);
    \draw[black, thick] (2) -- (4);
  \end{tikzpicture}  & 
  \pbox[c]{12cm}{\small$(1 + 3q + 4q^2 + 3q^3 + q^4) \bars_{1111} + (q + 2q^2 + q^3) \bars_{211} + (8 + 24q + 37q^2 + 37q^3 + 25q^4 + 13q^5 + 5q^6 + q^7) \bars_{11111} + (6q + 14q^2 + 14q^3 + 7q^4 + 3q^5) \bars_{2111} + (q + 3q^2 + 2q^3) \bars_{221} + (24q + 62q^2 + 82q^3 + 69q^4 + 45q^5 + 23q^6 + 10q^7 + 3q^8) \bars_{21111} + (6q + 20q^2 + 23q^3 + 13q^4 + 6q^5) \bars_{2211} + (q + 4q^2 + 3q^3) \bars_{222} + (24q + 86q^2 + 126q^3 + 116q^4 + 80q^5 + 42q^6 + 20q^7 + 6q^8) \bars_{22111} + (6q + 26q^2 + 34q^3 + 21q^4 + 11q^5) \bars_{2221} + (24q + 110q^2 + 183q^3 + 181q^4 + 134q^5 + 72q^6 + 36q^7 + 12q^8) \bars_{22211} + (6q + 31q^2 + 46q^3 + 31q^4 + 18q^5) \bars_{2222} + (24q + 133q^2 + 252q^3 + 267q^4 + 210q^5 + 117q^6 + 61q^7 + 22q^8) \bars_{22221} + (24q + 151q^2 + 318q^3 + 363q^4 + 300q^5 + 181q^6 + 98q^7 + 37q^8) \bars_{22222} + \dots $}
%  \\ [-4pt] && \\
%  %
%     \begin{tikzpicture}[scale=1,baseline=(void.base)]%4-path
%  \node at (0, 0)(1){1};
%    \node at (1, 0)(2){2};
%    \node at (1, 1)(4){4};
%     \node at (0, 1)(3){3};
%         \node at (0.5,.5)(void) {};
%    \draw[black, thick,->] (1) -- (4);
%    \draw[black, thick,->] (2) -- (4);
%  \end{tikzpicture} & 
%  %
%     \begin{tikzpicture}[scale=1,baseline=(void.base)]%4-path with /
%  \node at (0, 0)(1){1};
%    \node at (1, 0)(2){2};
%    \node at (1, 1)(3){3};
%     \node at (0, 1)(4){4};
%         \node at (0.5,.5)(void) {};
%    \draw[black, thick] (3) -- (2);
%    \draw[black, thick] (2) -- (1);
%    \draw[black, thick] (4) -- (3);
%    \draw[black, thick] (1) -- (3);
%  \end{tikzpicture}  & 
%    % 
%  \pbox[c]{12cm}{\small$(1 + 3q + 4q^2 + 3q^3 + q^4) \bars_{1111} + (q + 2q^2 + q^3) \bars_{211} + (8 + 24q + 37q^2 + 37q^3 + 25q^4 + 13q^5 + 5q^6 + q^7) \bars_{11111} + (6q + 14q^2 + 14q^3 + 7q^4 + 3q^5) \bars_{2111} + (q + 3q^2 + 2q^3) \bars_{221} + (24q + 62q^2 + 82q^3 + 69q^4 + 45q^5 + 23q^6 + 10q^7 + 3q^8) \bars_{21111} + (6q + 20q^2 + 23q^3 + 13q^4 + 6q^5) \bars_{2211} + (q + 4q^2 + 3q^3) \bars_{222} + (24q + 86q^2 + 126q^3 + 116q^4 + 80q^5 + 42q^6 + 20q^7 + 6q^8) \bars_{22111} + (6q + 26q^2 + 34q^3 + 21q^4 + 11q^5) \bars_{2221} + (24q + 110q^2 + 183q^3 + 181q^4 + 134q^5 + 72q^6 + 36q^7 + 12q^8) \bars_{22211} + (6q + 31q^2 + 46q^3 + 31q^4 + 18q^5) \bars_{2222} + (24q + 133q^2 + 252q^3 + 267q^4 + 210q^5 + 117q^6 + 61q^7 + 22q^8) \bars_{22221} + (24q + 151q^2 + 318q^3 + 363q^4 + 300q^5 + 181q^6 + 98q^7 + 37q^8) \bars_{22222} + \dots $}
%  %
    \\[-8pt] & & 
 \\ \hline & & 
\\[-8pt]  
     \begin{tikzpicture}[scale=1,baseline=(void.base)]
  \node at (0, 0)(1){1};
    \node at (0.5, 0.5)(2){2};
    \node at (1, 0.5)(3){3};
     \node at (0, 1)(4){4};
         \node at (0.5,.5)(void) {};
    \draw[black, thick,->] (1) -- (4);
  \end{tikzpicture} & 
     \begin{tikzpicture}[scale=1,baseline=(void.base)]%4-path with X
  \node at (0, 0)(1){1};
    \node at (1, 0)(2){2};
    \node at (1, 1)(3){3};
     \node at (0, 1)(4){4};
         \node at (0.5,.5)(void) {};
    \draw[black, thick] (2) -- (3);
    \draw[black, thick] (1) -- (2);
    \draw[black, thick] (4) -- (3);
    \draw[black, thick] (1) -- (3);
    \draw[black, thick] (2) -- (4);
  \end{tikzpicture}  & 
  \pbox[c]{12cm}{\small$(1 + 3q + 5q^2 + 5q^3 + 3q^4 + q^5) \bars_{1111} + (q^2 + q^3) \bars_{211} + (8 + 24q + 44q^2 + 54q^3 + 48q^4 + 34q^5 + 18q^6 + 8q^7 + 2q^8) \bars_{11111} + (5q^2 + 9q^3 + 6q^4 + 4q^5) \bars_{2111} + (q^2 + q^3) \bars_{221} + (17q^2 + 41q^3 + 46q^4 + 42q^5 + 23q^6 + 14q^7 + 5q^8) \bars_{21111} + (5q^2 + 9q^3 + 7q^4 + 5q^5) \bars_{2211} + (q^2 + q^3) \bars_{222} + (17q^2 + 41q^3 + 52q^4 + 50q^5 + 29q^6 + 18q^7 + 7q^8) \bars_{22111} + (5q^2 + 9q^3 + 8q^4 + 6q^5) \bars_{2221} + (17q^2 + 41q^3 + 58q^4 + 58q^5 + 36q^6 + 23q^7 + 9q^8) \bars_{22211} + (5q^2 + 9q^3 + 9q^4 + 7q^5) \bars_{2222} + (17q^2 + 41q^3 + 64q^4 + 66q^5 + 44q^6 + 29q^7 + 11q^8) \bars_{22221} + (17q^2 + 41q^3 + 69q^4 + 73q^5 + 53q^6 + 36q^7 + 13q^8) \bars_{22222} + \dots $}
      \\[-8pt] & & 
 \\ \hline & & 
\\[-8pt]  
     \begin{tikzpicture}[scale=1,baseline=(void.base)]%empty
  \node at (0, 0.5)(1){1};
    \node at (0.4, 0.5)(2){2};
    \node at (0.8, 0.5)(3){3};
     \node at (1.2, 0.5)(4){4};
         \node at (0.5,.5)(void) {};
  \end{tikzpicture} & 
     \begin{tikzpicture}[scale=1,baseline=(void.base)]%complete graph
  \node at (0, 0)(1){1};
    \node at (1, 0)(2){2};
    \node at (1, 1)(3){3};
     \node at (0, 1)(4){4};
         \node at (0.5,.5)(void) {};
    \draw[black, thick] (3) -- (2);
    \draw[black, thick] (2) -- (1);
    \draw[black, thick] (4) -- (3);
    \draw[black, thick] (1) -- (3);
    \draw[black, thick] (2) -- (4);
    \draw[black, thick] (1) -- (4);
  \end{tikzpicture}  & 
  \pbox[c]{12cm}{\small$(1 + 3q + 5q^2 + 6q^3 + 5q^4 + 3q^5 + q^6) \bars_{1111} + (8 + 24q + 44q^2 + 60q^3 + 64q^4 + 56q^5 + 40q^6 + 24q^7 + 12q^8 + 4q^9) \bars_{11111} + \dots $}
   \\[-8pt] & & 
  \\ \hline
\end{tabular}}}
\caption{Partial Grothendieck expansions of $\oX_G(q)$ for incomparability graphs $G=\inc(P)$
of some natural unit interval orders. The ellipses ``$\dots$'' mask higher order terms  indexed by partitions $\lambda$ with  $\ell(\lambda)>|V(G)|+1$.}\label{bars-tab}
\end{table}

\printbibliography

\end{document}

%% file: preamble.tex
\usepackage{latexsym}
\usepackage{amssymb}
\usepackage{amsthm}
\usepackage{amscd}
\usepackage{amsmath}
\usepackage{mathrsfs}
\usepackage{graphicx}
\usepackage{hyperref}
\usepackage{shuffle}
\usepackage{stmaryrd}
\usepackage{ytableau}

\usepackage[all]{xy}
\input xy \xyoption{frame}
\xyoption{dvips}

\usepackage[colorinlistoftodos]{todonotes}
\usetikzlibrary{chains,scopes,decorations.markings}

\usepackage{tikz-cd}

\theoremstyle{definition}
\newtheorem* {theorem*}{Theorem}
\newtheorem* {conjecture*}{Conjecture}
\newtheorem{theorem}{Theorem}[section]

\theoremstyle{definition}

\newtheorem* {example*}{Example}

\newtheorem{lemma}[theorem]{Lemma}
\theoremstyle{definition}
\newtheorem{definition}[theorem]{Definition}
\theoremstyle{definition}

\newtheorem{proposition}[theorem]{Proposition}
\newtheorem{corollary}[theorem]{Corollary}

\newtheorem{remark}[theorem]{Remark}
\theoremstyle{definition}
\newtheorem {example}[theorem]{Example}
\theoremstyle{definition}

\theoremstyle{definition}

\theoremstyle{definition}

\theoremstyle{definition}

\def\({\left(}
\def\){\right)}

\newcommand{\FF}{\mathbb{F}}

\newcommand{\QQ}{\mathbb{Q}}
\newcommand{\cP}{\mathcal{P}}

\def\NN{\mathbb{N}}

\def\Hom{\mathrm{Hom}}
\def\End{\mathrm{End}}

\def\ZZ{\mathbb{Z}}

\def\spanning{\textnormal{-span}}

\newcommand{\cM}{\mathcal{M}}

\newcommand{\sgn}{\mathrm{sgn}}

\def\fk{\mathfrak}
\def\fkm{\mathfrak{m}}

\def\barr{\begin{array}}
\def\earr{\end{array}}
\def\ba{\begin{aligned}}
\def\ea{\end{aligned}}
\def\be{\begin{equation}}
\def\ee{\end{equation}}

\def\quand{\quad\text{and}\quad}

\def\inv{\operatorname{inv}}
\def\Inv{\operatorname{Inv}}

\def\cH{\mathcal H}
\def\cM{\mathcal M}

\def\id{\mathrm{id}}
\def\PP{\mathbb{P}}

\def\ben{\begin{enumerate}}
\def\een{\end{enumerate}}

\def\bei{\begin{itemize}}
\def\eei{\end{itemize}}

\def\D{\mathsf{D}}

\def\Des{\operatorname{Des}}

\def\x{\textbf{x}}

\newcommand{\xRightarrow}[2][]{\ext@arrow 0359\Rightarrowfill@{#1}{#2}}

\renewcommand{\r}[1]{\textcolor{red}{#1}}

\def\arcstart{\ \xy<0cm,-.06cm>\xymatrix@R=.1cm@C=.10cm }
\newcommand{\arcstartc}[1]{\ \xy<0cm,-.15cm>\xymatrix@R=.1cm@C=#1cm}

\def\r{\mathbf{r}}

\def\cG{\mathcal{G}}

\def\sort{\operatorname{sort}}

\def\bar{\overline}

\definecolor{darkred}{rgb}{0.7,0,0} % darkred color
\newcommand{\defn}[1]{{\color{darkred}\emph{#1}}} % emphasis of a definition

\def\AO{\mathsf{AO}}
\def\MAO{\fkm\mathsf{AO}}
\newcommand{\Sym}{\mathsf{Sym}}
\newcommand{\QSym}{\mathsf{QSym}}

\def\wzeta{\bar\zeta}
\def\zetaq{\zeta_{\mathsf{Q}}}
\def\wzetaq{\wzeta_{\mathsf{Q}}}

\def\op{\mathsf{op}}
\def\asc{\operatorname{asc}}
\def\des{\operatorname{des}}

\def\oX{\overset{\to}{X}}

\def\bDelta{\Delta^{(\beta)}}

\def\mQSym{\fkm\QSym}
\def\mSym{\fkm\Sym}

\newcommand{\cO}{\mathcal{D}}
\def\wO{\bar\cO}
\def\wG{\bar\cG}
\def\wP{\bar\cP}
\def\htimes{\mathbin{\bar\otimes}}

\def\bX{\bar X}
\def\Set{\mathsf{Set}}

\def\bare{\bar e}
\def\bGamma{\bar \Gamma}
\def\mL{\cM}
\def\bL{\bar L}
\def\antipode{\mathsf{S}}

\def\wGwt{\mathfrak{m}\mathsf{WGraphs}}
\def\wzetaw{\wzeta_{\mathsf{WGraphs}}}

\def\kam{\overline{\widetilde m}}

\def\align{\mathsf{align}}

\def\SetSSYT{\mathsf{SetSSYT}}
\def\source{\mathsf{source}}

\def\oT{\mathscr{G}}
\def\inc{\mathsf{inc}}

\def\bS{S^{\fk m}}
\def\reduce{\mathsf{reduce}}
\def\cCl{\mathsf{Clan}}
\def\Cl{\cCl}

\def\oCl{\cCl^{\mathsf{dag}}}

\def\cO{\mathsf{DAGs}}
\def\cG{\mathsf{Graphs}}
\def\cP{\mathsf{LPosets}}

\def\mO{\mathfrak{m}\cO}
\def\mG{\mathfrak{m}\cG}
\def\mP{\mathfrak{m}\cP}
\def\bDelta{\blacktriangle}
\def\bars{{\overline s}}

\def\oG{\mathsf{O}\cG}

\def\qX{\overline{L}}
\def\oX{\overline{X}}

\usepackage{pbox}
\tikzset{every loop/.style={min distance=10 mm, in=60, out=120, looseness=10}}
\tikzset{
dot/.style = {circle, fill, inner sep=2.2,outer sep=0},
}

\numberwithin{equation}{section}
\allowdisplaybreaks[1]
\UseCrayolaColors

\def\KK{\mathbb{K}}

\def\op{\mathrm{decr}}
\def\ozeta{\zeta_{\oG}}